\documentclass[10pt,a4paper]{amsart}

\usepackage{dhucs}

\usepackage{amsmath,amsbsy}
\usepackage{amssymb}
\usepackage{amsfonts}
\usepackage{ mathrsfs }

\usepackage{stmaryrd}
\usepackage[all]{xy}
\usepackage{kotex}
\usepackage[pdftex]{color, graphicx}
\usepackage{mathdots}
\usepackage{tikz}
\usepackage{hyperref}
\usepackage{marginnote}
\usepackage{mathabx}

\newtheorem{proposition}{Proposition}[subsection]
\newtheorem{theorem}[proposition]{Theorem}
\newtheorem{lemma}[proposition]{Lemma}
\newtheorem{corollary}[proposition]{Corollary}
\theoremstyle{definition}
\newtheorem{definition}[proposition]{Definition}

\newtheorem*{claim}{Claim}
\theoremstyle{remark}
\newtheorem{remark}[proposition]{Remark}

\newcommand{\R}{\mathbb{R}}

\newcommand{\Z}{\mathbb{Z}}

\newcommand{\Hom}{\operatorname{Hom}}

\newcommand{\PSL}{\mathrm{PSL}}
\newcommand{\Tr}{\operatorname{Tr}}

\renewcommand{\phi}{\varphi}
\newcommand{\trace}{\operatorname{Tr}}
\newcommand{\Flag}{\mathrm{Flag}}

\newcommand{\Ad}{\operatorname{Ad}}

\newcommand{\Rep}{\mathcal{X}}
\newcommand{\Hit}{\operatorname{Hit}}

\renewcommand{\epsilon}{\varepsilon}
\renewcommand{\Tr}{\operatorname{Tr}}
\newcommand{\der}{\mathrm{d}}
\newcommand{\acts}{\cdot}
\renewcommand{\approx}{\cong}

\begin{document}

\title{Symplectic Coordinates on $\PSL_3(\R)$-Hitchin Components}
\author{Suhyoung Choi}
\address{Department of Mathematical Sciences, KAIST
	, Daejeon, Korea}
\email{schoi@math.kaist.ac.kr}
\thanks{The first and second authors were supported in part by NRF-2016R1D1A1B03932524.}

\author{Hongtaek Jung}
\address{Department of Mathematical Sciences, KAIST
	, Daejeon, Korea}
\email{htjung@kaist.ac.kr}

\author{Hong Chan Kim}
\address{Department of Mathematics Education,
	Korea University, Seoul, Korea}
\email{hongchan@korea.ac.kr}


\date{\today}



\begin{abstract}
Goldman parametrizes the $\PSL_3(\R)$-Hitchin component of a closed oriented hyperbolic surface of genus $g$  by  $16g-16$ parameters. Among them,  $10g-10$ coordinates are canonical. We prove that the $\PSL_3(\R)$-Hitchin component equipped with the Atiyah-Bott-Goldman symplectic form admits a global Darboux coordinate system such that the half of its coordinates are canonical Goldman coordinates.  To this end, we show a version of the action-angle principle and the Zocca-type decomposition formula for the symplectic form of H. Kim and Guruprasad-Huebschmann-Jeffrey-Weinstein given to  symplectic leaves of the Hitchin component.
\end{abstract}

\maketitle
\section{Introduction}
\subsection{History and motivation} Let $\Sigma$ be a closed oriented surface of genus at least 2. The Teichm\"uller space $\mathcal{T}(\Sigma)$ is the space of discrete faithful representations of $\pi_1(\Sigma)$ into the Lie group $\PSL_2(\R)$ modulo conjugation. It forms a connected component of the space of representations 
\[
\Rep(\pi_1(\Sigma), \PSL_2(\R))= \Rep_2(\pi_1(\Sigma)):=\Hom(\pi_1(\Sigma), \PSL_2(\R)) / \PSL_2(\R).
\] 
By replacing $2$ in $\PSL_2(\R)$ with a general natural number $n$, we can obtain the space $\Rep_n(\pi_1(\Sigma)):=\Rep(\pi_1(\Sigma), \PSL_n(\R))$. Observe that $\mathcal{T}(\Sigma)$ can be naturally embedded into $\mathcal{X}_n(\pi_1(\Sigma))$ as Fuchsian representations, that is, by definition,  representations of the form $\iota_n \circ \rho$ where $\rho\in \mathcal{T}(\Sigma)$ and $\iota_n:\PSL_2(\R) \to \PSL_n(\R)$ the unique irreducible representation of $\PSL_2(\R)$ into $\PSL_n(\R)$.  Then one may expect that a connected component containing a Fuchsian representation resembles the Teichm\"uller space. The first answer to this guess is given by Hitchin \cite{hitchin1992} in 1992. Indeed he shows that any component containing a Fuchsian representation is diffeomorphic to the $(n^2-1) (2g-2)$ dimensional cell. We call a connected component of $\Rep_n(\pi_1(\Sigma))$ containing a Fuchsian representation the $\PSL_n(\R)$-Hitchin component.  

Besides Hitchin's result, it is known the Hitchin component $\Hit_n(\Sigma)$ enjoys a lot of properties that the classical Teichm\"uller space has. Labourie \cite{labourie2006}, for instance, gives a dynamical characterization of $\Hit_n(\Sigma)$ and shows that each Hitchin representation is discrete  and faithful. 

There are many known global parametrizations of the Teichm\"uller space. Because $\Hit_n(\Sigma)$ is also a cell, we may expect the existence of global coordinate system for $\Hit_n(\Sigma)$. For the Hitchin component $\Hit_3(\Sigma)$, Goldman \cite{goldman1990} finds such a global coordinate system. However Goldman's argument cannot be applied to general $\Hit_n(\Sigma)$ cases because construction of Goldman coordinates essentially relies on the fact that $\Hit_3(\Sigma)$ represents the deformation space of convex projective structures on the surface $\Sigma$. See Choi-Goldman \cite{choi1993}. A uniform parametrization scheme  for general $\Hit_n(\Sigma)$ is obtained by  Bonahon-Dreyer \cite{bonahon2014}. Their method is based on Fock-Goncharov's theory \cite{fock2006} or, Thurston's construction of shearing coordinates. In $\Hit_3(\Sigma)$ Goldman coordinates and Bonahon-Dreyer's are related and the explicit coordinates transformation is given by Bonahon and I. Kim \cite{bonahon2018}.

It is well-known that the Teichm\"uller space $\mathcal{T}(\Sigma)$ carries the natural symplectic structure called the Weil-Petersson form $\omega_{WP}$. As a symplectic manifold, the Teichm\"uller space $(\mathcal{T}(\Sigma), \omega_{WP})$ has been studied by many mathematicians. One of the remarkable results is due to Wolpert \cite{wolpert1985} which states that the Fenchel-Nielsen coordinates are Darboux coordinates, namely,
\[
\omega_{WP}= \sum_{i=1} ^{3g-3} \der \ell_i \wedge \der\theta_i.
\]
$\Hit_n(\Sigma)$ also carries a symplectic form as the classical Teichm\"uller space does. Indeed it is work of Goldman \cite{goldman1984} that extends the Weil-Petersson symplectic form on $\mathcal{T}(\Sigma)$ into the Atiyah-Bott-Goldman symplectic form $\omega_G$ on $\Hit_n(\Sigma)$. Now, it is  natural to ask whether there is any global Darboux coordinate system with respect to $\omega_G$, analogues to the Fenchel-Nielsen coordinates.

For the Hitchin component $\Hit_3(\Sigma)$, H. Kim \cite{kim1999} claims that the Goldman coordinates \cite{goldman1990} are indeed Darboux coordinates for $\omega_G$. He first studies $\Hit_3(\Sigma)$ where $\Sigma$ is a compact surface with boundary. Although $\Hit_3(\Sigma)$ itself is not a symplectic manifold, it admits a foliation whose leaves are of the form $\Hit_3 ^{\mathscr{B}}(\Sigma)$,  a subspace of $\Hit_3(\Sigma)$ whose holonomies of boundary components are in prescribed conjugacy classes $\mathscr{B}$. H. Kim, as well as Guruprasad-Huebschmann-Jeffrey-Weinstein \cite{guruprasad1997}, show that each leaf $\Hit_3 ^{\mathscr{B}}(\Sigma)$ can be given a symplectic form $\omega_K ^{\Sigma}$.  When $\Sigma$ is a pants $P$, the space $\Hit_3 ^{\mathscr{B}}(P)$ can be parametrized by Goldman's coordinates $s$ and $t$. H. Kim shows that Goldman's $(s,t)$ parameters on $\Hit_3 ^{\mathscr{B}} (P)$ form  Darboux coordinates  with respect to $\omega_K ^P$. After then H. Kim tries to glue various $\Hit_3 ^{\mathscr{B}}(P)$ as Goldman does in \cite{goldman1990}. In the smooth category, this gluing process is relatively easy. In the symplectic category,  however,  it is more technical and his proof misses crucial intermediate steps. One goal of this paper is to fill the missing links and make the proof of H. Kim \cite{kim1999} more complete and clear. 

More recently, Sun-Zhang \cite{sz2017} and Sun-Wienhard-Zhang \cite{swz2017} construct Darboux coordinates for the $\PSL_n(\R)$-Hitchin components. Their tool is the deformation theory of Frenet curves.  

\subsection{Statement of results} Our first result is, roughly speaking, that the symplectic manifold $(\Hit_n ^{\mathscr{B}}(\Sigma), \omega_K ^\Sigma)$ can be decomposed into a product of simpler symplectic manifolds. 

Let $\Sigma$ be a compact oriented surface with negative Euler characteristic and possibly with boundary. By an \emph{essential simple closed curve}, we mean an embedded $S^1$ in $\Sigma$ that is not homotopic to a point nor a boundary component. Let $\xi$ be an essential simple closed curve. Given a path  $\eta$ from a base point $p$ to a point in $\xi$, we write $\xi^\eta$ to denote  the loop $\eta \ast \xi \ast \eta^{-1}$ at $p$, where $\ast$ is the concatenation. We sometimes regard $\xi$ as  an element of $\pi_1(\Sigma,p)$ up to conjugation by considering  $\xi^\eta$ for some implicitly chosen path $\eta$ from $p$ to a point in $\xi$. We abuse the notation $\langle \xi \rangle$ to denote the subgroup of $\pi_1(\Sigma, p)$ generated by $\xi^\eta$ when we do not care about particular choice of an element in its conjugacy class. 


Throughout this paper, $G$  denotes the Lie group $\PSL_n(\R)$ and $\mathfrak{g}$ its Lie algebra $\mathfrak{sl}_n(\R)$. If we have a representation $\rho : \pi_1(\Sigma) \to G$, $\mathfrak{g}$ becomes a $\pi_1(\Sigma)$-module $\mathfrak{g}_{\rho}$ via the action $\Ad_{\rho(\gamma)} (X)$, $\gamma\in \pi_1(\Sigma)$, $X\in \mathfrak{g}$. We sometimes write this action simply $\gamma\acts X$ if its meaning is clear from the context.  If there is no chance of confusion, we omit the subscript $\rho$ and simply write $\mathfrak{g}$ instead of $\mathfrak{g}_\rho$. 

We denote by $\Rep(\pi_1(\Sigma), \PSL_n(\R))=\Rep_n(\pi_1(\Sigma))$  the space of representations. Although $\Rep_n(\pi_1(\Sigma))$ itself is a  singular space it contains, as an open set, a smooth manifold 
\[
\overline{\Rep}_n(\pi_1(\Sigma)):=\{\rho\in \Hom(\pi_1(\Sigma), G)\,|\,\rho \text{ is irreducible and }Z_G(\rho) = \{1\}\}/G
\]
where 
\[
Z_G(\rho) = \{g\in G\,|\, g\rho(\gamma) g^{-1} = \rho (\gamma) \text{ for all } \gamma\in \pi_1(\Sigma)\}.
\]
We mostly focus on the smooth manifold $\overline{\Rep}_n(\pi_1(\Sigma))$ because  $\overline{\Rep}_n(\pi_1(\Sigma))$ contains $\Hit_n(\Sigma)$ as a connected component. 

Suppose that $\Sigma$ has boundary components say $\zeta_1, \cdots, \zeta_b$. Let 
\[
\mathscr{B}=\{(\zeta_1,B_1), \cdots, (\zeta_b,B_b)\}
\]
 be a set of pairs each of which consists of a boundary component and a conjugacy class of a purely loxodromic element. Then we can define the following subspace of  $\overline{\Rep}_n(\pi_1(\Sigma))$: 
\[
\overline{\Rep}_n ^{\mathscr{B}}(\pi_1(\Sigma))= \{ [\rho]\in \overline{\Rep}_n(\pi_1(\Sigma))\,|\, \rho(\zeta_i)\in B_i\text{, }i=1,2,\cdots, b \}.
\]
We may define similarly $\Hit^{\mathscr{B}} _n (\Sigma)$. $\overline{\Rep}_n ^{\mathscr{B}}(\pi_1(\Sigma))$ and $\Hit^{\mathscr{B}} _n (\Sigma)$ are interesting because they admit a natural symplectic form $\omega_K ^{\Sigma}$. See Theorem \ref{Ksymp} or \cite{kim1999,guruprasad1997}. 

Let $\mathcal{C}=\{\xi_1, \cdots, \xi_m\}$ be a family of mutually disjoint, non-isotopic essential simple closed curves. If we subtract these curves from $\Sigma$, we get a collection of subsurfaces $\Sigma_1, \cdots, \Sigma_l$.  We assume that $\Sigma_i$ are all  of hyperbolic type. 

 Define
\[
\Hit_n ^{\mathscr{B}}(\Sigma, \mathscr{C})=\{[\rho]\in \Hit_n ^{\mathscr{B}}(\Sigma) 
\,|\,\rho(\xi_i)\in C_i\text{, }i=1,2,\cdots,m\}
\]
where  $\mathscr{C}=\{(\xi_1,C_1), \cdots, (\xi_m,C_m)\}$ is a family of pairs each of which consists of an element of $\mathcal{C}$ and the conjugacy class of a purely loxodromic element. We know that there is a Hamiltonian $\R^{m(n-1)}$-action on $\Hit_n ^{\mathscr{B}} (\Sigma)$ and the moment map of this action takes $\Hit_n ^{\mathscr{B}}(\Sigma,\mathscr{C})$ as a level set over a regular value  (see Goldman \cite{goldman1986}).

Now we consider the quotient $q:\Hit_n ^{\mathscr{B}}(\Sigma,\mathscr{C})\to  \Hit_n ^{\mathscr{B}}(\Sigma,\mathscr{C})/\R^{m(n-1)}$. The restriction map $\Phi=(\iota^* _{\Sigma_1}, \cdots, \iota^* _{\Sigma_l})$ identifies this quotient space with an open subspace of the product space $\Hit^{\mathscr{B}_1} _n (\Sigma_1)\times \cdots \times \Hit^{\mathscr{B}_l} _n (\Sigma_l)$ where 
\[
\mathscr{B}_i=\{(\xi,B)\,|\,\xi\text{ is a component of } \partial \overline{\Sigma_i}\text{ and } (\iota_{\Sigma_i}(\xi),B)\in \mathscr{B}\cup \mathscr{C}\}.
\] 

The quotient $q$ is not only topological but also symplectic in the following sense: $q$ pushes forward the symplectic form $\omega_K ^\Sigma$  and  induces the symplectic form $\widetilde{\omega}_K ^\Sigma$ on the quotient space. On the other hand the the product space  $\Hit^{\mathscr{B}_1} _n (\Sigma_1)\times \cdots \times \Hit^{\mathscr{B}_l} _n (\Sigma_l)$ carries the symplectic form $\omega^{\Sigma_1}_K \oplus \cdots \oplus \omega^{\Sigma_l} _K$. Now we can state our first main theorem. We remark that the theorem holds for general $n$.

\begin{theorem}\label{decompthmintro}
	Let  $\Sigma$ be a compact oriented hyperbolic surface. Then  the map $\Phi$ is a symplectic diffeomorphism from  $\Hit^{\mathscr{B}} _n(\Sigma,\mathscr{C})/\R^{m(n-1)}$ onto  an open submanifold of $\Hit^{\mathscr{B}_1} _n (\Sigma_1)\times \cdots \times \Hit^{\mathscr{B}_l} _n (\Sigma_l)$. 
\end{theorem}
For the precise statement, see Theorem \ref{gendecomp}. 

Theorem \ref{decompthmintro} decomposes $\widetilde{\omega}_K ^\Sigma$ into a sum of  symplectic forms and allows us to obtain some information about the symplectic structure on $\Hit^{\mathscr{B}} _n(\Sigma,\mathscr{C})$ by studying smaller symplectic manifolds individually.  We apply Theorem \ref{decompthmintro} to the case when $\Sigma$ is closed and $\Sigma_i$'s come from a pants decomposition. 

Let $\Sigma$ be a closed oriented hyperbolic surface. Take a pants decomposition of $\Sigma$. That is a choice of a maximal collection of mutually disjoint, non-isotopic essential simple closed oriented curves $\{\xi_1, \cdots, \xi_{3g-3}\}$. Goldman \cite{goldman1990} proves that $\Hit_3(\Sigma)$ can be parametrized by $16g-16$ global parameters which can be classified into three types
\begin{itemize}
	\item \emph{internal parameters}  $(\mathbf{s}_i,\mathbf{t}_i)$ parametrize  $\Hit_3 ^{\mathscr{B}}(P_i)$ for each pants $P_i$. 
	\item \emph{length parameters} $(\ell_i,m_i)$ are positive  numbers associated to each $\xi_i$. 
	\item \emph{twist-bulge parameters} $(u_i,v_i)$ are dual of the length parameters. 
\end{itemize}

Internal and length parameters are canonical. Whereas twist-bulge parameters $u_i$, $v_i$ are rather ambiguous. These $(u_i,v_i)$ coordinates measure the amount of twist-bulge along a curve $\xi_i$ with respect to a certain origin and there is no canonical choice of such a reference point. To remove this ambiguity, we use the relationship between Goldman coordinates and Bonahon-Dreyer coordinates \cite{bonahon2018}. 

After we obtain the well-defined Goldman coordinates, we prove that a canonical part of Goldman coordinates 
\[
(\mathbf{s}_1, \cdots, \mathbf{s}_{2g-2}, \ell_1, m_1, \cdots, \ell_{3g-3},m_{3g-3})
\]
can be completed to  a global Darboux coordinate system.  A  version of the action-angle principle (Theorem \ref{existenceofdarboux}) is essentially used to prove the result. 

\begin{theorem}\label{globaldarbouxintro}
Let $\Sigma$ be a closed oriented surface with genus $g>1$. There is a smooth $\R^{8g-8}$-valued function 
	\[
	(\overline{\mathbf{s}}_1, \cdots, \overline{\mathbf{s}}_{2g-2},  \overline{\ell}_1, \overline{m}_1, \cdots, \overline{\ell}_{3g-3}, \overline{m}_{3g-3})
	\]
on  $(\Hit_3(\Sigma),\omega_G)$ such that  
\[
(\mathbf{s}_1,\cdots, \mathbf{s}_{2g-2}, \ell_1,m_1  \cdots, \ell_{3g-3},m_{3g-3}, \overline{\mathbf{s}}_1, \cdots, \overline{\mathbf{s}}_{2g-2},  \overline{\ell}_1, \overline{m}_1, \cdots, \overline{\ell}_{3g-3}, \overline{m}_{3g-3})
\]
becomes a global Darboux coordinate system. 
\end{theorem}

Recently, Casella-Tate-Tillmann \cite{casella2018} shows that the Goldman bracket and  Fock-Goncharov bracket of $\PSL_3(\R)$ character variety of an open surface coincide on the trace algebra. Their results may cooperate with ours to give a further generalization of Theorem \ref{globaldarbouxintro}.

\subsection{About the proofs}
We first prove a variant of the action-angle principle.  Suppose that we are given a Lagrangian fiber bundle $f: M\to B$ over a connected open subset $B$ of $\R^n$ such that $H^2(B;\R)=0$.  Under certain conditions, the bundle map $f=(f_1, \cdots, f_n)$ has complementary coordinates $g=(g_1, \cdots, g_n)$ such that $(f,g)$ forms a global Darboux coordinate system (Theorem \ref{existenceofdarboux}). Indeed  if the bundle map $f$ has a global Lagrangian section, then we can find complementary coordinate functions (Lemma \ref{sectionimpliescoordinates}). So it is enough to prove the existence of a global Lagrangian section under the given conditions.  We borrow the idea of Duistermaat \cite{duistermaat1980} to show this. We prove that one can find a Lagrangian section locally (Lemma \ref{local}) and then, using  sheaf cohomology theory, we show that the obstruction for gluing local Lagrangian sections vanishes. 

Then we prove the decomposition theorem. We prove Theorem \ref{decompthmintro} by the induction on the number of cutting curves and  we eventually end up with the situation where we cut the surface by a single simple closed curve. There are two cases depending on whether the curve separates the surface or not. We prove the decomposition formulas for each of these cases. 

To this end, we first decompose the tangent space of the Hitchin component. It can be done by  means of the Mayer-Vietoris sequence which is known for the cohomology of group systems  \cite{bieri1978} .  We construct a similar sequence for the parabolic cohomology.

Suppose that $\xi$ is a separating essential simple closed curve in $\Sigma$, so that $\Sigma\setminus \xi = \Sigma_1 \cup \Sigma_2$. The Mayer-Vietoris sequence tells us that the natural inclusion map $\iota_{\Sigma_i}: \pi_1(\Sigma_i)\to \pi_1(\Sigma)$,  $i=1,2$ induces a homomorphism between  tangent spaces 
\[
(\iota_{\Sigma_1} ^* , \iota_{\Sigma_2} ^*):T_{[\rho]} \Hit_n ^{\mathscr{B}}(\Sigma,\mathscr{C})\to T_{[ \rho\circ \iota_{\Sigma_1}]}\Hit^{\mathscr{B}_1} _n (\Sigma_1) \oplus T_{[ \rho\circ \iota_{\Sigma_2}]} \Hit_n ^{\mathscr{B}_2} (\Sigma_2)
\]
whose kernel is spanned by the tangent vectors along twist flows.  Given two vectors $\alpha, \beta \in  T_{[\rho]} \Hit_n ^{\mathscr{B}}(\Sigma,\mathscr{C})$, we  prove  in Theorem \ref{decomppairingsep} that 
\begin{equation}\label{decompgeneralintro}
	\omega^{\Sigma}_K(\alpha,  \beta) =  \omega_K ^{\Sigma_1}(\iota_{\Sigma_1} ^*\alpha, \iota_{\Sigma_1} ^*\beta) +\omega_K ^{\Sigma_2} (\iota_{\Sigma_2} ^* \alpha, \iota_{\Sigma_2} ^* \beta).
\end{equation}

When $\xi$ is non-separating, $\Sigma\setminus \xi = \Sigma_0$, we have a similar homomorphism 
\[
\iota_{\Sigma_0} ^* : T_{[\rho]}\Hit_n ^{\mathscr{B}} (\Sigma) \to T_{[ \rho\circ \iota_{\Sigma_0}]} \Hit_n ^{\mathscr{B}_0}(\Sigma_0)
\]
induced from $\iota_{\Sigma_0} : \pi_1(\Sigma_0) \to \pi_1(\Sigma)$ whose kernel is again spanned by twist flows.  Then, we  show in Theorem \ref{decomppairingnonsep} that
\begin{equation}\label{decompgeneralintrononsep} 
	\omega_K ^{\Sigma}(\alpha, \beta) = \omega_K ^{\Sigma_0} (\iota_{\Sigma_0} ^* \alpha, \iota_{\Sigma_0} ^*\beta). 
\end{equation}
In fact, (\ref{decompgeneralintro}) and  (\ref{decompgeneralintrononsep}) hold under weaker assumptions on $[\rho]$. See Theorem \ref{decomppairingsep} and   Theorem \ref{decomppairingnonsep} for  precise statements. 

We  prove  (\ref{decompgeneralintro}) and (\ref{decompgeneralintrononsep})  by using the Fox calculus.  The key point, which stems from Zocca \cite{zocca1998}, is that we can decompose a relative fundamental class  (Lemma \ref{funclcpt}) into a sum of relative fundamental classes of subsurfaces together with some extra terms. Roughly writing $[\Sigma] = [\Sigma_1]+[\Sigma_2] + \text{extra}$ for the separating case and $[\Sigma] =[\Sigma_0] + \text{extra}$ for the non-separating case.  Then we choose a nice representative in the given cohomology class in such a way that all the extra terms vanish. Applying the decomposition formulas inductively, we can prove Theorem \ref{decompthmintro}. 

Lemma \ref{lagrangian} implies that the map $F:\Hit_3(\Sigma)\to \R^{8g-8}$  assigning to each $[\rho]$ the coordinates 
\[
(\mathbf{s}_1([\rho]),\cdots,\mathbf{s}_{2g-2}([\rho]), \ell_1([\rho]), m_1([\rho]),\cdots, \ell_{3g-3}([\rho]),m_{3g-3}([\rho]))
\]
 is a Lagrangian fiber bundle. Then we show that $F$ satisfies all the conditions of Theorem \ref{existenceofdarboux}. Therefore Theorem \ref{globaldarbouxintro} follows as a consequence of Theorem \ref{existenceofdarboux}.

\subsection{Acknowledgements}
Discussions with Francis Bonahon, William Goldman, Michael Kapovich, Inkang Kim, Ana Cannas da Silva, and Tengren Zhang were very helpful for us to complete our paper.  Sun Zhe  and Johannes Huebschmann kindly explained their work and gave many constructive comments. We specially appreciate their help. The second author would like to give special thanks to Francis Bonahon, Daniel Douglas and Hatice Zeybek for their hospitality and helpful conversations during the visit to USC. Finally the first and second author visited Stanford University for GEAR retreat in 2017 and UC Davis in 2015, where we made large progression on this paper.

\section{The space of representations and Hitchin components}
In this section, we review basic facts on  Hitchin components for both closed and compact surfaces. To describe their tangent spaces, we introduce the group cohomology and the parabolic group cohomology which represent the tangent spaces of $\Hit_n(\Sigma)$ and $\Hit_n ^{\mathscr{B}}(\Sigma)$ respectively. 

\subsection{Definition and properties}

Let $\Sigma$ denote a closed oriented hyperbolic surface. Throughout this paper, the Lie group $G$ always denotes $\PSL_n(\R)$ and $\mathfrak{g}$ the Lie algebra $\mathfrak{sl}_n(\R)$ of $\PSL_n(\R)$. Let
 \[
\Rep(\pi_1(\Sigma), \PSL_n(\R))=\Rep_n(\pi_1(\Sigma)):=\Hom(\pi_1(\Sigma),\PSL_n(\R))/\PSL_n(\R)
\]
be the space of representations. We sometimes consider the GIT quotient instead of the usual one. However  do not have to distinguish them because these two quotients coincide on a subspace $\overline{\Rep}_n(\pi_1(\Sigma))$ defined below and we focus only on $\overline{\Rep}_n(\pi_1(\Sigma))$  throughout this paper.

\begin{definition}
	A \emph{Hitchin component} $\Hit_n(\Sigma)$ is a connected component of $\Rep(\pi_1(\Sigma), \PSL_n(\R))$ that contains a Fuchsian representation.
\end{definition}

When $n=2$, $\Hit_2(\Sigma)$ coincides with the usual Teichm\"uller space. It is known that the Teicum\"uller space is homeomorphic to the cell of dimension $6g-6$ where $g$ is the genus of $\Sigma$. Similar result holds for Hitchin components. Indeed Hitchin \cite{hitchin1992} himself shows that  $\Hit_n(\Sigma)$ is homeomorphic to the cell of dimension $(n^2 -1) \cdot (2g-2)$.

Now suppose that $\Sigma$ is a compact hyperbolic surface possibly with boundary. We can naturally generalize the notion of Hitchin components for such a non-closed surface.

\begin{definition}[Labourie-McShane \cite{labourie2009}]
	Let $\Sigma$ be a compact oriented hyperbolic surface. $[\rho]\in \Rep_n(\pi_1(\Sigma))$ is said to be a \emph{Hitchin representation} if $\rho$ can be continuously deformed into a Fuchsian representation in such a way that the holonomies of boundary components are purely loxodromic (i.e., it is diagonalizable and all eigenvalues are distinct positive real numbers) in the course of the deformation. A connected component of $\Rep_n (\pi_1(\Sigma))$ that consists of Hitchin representations is denoted by the same notation $\Hit_n(\Sigma)$. 
\end{definition}

Suppose that $\Sigma_0$ is a hyperbolic incompressible subsurface of a closed hyperbolic surface $\Sigma$. Given $[\rho]\in \Hit_n(\Sigma)$, its restriction $\rho|_{\pi_1(\Sigma_0)}$ to  the subgroup $\pi_1(\Sigma_0)$ is also in $\Hit_n (\Sigma_0)$. See Theorem 9.1 of Labourie-McShane \cite{labourie2009}. 

The space $\Rep_n(\pi_1(\Sigma))$ contains an open subspace, the space of `good' representations 
\[
\overline{\Rep}_n (\pi_1(\Sigma)):=\Hom_s(\pi_1(\Sigma),G)/G
\]
where 
\[
\Hom_s(\pi_1(\Sigma),G):=\{\rho\in \Hom(\pi_1(\Sigma),G)\,|\,\rho\text{ is irreducible and } Z_G(\rho)=\{1\} \}.
\]

Suppose that $\rho\in \Hom_s(\pi_1(\Sigma),G)$ is given. Let $X \in \mathfrak{sl}_n(\R)$ be an $\Ad_\rho$-invariant element. Then $\exp X$ is in $Z_G(\rho)$. So by definition of $\overline{\Rep}_n(\pi_1(\Sigma))$, we have $\exp X=1$ or, equivalently, $X=0$. It follows that $\mathfrak{g}_{\rho}$ has no nontrivial $\Ad_\rho$ invariant element. Therefore $\Hom_s(\pi_1(\Sigma),G)$ is a smooth manifold. Moreover,  it is proven by Johnson-Millson \cite{Johnson1987} that the $G$-action on $\Hom_s(\pi_1(\Sigma),G)$ is proper and free. Consequently the quotient space $\overline{\Rep}_n(\pi_1(\Sigma))$ is also a smooth manifold. 

Hitchin representations for a compact surface have many interesting properties. We summarize them as the following lemma, which is implicitly used several times throughout this paper. 

\begin{lemma}\label{NoInvariantElement} Let $\Sigma$ be a compact oriented hyperbolic surface. Let $[\rho]\in \Hit_n(\Sigma)$. 
\begin{itemize}
\item $\rho$ is faithful, irreducible and discrete.
\item For each nontrivial $\gamma\in \pi_1(\Sigma)$, $\rho(\gamma)$ is  purely loxodromic.
\item The centralizer of $\rho$, $Z_G(\rho)$, is trivial. 
\end{itemize}
In particular $\Hit_n(\Sigma)$ is a connected component of $\overline{\Rep}_{n}(\pi_1(\Sigma))$.
\end{lemma}
\begin{proof}
First assume that $\Sigma$ is closed. Then first two statements are nothing but Proposition 3.4 of Labourie \cite{labourie2006}. 

Suppose that $\Sigma$ has a nonempty boundary component. We consider the Hitchin double $\widehat{\rho}$. It is known that $\widehat{\rho}$ is in the Hitchin component $\Hit_n(\widehat{\Sigma})$ of the double $\widehat{\Sigma}$ of $\Sigma$. See Corollary 9.2.2.4 of \cite{labourie2009}.  As $\widehat{\Sigma}$ is closed, we  know that  $\widehat{\rho}$ is discrete, faithful and  $\widehat{\rho}(\gamma)$ is purely loxodromic for any nontrivial element $\gamma$. It follows that the $\rho$ has the same properties. 

We now show that $\rho$ is irreducible. By Theorem 9.1 of \cite{labourie2009}, $\rho$ is a positive representation. Therefore by Lemma 5.12 of Guichard-Wienhard \cite{guichard2012}, $\rho$ is irreducible. 

	For the third statement, suppose that $X$ is in the center of $\rho(\pi_1(\Sigma))$. Since $\rho(\gamma)$ is purely loxodromic, we observe that $X$ must be diagonal. Therefore, Schur's lemma is applied so we can conclude that $X$ must be a scalar. 
\end{proof}

Let us introduce a nice submanifold of $\Hit_n(\Sigma)$ which is essential for our discussion. 
\begin{definition}
	Let  $\Sigma$ be compact oriented hyperbolic surface with boundary components $\{\zeta_1, \cdots, \zeta_b\}$.  By a \emph{boundary frame}, we mean a collection  $\mathscr{B}=\{(\zeta_1,C_1), \cdots, (\zeta_b,C_b)\}$ of pairs each of which consists of a boundary component and conjugacy class in $G$.  Given a boundary frame $\mathscr{B}$, we define the following space
	\[
	\Hit^{\mathscr{B}} _n (\Sigma)= \{[\rho]\in \Hit_n(\Sigma)\,|\, \rho(\zeta_i) \in C_i\text{ for }i=1,2,\cdots, b\}. 
	\] 
We also define $\overline{\Rep}_n ^{\mathscr{B}} (\pi_1(\Sigma))\subset \overline{\Rep}_n  (\pi_1(\Sigma))$  in the same fashion.

Let $\mathcal{C}= \{\xi_1, \cdots, \xi_m\}$ be a collection of pairwise disjoint, non-isotopic essential simple closed curves. A $\mathcal{C}$-frame is a family  $\mathscr{C}=\{(\xi_1,C_1), \cdots, (\xi_m,C_m)\}$ of pairs each of which consists of $\xi_i$ and a conjugacy class in $G$. Given a $\mathcal{C}$-frame $\mathscr{C}$, define
\begin{align*}
	\Hit^{\mathscr{B}} _n (\Sigma,\mathscr{C})& = \{[\rho]\in \Hit_n ^{\mathscr{B}} (\Sigma)\,|\, \rho(\xi_i) \in C_i\text{, }i=1,2,\cdots,m\},\quad \text{and}\\
	\overline{\Rep}_n ^{\mathscr{B}}(\pi_1(\Sigma),\mathscr{C})&=\{[\rho]\in \overline{\Rep}_n ^{\mathscr{B}} (\pi_1(\Sigma))\,|\, \rho(\xi_i) \in C_i\text{, }i=1,2,\cdots,m\}. 
\end{align*}
\end{definition}
To be more precise, we should understand $\zeta_i$ (and $\xi_i$) as a loop at the base point $p$ of $\pi_1(\Sigma, p)$ by choosing a path from $p$ to a point in $\zeta_i$ (and $\xi_i$). However  since we are dealing with the conjugacy classes, we may ignore such a technicality.

We observe that $\Hit_n(\Sigma)=\bigcup \Hit_n ^{\mathscr{B}}(\Sigma)$ where the union runs over all possible choice of boundary frames. This foliation plays the key role in the study of Poisson geometry of $\Hit_n(\Sigma)$.

\subsection{Group cohomology}
Cohomology of group is a model for the tangent spaces of $\overline{\Rep}_n(\pi_1(\Sigma))$. In this subsection we review a definition of group cohomology and its properties.

Let $\Gamma$ be a finitely presented group. Given a representation $\rho:\Gamma \to G$, we denote by $\mathfrak{g}_{\rho}$  the $\Gamma$-module $\mathfrak{g}$ under the action $\Ad _\rho$.  If the action is clear from the context, we omit the subscript $\rho$ and simply write $\mathfrak{g}$ instead of $\mathfrak{g}_\rho$. 

By a resolution over $\Gamma$, we mean any projective resolution of $M=\R$ or $\mathbb{Z}$
\[
\cdots \to R_2\to R_1 \to R_0 \to M \to 0
\]
where $M$ is regarded as a trivial $M \Gamma$-module. Then the group cohomology $H^q(\Gamma; V)$ with coefficient in a $M\Gamma$-module $V$ is the cohomology  of the complex $\Hom_{\Gamma} (R_\ast (\Gamma), V)$. Our major concern is the case where $M=\R$ and $V=\mathfrak{g}_\rho$. 

We mostly use the \emph{normalized bar resolution} $(\mathbf{B}_\ast(\Gamma),\der_\Gamma)$ throughout this paper. Recall that $\mathbf{B}_q (\Gamma)$ is a free $\Gamma$-module on symbols $[x_1|x_2| \cdots |x_q]$ where $x_1, \cdots, x_q\in \Gamma\setminus\{1\}$.

Now we give a relative version of group cohomology. We  follow Trotter's paper \cite{trotter1962} (see also  \cite[section 1]{guruprasad1997}). A group system  is, by definition,  a pair  $(\Gamma, \mathcal{S})$ of finitely presented group $\Gamma$ and a collection $\mathcal{S}=\{\Gamma_1, \cdots, \Gamma_m\}$ of its finitely presented subgroups $\Gamma_1, \cdots, \Gamma_m$.

\begin{definition}
	Let $M=\R$ or $\Z$. An \emph{auxiliary resolution} $(R_\ast,A_\ast ^i)$ over the group system $(\Gamma, \mathcal{S})$, or simply $(\Gamma, \mathcal{S})$-resolution, consists of 
	\begin{itemize}
		\item $R_\ast$, a  resolution over  $\Gamma$
		\item $A^i _\ast$,  a  resolution over  $\Gamma_{i}$
		\item $A_\ast:= \bigoplus_{i=1} ^m  M\Gamma \otimes_{M\Gamma_{i}} A^i _\ast$  is a direct summand of $R_\ast$.
	\end{itemize}
\end{definition}
Since $A_\ast$ is a direct summand of $R_\ast$, we can form a short exact sequence of chain complexes
\begin{equation}\label{esofrelative}
0\to A_\ast \to R_\ast \to R_\ast/A_\ast \to 0.
\end{equation}
For a given $\Gamma$-module $\mathfrak{g}_{\rho} =\mathfrak{g}$, we apply the $\Hom_{\Gamma}(-,\mathfrak{g})$ functor on this exact sequence. Then we get the long exact sequence
\[
\cdots \to H^q (\Gamma,\mathcal{S}; \mathfrak{g}) \to H^q(\Gamma; \mathfrak{g}) \to H^q(\mathcal{S};\mathfrak{g}) \to H^{q+1}(\Gamma,\mathcal{S}; \mathfrak{g}) \to \cdots. 
\]

Note that $H^q(\mathcal{S};\mathfrak{g}_\rho)\approx \bigoplus H^q(\Gamma_{i};\mathfrak{g}_{\rho|_{\Gamma_i}})$.

\begin{definition}
	The parabolic cohomology of $\Gamma$ of degree $q$ with coefficient in $\mathfrak{g}_\rho$,   $H^q _{\mathrm{par}} (\Gamma,\mathcal{S}; \mathfrak{g}_\rho)$, is the image of $H^q(\Gamma, \mathcal{S};\mathfrak{g}_\rho)$ in $H^q(\Gamma; \mathfrak{g}_\rho)$. In other words
	\[
	H^q _{\mathrm{par}} (\Gamma, \mathcal{S}; \mathfrak{g}_\rho) \approx H^q(\Gamma,\mathcal{S}; \mathfrak{g}_\rho) / H^{q-1} (\mathcal{S}; \mathfrak{g}_\rho). 
	\]
\end{definition}

We are interested in the case where $q=1$ and $\Gamma=\pi_1(\Sigma)$. In the appendix, we describe how to compute the (1st) parabolic cohomology by finding a nice resolution over a group system $(\Gamma, \mathcal{S})$. In terms of the  normalized bar resolution $\mathbf{B}_\ast(\Gamma)$, elements of $H^1_{\mathrm{par}}(\Gamma, \mathcal{S};\mathfrak{g})$ can be represented by parabolic cocycles
\[
Z^1_{\mathrm{par}} (\Gamma, \mathcal{S}; \mathfrak{g}):=\{\alpha\in Z^1(\Gamma, \mathfrak{g})\,|\, \iota^\# _{\Gamma_{i}} (\alpha) \in B^1(\Gamma_i, \mathfrak{g})\}
\]
where $\iota^\# _{\Gamma_{i}}$ is the restriction defined at the end of this subsection. 

\begin{remark}\label{conj}
	Consider a group system $(\Gamma, \mathcal{S}')$, $\mathcal{S}' = \{\Gamma_1' ,\cdots, \Gamma_m'\}$, which is conjugate to $(\Gamma,\mathcal{S})$ in the sense that for each $i=1,2,\cdots, m$ there is a $g_i\in \Gamma$ such that  $\Gamma_{i} ' = g_i \Gamma_{i} g_i ^{-1}$. Then the parabolic cohomology $H^1_{\mathrm{par}} (\Gamma,\mathcal{S}';\mathfrak{g})$ is the same as $H^1_{\mathrm{par}} (\Gamma,\mathcal{S};\mathfrak{g})$ because $\iota^\# _\gamma (\alpha)\in B^1(\Gamma_i; \mathfrak{g})$ if and only if $\iota^\# _{g\gamma g^{-1}}(\alpha) \in B^1(\Gamma_i; \mathfrak{g})$. In other words, the parabolic cohomology $H^q _{\mathrm{par}} (\Gamma, \mathcal{S}; \mathfrak{g})$ depends only on the conjugacy class of $\Gamma_{i}$. So in regard of parabolic cohomology, we define a group system as a pair of group $\Gamma$ and a family of \emph{conjugacy classes} of subgroups $\Gamma_1, \cdots, \Gamma_m$ of $\Gamma$. 
\end{remark}

We finish this subsection by introducing  the restriction map. Let $\mathbf{B}_\ast(\Gamma)$ be the normalized bar resolution over $\Gamma$. Suppose that we are given a subgroup $\iota_{\Gamma'} : \Gamma'\to \Gamma$. Then we have a natural chain map  $\iota_{\Gamma'} ^\# : \Hom_{\Gamma} (\mathbf{B}_\ast (\Gamma),\mathfrak{g})\to \Hom_{\Gamma} (\R \Gamma\otimes \mathbf{B}_\ast(\Gamma') , \mathfrak{g}) $. Since $\Hom_{\Gamma} (\R \Gamma\otimes \mathbf{B}_\ast(\Gamma') , \mathfrak{g})\approx \Hom_{\Gamma'} (\mathbf{B}_\ast (\Gamma'), \mathfrak{g})$ as chain complexes of $\R$-vector spaces, this inclusion induces a homomorphism  $H^1(\Gamma,\mathfrak{g}) \to H^1(\Gamma', \mathfrak{g})$ which we  denote by $\iota_{\Gamma'}^*[\alpha]=[\iota_{\Gamma'}^\# (\alpha)]$ where $[\alpha]\in H^1(\Gamma, \mathfrak{g})$.  Similarly, if $(\Gamma', \mathcal{S}')$ is a group subsystem of $(\Gamma, \mathcal{S})$, we have the natural restriction map $H^1_{\mathrm{par}}(\Gamma, \mathcal{S};\mathfrak{g}) \to H^1_{\mathrm{par}} (\Gamma', \mathcal{S}';\mathfrak{g})$.

\subsection{Tangent spaces of $\overline{\Rep}_n (\pi_1(\Sigma))$}
It is well-know that the tangent space of $\overline{\Rep}_n (\pi_1(\Sigma))$ at each point $[\rho]\in \overline{\Rep}_n (\pi_1(\Sigma))$ can be identified with $H^1(\pi_1(\Sigma); \mathfrak{g}_{\rho})$. See for example \cite{weil1964}, \cite{goldman1984}, \cite{guruprasad1997} and \cite{labourie2013}.  Since   $\Hit_n(\Sigma)$ is a component of $\overline{\Rep}_n(\pi_1(\Sigma))$ (Lemma \ref{NoInvariantElement}), we can say that the tangent space of $\Hit_n(\Sigma)$ at $[\rho]$ is $H^1(\pi_1(\Sigma); \mathfrak{g}_{\rho})$. 

As in the closed surface case, we have a nice description of local geometry for $\Hit_n ^{\mathscr{B}}(\Sigma)$. Guruprasad-Huebschmann-Jeffrey-Weinstein \cite{guruprasad1997} shows that the tangent space of $\Hit_n ^{\mathscr{B}}(\Sigma)$ at $[\rho]\in \Hit_n ^{\mathscr{B}}(\Sigma)$ can be identified with the parabolic group cohomology
	\begin{align*}
	T_{[\rho]} \overline{\Rep}_n ^{\mathscr{B}}(\pi_1(\Sigma)) &\approx H^1 _{\mathrm{par}} (\pi_1(\Sigma),\{\langle \zeta_1\rangle, \cdots, \langle \zeta_b \rangle\} ;\mathfrak{g}_{ \rho}).
	\end{align*}
More generally, we show the following in the appendix, Proposition \ref{tangentA}.

\begin{proposition}\label{tangent}
	Let $\Sigma$ be a compact oriented hyperbolic surface possibly with boundary components $\{\zeta_1, \cdots, \zeta_b\}$. Let $\{\xi_1, \cdots, \xi_m\}$ be mutually disjoint, non-isotopic essential simple closed curves.  At each $[\rho]\in \overline{\Rep}_n ^{\mathscr{B}}(\pi_1(\Sigma),\mathscr{C})$, 
	\[
	T_{[\rho]}\overline{\Rep}_n ^{\mathscr{B}}(\pi_1(\Sigma),\mathscr{C}) \approx H^1 _{\mathrm{par}}(\pi_1 (\Sigma),\{\langle{\xi_1}\rangle, \cdots, \langle\xi_m\rangle,\langle\zeta_1\rangle, \cdots, \langle \zeta_b\rangle\};\mathfrak{g}_\rho)
	\]
	where $\rho$ is a representative of the class $[\rho]$. 
\end{proposition}

Recall, by Remark \ref{conj}, that particular choices of a subgroups $\langle\xi_i \rangle$ and $\langle \zeta_i \rangle$ within their conjugacy classes are not important.

\section{Aspects of symplectic geometry}
In this section, we collect elements of symplectic geometry related to our discussion. We review the construction of the Atiyah-Bott-Goldman symplectic form on $\Hit_n(\Sigma)$ as well as the  symplectic form of  H. Kim  on $\Hit_n ^\mathscr{B}(\Sigma)$.  

The key part of this section is subsection \ref{aaprinciple} where we  prove a variant of the action-angle principle    (Theorem \ref{existenceofdarboux}) that allows us to find global Darboux coordinates under certain conditions. 

\subsection{Definitions and properties}
A symplectic manifold is a smooth manifold $M$ with a non-degenerate closed 2-form $\omega$. Given a smooth function $f\in C^\infty(M)$, the Hamiltonian vector field  associated to $f$ is characterized by a unique vector field $\mathbb{X}_f$ such that
\[
\omega(\mathbb{X}_f,Y)= \der f (Y) = Yf
\]
for any vector field $Y$. Since $\der \omega =0$, we have
\[
\mathcal{L}_X \omega = \der \iota_X \omega + \iota_X \der \omega = \der \iota_X \omega
\] 
where $\mathcal{L}$ denotes the Lie derivative. Hence a vector field $X$ preserves $\omega$ if and only if the induced 1-form $\iota_X \omega=\omega(X,-)$ is closed. In particular,  $\mathcal{L}_{\mathbb{X}_f}\omega=0$. It follows that the flow $\Psi^t$ associated to the vector field $\mathbb{X}_f$ is a symplectomorphism for each $t\in\R$ whenever $\Psi^t$ is defined. 

Suppose that we have a symplectic manifold $(M, \omega)$. By defining the Poisson bracket of smooth functions $f,g\in C^\infty(M)$ by $\{f,g\} = \omega(\mathbb{X}_f, \mathbb{X}_g)$, we can turn $M$ into a Poisson manifold.

Let $x\in M^{2n}$ be any point of a symplectic manifold. \emph{Darboux's theorem} states that there is a coordinate neighborhood $(U, (f_1, \cdots, f_n,g_1, \cdots, g_n))$ of $x$ such that $\omega|_U = \sum_{i=1} ^n  \der f_i \wedge \der g_i$. Such coordinates are called (local) \emph{Darboux coordinates}. Global Darboux coordinates are global coordinates 
\[
(f_1, \cdots, f_n,g_1, \cdots, g_n):M\to \R^{2n}
\]
of $M$ where $\omega$ can be expressed as $\omega= \sum_{i=1} ^{n} \der f_i \wedge \der g_i$. 

There is a particularly important symplectic manifold which arises naturally from any manifold. Let $M$ be any smooth $n$-manifold. The cotangent bundle $p:T^*M\to M$ has a canonical 1-form $\lambda_{\text{can}}$ which is characterized by the following property: $\sigma^* \lambda_{\text{can}} = \sigma$ for every 1-form $\sigma$. If $(U, (x_1, \cdots, x_n))$ is a local coordinate chart of $M$, then there are natural coordinates $y_i$ that parametrize each $T^* _q M$, $q\in M$, with respect to an ordered basis $\{\der x_1, \cdots, \der x_n\}$. Then we observe that $(p^{-1}(U),(x_1,\cdots, x_n, y_1, \cdots, y_n))$ is a local coordinate chart of $T^* M$.  With respect to this coordinates, $\lambda_{\text{can}}$ can be written as $\sum_{i=1} ^n y_i \der x_i$. Define the 2-form $\omega_{\text{can}}$ by $\omega_{\text{can}} = - \der \lambda_{\text{can}}$. Then $(T^*M , \omega_{\text{can}})$ is a symplectic manifold of dimension $2n$. Observe that that the Hamiltonian flow on $p^{-1}(U)$ associated to  each coordinate function $x_i$ is linear. 

\subsection{Marsden-Weinstein quotient}
Let $(M, \omega)$ be a symplectic manifold. Suppose that a Lie group $K$ acts on $M$ as symplectomorphisms. This action is called \emph{weakly Hamiltonian} if for each $X\in \mathfrak{k}$, its fundamental vector field $\xi_X$ is a Hamiltonian vector field. That is, for each $X\in \mathfrak{k}$, there is a unique smooth function $H_X$ such that $\iota_{\xi_X}\omega = \der H_X$. A weakly Hamiltonian  action of $K$ on $M$  is \emph{Hamiltonian} if the rule $X\mapsto H_X$ is a Lie algebra homomorphism, i.e., $H_{[X,Y]} = \{H_X , H_Y\}$ for all $X, Y \in \mathfrak{k}$. The obstruction for a weakly Hamiltonian action to be  Hamiltonian  stays in the Lie algebra cohomology $H^2(\mathfrak{k}; \R)$. In particular, if $H^2(\mathfrak{k};\R)=0$, every weakly Hamiltonian action of $K$ becomes Hamiltonian. 

Suppose that we have a Hamiltonian action by a Lie group $K$. For each $x\in M$  there is a unique element $\mu(x)\in \mathfrak{k}^*$ such that  $H_X (x) = \langle \mu(x) , X \rangle$ for all $X\in \mathfrak{k}$ where $\langle\cdot, \cdot\rangle$ is the canonical pairing between $\mathfrak{k}^*$ and $\mathfrak{k}$. The map $\mu: x\mapsto \mu(x)$ so defined is called the \emph{moment map}. 

If the action is Hamiltonian then $\mu$ is $K$-equivariant where $\mathfrak{k}^*$ is equipped with the coadjoint action. Observe that if $z\in \mathfrak{k}^*$ is a regular value of $\mu$ and is a fixed point of a coadjoint action,  then $\mu^{-1}(z)$ is an  invariant coisotropic submanifold. For each $x\in \mu^{-1}(z)$, the symplectic complement $T_x \mu^{-1}(z)^\omega$ is precisely the tangent space of the orbit space $K\cdot x$. Therefore we can hope that a new symplectic manifold may be constructed by collapsing this `bad' directions. That is the ideal of \emph{symplectic reduction} which we state as follow:

\begin{theorem}[Symplectic reduction or Marsden-Weinstein quotient]\label{MWq} Let $(M, \omega)$ be a symplectic manifold with a Hamiltonian action of a Lie group $K$. Let $\mu$ be the  moment map.  Suppose that $z\in \mathfrak{k}^*$ is a  fixed point of the coadjoint action and that it is a regular value of $\mu$.  If, in addition, $K$ acts properly and freely on $\mu^{-1}(z)$, then the quotient 
\[
\mu^{-1}(z) /K
\]
is a smooth manifold and carries the canonical symplectic structure  $\widetilde{\omega}$ which is uniquely determined by the property $\omega|_{\mu^{-1}(z)} = (q^* \widetilde{\omega})|_{\mu^{-1}(z)}$ where  $q:\mu^{-1}(z) \to \mu^{-1}(z) /K$ is the quotient map. 
\end{theorem}
One can find more details about symplectic reduction for example in \cite{mcduff2017, marsden2007, da2001}.

\subsection{The Fox calculus and the Atiyah-Bott-Goldman symplectic form} Motivated by Atiyah-Bott \cite{atiyah1983}, Goldman \cite{goldman1984} gives an algebraic construction of the symplectic form on  $\overline{\Rep}_n (\pi_1(\Sigma))$. This symplectic form is now called the Atiyah-Bott-Goldman symplectic form which we denote by $\omega_G ^\Sigma$ or simply $\omega_G$ if the surface $\Sigma$ is understood from the context.  The following theorem,  one of the main result of free differential calculus by Fox \cite{fox1953}, is the key ingredient of Goldman's construction.

\begin{theorem}[Fox \cite{fox1953}]\label{fox}
Let $\Gamma$ be a free group on free generators $s_1, \cdots, s_n$.  Let $\Z\Gamma$ be the group ring. There is a collection of operators $\frac{\partial}{\partial s_i}: \Z \Gamma \to \Z \Gamma$, $i=1,2,\cdots, n$ having the following properties
	\begin{itemize}
		\item Given $x,y\in \Z \Gamma$, 
		\[
		\frac{\partial xy}{\partial s_i} = x \frac{\partial y}{\partial s_i} + \frac{\partial x}{\partial s_i } \epsilon(y)
		\]
		where $\epsilon(x)$ denotes the sum of coefficients of $x$. 
		\item $\displaystyle \frac{\partial s_i } {\partial s_j } = \begin{cases}1 & i=j \\ 0 &i\ne j \end{cases}$
		\item For any $x \in \Z\Gamma$, 
		\[
		x = \epsilon(x)1+ \sum_{i=1} ^n \frac{\partial x} {\partial s_i} (s_i-1).
		\] 
	\end{itemize}
\end{theorem}

Theorem \ref{fox} allows us to construct a non-trivial homology class in $H_2(\pi_1(\Sigma), \Z)$. To do this we use the normalized bar resolution  $\mathbf{B}_\ast(\Gamma)$  over $\Gamma$. For our convenience, let us make the following convention: $[a\pm b| x] = [a|x]\pm [b|x] \in \mathbf{B}_2(\Gamma)$ for any $a,b,x\in \Gamma\setminus\{1\}$. Now, choose a presentation
\[
	\langle x_1, y_1, x_2, y_2, \cdots, x_g, y_g\,|\, R\rangle
\]
for $\Gamma=\pi_1(\Sigma)$ where $R=\prod_{i=1} ^g [x_i,y_i] $. Then 
\[
	[\Sigma] = \sum_{i=1} ^g \left. \left[\frac{\partial R}{\partial x_i}\right|   x_i\right] + \sum_{i=1} ^g \left. \left[ \frac{\partial R}{\partial y_i}\right| y_i\right] 
\]
represents a generator of $H_2(\Gamma; \Z)\approx \Z$. See Proposition 3.9 of Goldman \cite{goldman1984}.  We call $[\Sigma]$ a \emph{fundamental class} of $\Gamma$. If we use a different relation say $R' =hRh^{-1}$ for some $h\in \Gamma$, then the fundamental class with respect to the new relation $R'$ reads 
\[
 \sum_{i=1} ^g \left. \left[ h\frac{\partial R}{\partial x_i}\right| x_i\right] + \sum_{i=1} ^g \left.\left[ h \frac{\partial R}{\partial y_i}\right|  y_i\right]
\]
which is homologues to the original fundamental class $[\Sigma]$.

\begin{theorem}[Goldman \cite{goldman1984}]\label{goldmansymplectic} Let $\Sigma$ be a closed oriented hyperbolic surface.  Then $\overline{\Rep}_n(\Gamma)$ is a symplectic manifold with the symplectic form defined  at each point $[\rho]\in \overline{\Rep}_n (\Gamma)$ by
	\[
	\omega_G  ^\Sigma([\alpha],[\beta]) = \langle \alpha \smile \beta, [\Sigma] \rangle,
	\]
	where $[\alpha],[\beta]\in H^1(\Gamma, \mathfrak{g}_{\rho})$. 
\end{theorem}

Explicitly, in terms of normalized bar resolution,
\begin{equation}\label{explicitformula}
\langle \alpha \smile \beta, [\Sigma]\rangle = - \sum_{i=1} ^g \trace \alpha\left(\overline{ \frac{\partial R}{\partial x_i}}\right) \beta(x_i) -\sum_{i=1} ^g \trace \alpha\left(\overline{ \frac{\partial R}{\partial y_i}}\right) \beta(y_i)
\end{equation}
where $\overline{(\cdot)}: \R \Gamma \to \R \Gamma$ is the map induced from the inversion that sends an element $g$ of $\Gamma$ to $g^{-1}$.

Suppose that $\Sigma$ has a boundary component. Then $\overline{\Rep}_n (\pi_1(\Sigma))$ is no longer symplectic. However by controlling the boundary conjugacy classes, we get a foliation each leaf of which is a symplectic manifold. See Theorem 2.2.1 of Audin \cite{audin1997} for more details. To present the result, we need a relative version of a fundamental class.  Choose a path $\eta_i$ from a base point $p$ to a point in $\zeta_i$ in such a way that $z_i := \zeta_i ^{\eta_i}$  fits into a standard presentation 
\begin{equation}\label{presentationforcompact}
	\Gamma=\pi_1(\Sigma,p)= \langle x_1, y_1, x_2, y_2, \cdots, x_g, y_g, z_1, \cdots, z_b\,|\,R\rangle
\end{equation}
where  $R= \prod_{i=1} ^g [x_i,y_i] \prod_{j=1} ^b z_j$. 

\begin{lemma}\label{funclcpt}
 Let 
\[
	[\Sigma] = \sum_{i=1} ^g\left(\,\left.\left[ \frac{\partial R} {\partial x_i}\right| x_i\right] + \left[\left. \frac{\partial R}{\partial y_i} \right| y_i\right]\, \right) + \sum_{j=1} ^ b  \left.\left[ \frac{\partial R}{\partial z_j}\right| z_j\right]
\]
 be a (absolute) 2-chain in $\mathbf{B}_2(\Gamma)\otimes \Z$. There is  an auxiliary resolution  $(R_\ast, A^i _\ast)$ over the group system  $(\Gamma, \{\langle z_i \rangle\})$ with a chain equivalence $\mathbf{B}_\ast (\Gamma)\otimes \Z \to R_\ast \otimes \Z$ such that  the  image of $[\Sigma]$ under the map 
\[
	\mathbf{B}_2(\Gamma)\otimes \Z \to R_2 \otimes \Z \to (R_2/A_2) \otimes \Z
\]
represents a generator of $H_2(\Gamma, \{ \langle z_i\rangle\}; \Z)\approx \Z$.  We call $[\Sigma]$ a \emph{relative fundamental class} of $\Gamma$. 
\end{lemma}

We  prove Lemma \ref{funclcpt} in the appendix.

Note that if $\Sigma$ is not closed,  $H_2(\Gamma;\Z)=0$ and that  $[\Sigma]$ itself is not even a 2-cycle in the absolute chain complex $\mathbf{B}_\ast (\Gamma) \otimes \Z$. 

\begin{theorem}[Guruprasad-Huebschmann-Jeffrey-Weinstein \cite{guruprasad1997}, H. Kim \cite{kim1999}]\label{Ksymp}	Let $\Sigma$ be a compact oriented hyperbolic surface possibly with boundary. Let $\mathscr{B}$ be a boundary frame. Fix a presentation of $\pi_1(\Sigma)$ as in (\ref{presentationforcompact}) and a representation $\rho$ such that $[\rho]\in \overline{\Rep}_{n} ^{\mathscr{B}}(\Gamma)$. Let $[\alpha],[\beta]\in H^1_{\mathrm{par}} (\Gamma,\{\langle z_i \rangle\} ; \mathfrak{g}_{\rho})$. We choose, for each boundary component $z_i$, an element $X_i \in \Hom_\Gamma(\mathbf{B}_0(\langle z_i\rangle),\mathfrak{g}) \approx \mathfrak{g}$ such that  $\iota_{\langle z_i\rangle} ^{\#} \alpha = \der_{\langle z_i\rangle} X_i$.  Define $\omega_{K} ^\Sigma$ to be
	\[
	\omega_K ^{\Sigma}([\alpha], [\beta]) = \langle \alpha \smile \beta, [\Sigma]\rangle - \sum_{i=1} ^b \Tr X_i \beta(z_i)
	\]
where $ \langle \alpha \smile \beta, [\Sigma]\rangle$ is defined as in (\ref{explicitformula}). Then $\omega_K$ is a closed, non-degenerate 2-form and $(\overline{\Rep}_n ^{\mathscr{B}}(\Gamma),\omega_K^\Sigma)$ becomes a symplectic manifold.
\end{theorem}

\begin{remark} Unlike Theorem \ref{goldmansymplectic}, the operation $\langle \alpha \smile \beta, [\Sigma]\rangle$ in Theorem \ref{Ksymp} is defined only on the chain level and cannot descend to cohoomlogy.  In fact, H. Kim \cite{kim1999} computes that for any $X\in C^0(\Gamma;\mathfrak{g})$ and $\beta \in Z^1(\Gamma; \mathfrak{g})$, 
\[
\langle \der_\Gamma X \smile \beta, [\Sigma]\rangle = \sum_{i=1} ^b \Tr X \beta(z_i)\ne 0.
\]
Nevertheless, the whole expression $\omega_K ^\Sigma ([\alpha], [\beta])$ is a nice cohomological operation. 
\end{remark}

\begin{remark}
	The formula given in Lemma 8.4 of Guruprasad-Huebschmann-Jeffrey-Weinstein \cite{guruprasad1997} is incorrect. Since (with notation in \cite{guruprasad1997}) $\langle c, u\smile v\rangle$ is not antisymmetric, we have to replace $\langle c, u\smile v\rangle$ with $\frac{1}{2}(\langle c, u\smile v\rangle-\langle c, v\smile u\rangle)$. Then the formula of Lemma 8.4 of Guruprasad-Huebschmann-Jeffrey-Weinstein \cite{guruprasad1997} and Theorem 5.6 of H. Kim \cite{kim1999} are identical. 
\end{remark}

We state the relevant lemmas to prove Theorem \ref{Ksymp}.

\begin{lemma}[H. Kim \cite{kim1999}] Suppose that $X_i'\in \mathfrak{g}$ is another element such that  $\iota_{\langle z_i\rangle} ^{\#} \alpha = \der_{\langle z_i\rangle} X_i'$, for $i=1,2,\cdots,b$. Then 
\[
 \langle \alpha \smile \beta, [\Sigma]\rangle - \sum_{i=1} ^b \Tr X_i \beta(z_i)= \langle \alpha \smile \beta, [\Sigma]\rangle - \sum_{i=1} ^b \Tr X'_i \beta(z_i).
\]
In particular, $\omega_K ^\Sigma$ is well-defined in the chain level. 
\end{lemma}
\begin{proof}
This lemma is also proven in H. Kim \cite{kim1999} but is not stated in an explicit form. So we recall the proof here. 

Suppose that  there are two elements $X_i$ and $X_i '$ of $\mathfrak{g}=C^0(\langle z_i\rangle , \mathfrak{g})$ such that $\iota^\# _{\langle z_i\rangle} \alpha = \der_{\langle z_i\rangle} X_i=\der_{\langle z_i\rangle} X_i'$. Let $Y_i = X_i - X_i'$. Then we have  $\der_{\langle z_i\rangle} Y_i =0$. That is $z_i \cdot Y_i - Y_i =0$.   Since $\beta\in Z^1_{\mathrm{par}}(\Gamma,\{\langle z_i\rangle \};\mathfrak{g})$, we can find $Z_i\in \mathfrak{g}$ such that $\beta(z_i) = z_i \cdot Z_i -Z_i$. Then we compute
\begin{align*}
\sum_{i=1} ^b\Tr  (X_i - X_i') \beta(z_i) & = \sum_{i=1} ^b\Tr  Y_i \beta(z_i)\\
&= \sum_{i=1} ^b \Tr Y_i (z_i \cdot Z_i -Z_i) \\
&= \sum _{i=1} ^b \Tr (z_i ^{-1} \cdot Y_i -Y_i )Z_i = 0.
\end{align*}
Therefore, $\omega_K ^\Sigma$ is independent of choice of $X_i$. 
\end{proof}
\begin{lemma} Suppose that $\alpha=  \der_{\Gamma} X$ for some  $X \in C^0(\Gamma;\mathfrak{g})$.  For any $\beta\in Z^1_{\mathrm{par}} (\Gamma, \{\langle z_i \rangle\};\mathfrak{g})$, we have
\[
 \langle\alpha \smile \beta, [\Sigma]\rangle - \sum_{i=1} ^b \Tr X_i \beta(z_i)= 0.
\]
Likewise, for any $\alpha\in Z^1_{\mathrm{par}} (\Gamma, \{\langle z_i \rangle\};\mathfrak{g})$, and $\beta=  \der_{\Gamma} X$, 
\[
 \langle \alpha  \smile \beta, [\Sigma]\rangle - \sum_{i=1} ^b \Tr X_i \beta (z_i)= 0.
\]
\end{lemma}
\begin{proof}
See Proposition 5.4 of H. Kim \cite{kim1999}. 
\end{proof}
In other words,  $\omega^\Sigma _K$ is well-defined and descends to a pairing on parabolic cohomology groups. 

\begin{lemma}For any $[\alpha],[\beta]\in H^1_{\mathrm{par}}(\Gamma, \{\langle z_i \rangle \} ;\mathfrak{g})$, 
\[
\omega^\Sigma _K ([\alpha],[\beta]) = - \omega^\Sigma _K ([\beta],[\alpha]).
\]
Moreover, if $\omega^\Sigma _K ([\alpha],[\beta])=0$ for all $[\alpha]\in H^1_{\mathrm{par}}(\Gamma, \{\langle z_i \rangle \} ;\mathfrak{g})$,  we have $[\beta]=0$. 
\end{lemma}
\begin{proof}
See Proposition 5.5 of H. Kim \cite{kim1999}.
\end{proof}
Therefore, $\omega^\Sigma _K$ is a nondegenerate 2-form on  $\overline{\Rep} ^{\mathscr{B}} _n (\Gamma)$. 

Then we have to show that $\omega^\Sigma _K$ is closed to conclude that it is indeed a symplectic form. This is not a trivial result and can be proven by various ways. See, for instance, H. Kim \cite{kim1999}, Guruprasad-Huebschmann-Jeffrey-Weinstein \cite{guruprasad1997}, and Karshon \cite{karshon1992}.

The expression of $[\Sigma]$  depends on the choice of the relation $R$.  We may wonder the value of $\omega_K$ changes if we use another relation. The following lemma shows that it is not the case.

\begin{lemma}\label{welldefinedness}
	Let $R' = h Rh^{-1}$ for some $h\in \Gamma$. Let $[\Sigma']$ be the relative fundamental class defined as in  Lemma \ref{funclcpt} with respect to $R'$. Then $\langle \alpha \smile \beta, [\Sigma]\rangle = \langle \alpha \smile \beta, [\Sigma']\rangle$ for all $\alpha, \beta \in H^1_{\mathrm{par}}(\Gamma,\{\langle z_i \rangle \}; \mathfrak{g})$. 
\end{lemma}
\begin{proof}
It is straightforward to obtain
	\[
[\Sigma'] = \sum_{i=1} ^g\left(\,\left.\left[ h \frac{\partial R} {\partial x_i} \right| x_i\right] + \left.\left[ h \frac{\partial R}{\partial y_i} \right| y_i\right]\,\right) + \sum_{j=1} ^ b \left.\left[ h\frac{\partial R}{\partial z_j}\right| z_j\right].
\]
Since $\alpha$ is a cocycle, $\alpha\left(\overline{ h \frac{\partial R} {\partial x_i}}\right) = \alpha\left(\overline{\frac{\partial R}{\partial x_i}} \right) + \overline{\frac{\partial R}{\partial x_i}}\cdot \alpha(h^{-1})$. Here we use the convention $(x\pm y) \acts X = x\acts X \pm y\acts X$ where $x,y\in \Gamma$, $X\in \mathfrak{g}$.  By definition of $\langle \alpha\smile \beta, [\Sigma']\rangle$, 
\begin{multline*}
\langle \alpha\smile \beta, [\Sigma']\rangle  = \langle \alpha \smile \beta, [\Sigma]\rangle \\
 +  \trace \alpha(h^{-1})\beta\left(\sum _{i=1} ^g \left(\frac{\partial R}{\partial x_i} (x_i-1)  + \frac{\partial R}{\partial y_i} (y_i-1)\right) +\sum _{i=1} ^b\frac{\partial R}{\partial z_i} (z_i-1)\right).
\end{multline*}
By Theorem \ref{fox}, we conclude that the second term is $\trace \alpha(h^{-1}) \beta(R-1) = 0$. 
\end{proof}

\subsection{The existence of global Darboux coordinates}\label{aaprinciple}
Now we prove series of results toward the existence of a global Darboux coordinate system. Our main goal of this subsection is Theorem \ref{existenceofdarboux}.

\begin{lemma}[A variation of Theorem 18.12 of da Silva \cite{da2001}]\label{sectionimpliescoordinates}
	Let $f=(f_1, \cdots, f_n):M^{2n}\to B$ be a fiber bundle over a connected open subset $B$ of $\R^n$. Suppose that $M$ is given a symplectic structure $\omega$ such that each fiber is a simply connected Lagrangian submanifold. Suppose moreover that the Hamiltonian vector fields  $\mathbb{X}_{f_1}, \cdots, \mathbb{X}_{f_n}$ are linearly independent at each point in $M$ and complete. Then the following hold:
	\begin{itemize}
		\item $f:M\to B$ becomes an affine bundle over $B$.
		\item If $f:M\to B$ admits a global Lagrangian section, then there is a function $g=(g_1, \cdots, g_n):M \to \R^n$ such that $(f_1, \cdots, f_n, g_1, \cdots, g_n)$ is a global Darboux coordinate system. 
	\end{itemize}
\end{lemma}
\begin{proof}
	Since $\mathbb{X}_{f_1}, \cdots, \mathbb{X}_{f_n}$ are complete, they induce a Hamiltonian $\R^n$-action  on $M$. Observe that $\mathbb{X}_{f_1}, \cdots, \mathbb{X}_{f_n}$ are tangent to each fiber. Since fibers are Lagrangian, the rule $x\mapsto (\mathbb{X}_{f_1}|_x,\cdots, \mathbb{X}_{f_n}|_x)$ defines a completely integrable distribution. Since $\mathbb{X}_{f_1}, \cdots, \mathbb{X}_{f_n}$ are linearly independent, the integral manifold is $n$-dimensional, and thus, is an open subset of each fiber. Since each fiber is connected, the integral manifold must be the whole fiber. We know that $\omega([\mathbb{X}_{f_i},\mathbb{X}_{f_j}],Z) = Z \omega(\mathbb{X}_{f_i},\mathbb{X}_{f_j})$ for any vector field $Z$. Since $\omega(\mathbb{X}_{f_i},\mathbb{X}_{f_j})=0$, we have  $[\mathbb{X}_{f_i},\mathbb{X}_{f_j}]=0$ for all $i,j$. This, together with the fact that $\mathbb{X}_{f_i}$'s are complete,  yields that the Hamiltonian flows associated to $f_1,\cdots, f_n$ induce an $\R^n$-action.  Since the integral manifold is the whole fiber, the action must be fiberwise transitive. Since each fiber is simply connected and since the $\R^n$-action is transitive, we conclude that the action is a free action and this gives an affine bundle structure on $M\to B$.  
	
	Let $\sigma:B \to M$ be a Lagrangian section.   Since the action is free and fiberwise transitive, for each $x\in M$, there is a unique $\mathbf{t}_x\in \R^n$  such that $x=\mathbf{t}_x \cdot \sigma(f(x))$. Define the smooth function $g=(g_1, \cdots, g_n):M\to \R^n$ by $g(x) = \mathbf{t}_x$. 
	
	We claim that $x\mapsto (f(x), g(x))$ is a global Darboux coordinate system. We first observe that $(f,g)$ is regular and one-to-one. Hence $(f,g)$ is a global coordinate system. Moreover  $\frac{\partial}{\partial f_i}=\der \sigma\frac{\partial}{\partial x_i}$ spans a Lagrangian subspace at each point $x$ of $\sigma(B)$. Then we compute 
	\begin{align*}
	\omega_x \left(\frac{\partial}{\partial g_i}|_x, \frac{\partial}{\partial f_j}|_x\right)& =\omega_x \left(\mathbb{X}_{f_i}|_x, (\der \sigma \frac{\partial}{\partial x_j})|_x\right) \\
	&=\der f_i\left(  \der \sigma\frac{\partial}{\partial x_j}\right)\\
	&= \der (f_i \circ \sigma) \frac{\partial}{\partial x_j} \\
	&= \frac{\partial x_i}{\partial x_j}= \begin{cases} 1 & i=j \\ 0 & i\ne j\end{cases}
	\end{align*}
	at each $x\in\sigma(B)$. Now consider a general point  $x\in M$. We may assume that $x$ can be reached from $\sigma(f(x))$ by the Hamiltonian flow  $\Psi$ associated to some $f_i$. That is $x= \Psi^t (\sigma(f(x)))$ for some $t\in \R$. Since $\Psi^t$ preserves $\omega$, we have  
	\begin{align*}
	\omega_x\left(\frac{\partial}{\partial g_i}|_x, \frac{\partial}{\partial f_j}|_x\right)&=\omega_x \left(\mathbb{X}_{f_i}|_x , \der \Psi^t \left( \frac{\partial}{\partial f_j}|_{\sigma(f(x))}\right)\right) \\
	&= ((\Psi^{-t} )^\ast \omega)_{x}\left(\mathbb{X}_{f_i}|_x, \der \Psi^t \left( \frac{\partial}{\partial f_j}|_{\sigma(f(x))}\right)\right) \\
	&= \omega_{\sigma(f(x))}\left(\mathbb{X}_{f_i}|_{\sigma(f(x))}, {\frac{\partial}{\partial f_j}}|_{\sigma(f(x))}\right)=\begin{cases} 1 & i=j \\ 0 & i\ne j\end{cases}.
	\end{align*}
	Therefore $(f,g)$ is a global Darboux coordinate system. 
 \end{proof}

\begin{remark} Lemma \ref{sectionimpliescoordinates}, looks similar to the well-known action-angle principle  (see for example Theorem 18.12 of da Silva \cite{da2001}). One difference is that, in our result, the given ``integral of motion'' can be taken as action coordinates without any modification. An additional condition, the existence of a Lagrangian section, must be imposed to obtain this stronger conclusion. 
 \end{remark}

\begin{lemma}\label{local}
      Let $(M^{2n},\omega)$ be a symplectic manifold and $f=(f_1,\cdots, f_n):M\to B\subset \R^n$ be as in Lemma \ref{sectionimpliescoordinates}. Then each $c \in B$ has a neighborhood $U$  such that there is  a symplectomorphism $F:f^{-1}(U)\to T^*B$ which is also a morphism of affine bundles.  
\end{lemma}
\begin{proof}
Note that, by the first assertion of Lemma \ref{sectionimpliescoordinates}, $M$ is an affine bundle. 

We first show that there is a neighborhood $U$ of $c$ where a local Lagrangian section $\sigma|_U$ on $U$ exists. For this we first choose a neighborhood  $V_0$ of $c\in B$ where both $M$ and $T^* B$ are trivialized. Let $T: f^{-1}(V_0) \to V_0 \times \R^n$ be a trivialization. 
	
Let $x\in f^{-1}(c)$. Carath\'eodory-Jacobi-Lie theorem states that there is a neighborhood $U_0$ of $x$ and  function $g=(g_1, \cdots, g_n)$ such that $(U_0,(f,g))$ is a local Darboux chart. We may assume that $U_0\subset f^{-1}(V_0)$ and that $T(U_0)  = U \times I$ for some open box $I$ of $\R^n$ and an open neighborhood $U$ of $c$.  So locally, $\omega|_{U_0} = \sum_{i=1} ^n  \der f_i \wedge \der g_i$. Therefore, $f:M\to B$ admits a local Lagrangian section over $U:=f(U_0)\subset V_0$. Let $\sigma$ be this section.

Let $z:U \to z(U)$ be the zero section of $T^* B \to B$. Define a map $F_0 : \sigma(U) \to z(U)$ by $z\circ f$.  Observe that $(T^* B, \omega_{\text{can}})$ is a vector bundle with fiber preserving the Hamiltonian $\R^n$-action acting fiberwise freely, linearly and  transitively.  We also have the fiberwise free, transitive and linear Hamiltonian $\R^n$-action on $M$. So, for each $x\in f^{-1}(c)$, there is a unique $\mathbf{t}_x\in \R^n$ such that $\mathbf{t}_x \cdot \sigma(c) = x$.  Extend $F_0$ to the map $F: (f^{-1}(U) ,\sigma(U))\to (T^* B|_U,z(U))$ by  $F(x)=\mathbf{t}_x \cdot z(f(x))$. Then $F$ is clearly an affine bundle map. Lastly we have to prove that $F$ is symplectomorphic. To this end, observe that $F$ is symplectomorphic at each point of $\sigma(U)$. Let $\Phi^t$ and $\Psi^t$ be Hamiltonian flows on each bundle corresponding to the same 1-dimensional subgroup of $\R^n$. Then  $F\circ \Phi^t = \Psi^t$. Since Hamiltonian flows preserve the symplectic form, we conclude that $F$ is symplectomorphic.
\end{proof}


Assuming further that $B$ has the trivial 2nd cohomology, we can prove the existence of a global Lagrangian section.   The proof is based on sheaf cohomology theory and the idea of Duistermaat \cite{duistermaat1980}.  

\begin{proposition}\label{existsection}
Assume that $B$ is connected and $H^2(B;\R)=0$. Under the assumptions of Lemma \ref{local}, $f:M\to B$ admits a global Lagrangian section.
\end{proposition}
\begin{proof}
For each $y\in B$,  vector fields $\mathbb{X}_{f_1},\cdots,\mathbb{X}_{f_n}$ are tangent to the fiber $M_y:=f^{-1}(y)$. We write $\mathbb{X}_{f_i}(M_y)$ for the vector field on $M_y$ induced by $\mathbb{X}_{f_i}$. Let $N$ be a vector bundle over $B$ whose fiber over $y$ is the $\R$-vector space spanned by $\mathbb{X}_{f_1}(M_y),\cdots, \mathbb{X}_{f_n}(M_y)$. Since $\der f_i$ annihilates $T _x M_y\subset T_x M$, $x\in M_y$, there is a unique \emph{closed} 1-form $\eta_i$ on $B$ such that $f^* \eta_i= \der f_i$. The assignment $\eta: \mathbb{X}_{f_i} \mapsto \eta_i$  induces the isomorphism of vector bundles $\eta : N \to T^* B$ in the obvious way.

Note that under the assumptions of Lemma \ref{local}, $M$ has the structure of affine bundle over $B$ and the vector bundle $N$ acts on $M$ by fiberwise translation. 

\begin{claim}
	Let $\sigma_1$ be a local Lagrangian section of the affine bundle $M\to B$ over an  open set $U\subset B$. 
	\begin{itemize}
		\item If $\sigma_2$ is another local Lagrangian section over $U$, then $\sigma_1-\sigma_2$ naturally defines a local section of $N\to B$. Moreover, $\eta(\sigma_1 - \sigma_2 ) $ is a closed 1-form on $U$.
		\item Conversely, let $\gamma$ be a local section of $N\to B$ over $U$ such that $\eta(\gamma)$ is a closed 1-form. Then $\sigma_1 + \gamma$ is another Lagrangian local section of $M\to B$ on $U$. 
	\end{itemize} 
\end{claim}
\begin{proof}[Proof of the Claim]
It is clear that $\sigma_1-\sigma_2$ is naturally a section of $N$ since for each $y\in B$,  there is a unique translation vector from $\sigma_1(y)$ to $\sigma_2(y)$. 

Let $y\in B$. Lemma \ref{local} guarantees that we can find a neighborhood $V$ of $y$ such that there is a symplectomorphism $F: f^{-1}(V)\to T^* B$ sending $\sigma_1(V)$ to the zero section.  By construction of $F$, we have $F(\sigma_1-\sigma_2) = \eta (\sigma_1 - \sigma_2)$. Since $F$ is a symplectomorphism $F(\sigma_1)$ and $F(\sigma_2)$ are both closed 1-forms. Since $F$ is an affine bundle morphism, we observe that $\eta(\sigma_1 -\sigma_2) = F(\sigma_1) - F(\sigma_2)$ is also a closed 1-form. Therefore, each point $y\in U$ has a neighborhood where $\eta(\sigma_1-\sigma_2)$ is closed,  which proves the first part of the claim.

If $\gamma$ is a local section of $N$ such that $\eta(\gamma)$ is closed, then we have that $F(\sigma_1 + \gamma)=\eta(\sigma_1+\gamma)$ is a closed 1-form so it is a local Lagrangian section of $T^*B$. Since $F$ is symplectomorphism, $\sigma_1+\gamma$ must be a local Lagrangian section. 
\end{proof}

We can cover $B$ by open sets $\{W_i\}$ such that the affine bundle $M\to B$ is trivial over $W_i$ for each $i$ and that there is a local Lagrangian section $\sigma_i$ on each $W_i$. We can further assume that each finite intersection of $\{W_i\}$ is contractible. Observe, by the above claim, that the difference $\sigma_{ij} := \sigma_i - \sigma_j$ of the sections on $W_i\cap W_j$ gives a (\v Cech) 1-cocycle $\{\eta (\sigma_{ij})\}$ of the sheaf $\operatorname{Ker} \der^1$. By the above claim again, the cohomology class in $H^1(B, \operatorname{Ker} \der^1)$ represented by  $\{\eta (\sigma_{ij})\}$ is the obstruction of finding a global Lagrangian section. We show that this obstruction class vanishes. 

Consider an exact sequence of sheaves
\[
0\to \underline{\R} \to \Omega_{B} ^0 \to \operatorname{Im} \der^0 \to 0.
\] 
Here, $\underline{\R}$ denotes the constant sheaf and $\Omega_B ^0$ is the sheaf of smooth functions on $B$. The above exact sequence induces the long exact sequence
\[
\cdots \to H^1(B,\Omega^0 _{B}) \to H^1(B , \operatorname{Im} \der^0) \to H^2 (B, \underline{\R}) \to H^2(B,\Omega^0 _{B}) \to \cdots.
\]
Observe that $\operatorname{Ker} \der^1=\operatorname{Im} \der^0$ as sheaves.  Moreover because $\Omega^0 _{B}$ is a soft sheaf, it follows that
\[
H^1(B,\Omega^0 _{B})=H^2(B,\Omega^0 _{B})=0.
\]
 Therefore 
\[
H^1(B , \operatorname{Ker} \der^1) \approx H^1(B , \operatorname{Im} \der^0) \approx H^2 (B, \underline{\R})\approx H^2(B;\R) =0.
\]
Consequently, $\{\eta (\sigma_{ij})\}$ represents the trivial class so Proposition \ref{existsection} follows. 
\end{proof}

Finally,  by putting all the above results together, one can deduce the following theorem.

\begin{theorem}[A variation of Duistermaat \cite{duistermaat1980}]\label{existenceofdarboux}
Let $(M^{2n},\omega)$ be a symplectic manifold and $f=(f_1, \cdots, f_n):M^{2n}\to B$ be a fiber bundle over a connected open subset $B$ of $\R^n$. Suppose:
\begin{itemize}
\item $ H^2(B;\R)=0$,
\item  each fiber is a simply connected Lagragian submanifold, and
\item  the Hamiltonian vector fields  $\mathbb{X}_{f_1}, \cdots, \mathbb{X}_{f_n}$ are linearly independent at each point in $M$ and complete. 
\end{itemize}
Then there is a function $g=(g_1, \cdots, g_n):M \to \R^n$ such that 
\[
(f_1, \cdots, f_n, g_1, \cdots, g_n)
\] 
is a global Darboux coordinate system. \end{theorem}

\section{Decomposition formulas}

This section is devoted to the proof of our first main result, Theorem \ref{decompthmintro}. As mentioned in the introduction, we do induction on the number of curves. We deal with the base cases by using  the Fox calculus and cocycle computations. Induction process is somewhat technical particularly when we try to cut the surface by more than three curves. Suppose for instance that three curves $\xi_1$, $\xi_2$ and $\xi_3$ are positioned as in Figure \ref{f1}. Then $\xi_1$, $\xi_2$ and $\xi_3$ are all non-separating in $\Sigma$. However $\xi_3$ becomes separating seen as a curve in $\Sigma\setminus( \xi_1\cup \xi_2)$. On the other hand, $\xi_2$ becomes separating if we subtract $\xi_1$ and $\xi_3$. Therefore we get at least three different decompositions of $\pi_1(\Sigma)$ depending on the order of cutting. To treat this technicality systematically, we borrow the notation of graph of groups.

We then present a (parabolic) group cohomology version of the Mayer-Vietoris sequence to prove the decomposition formula for a single cutting.  This sequence decomposes the tangent space into one or two components and our formulas show that the pairing $\omega^\Sigma _K$ is additive with respect to this decomposition.  

\begin{figure}[ht]
\begin{tikzpicture}
\draw(0,0) node {\includegraphics[scale=0.3]{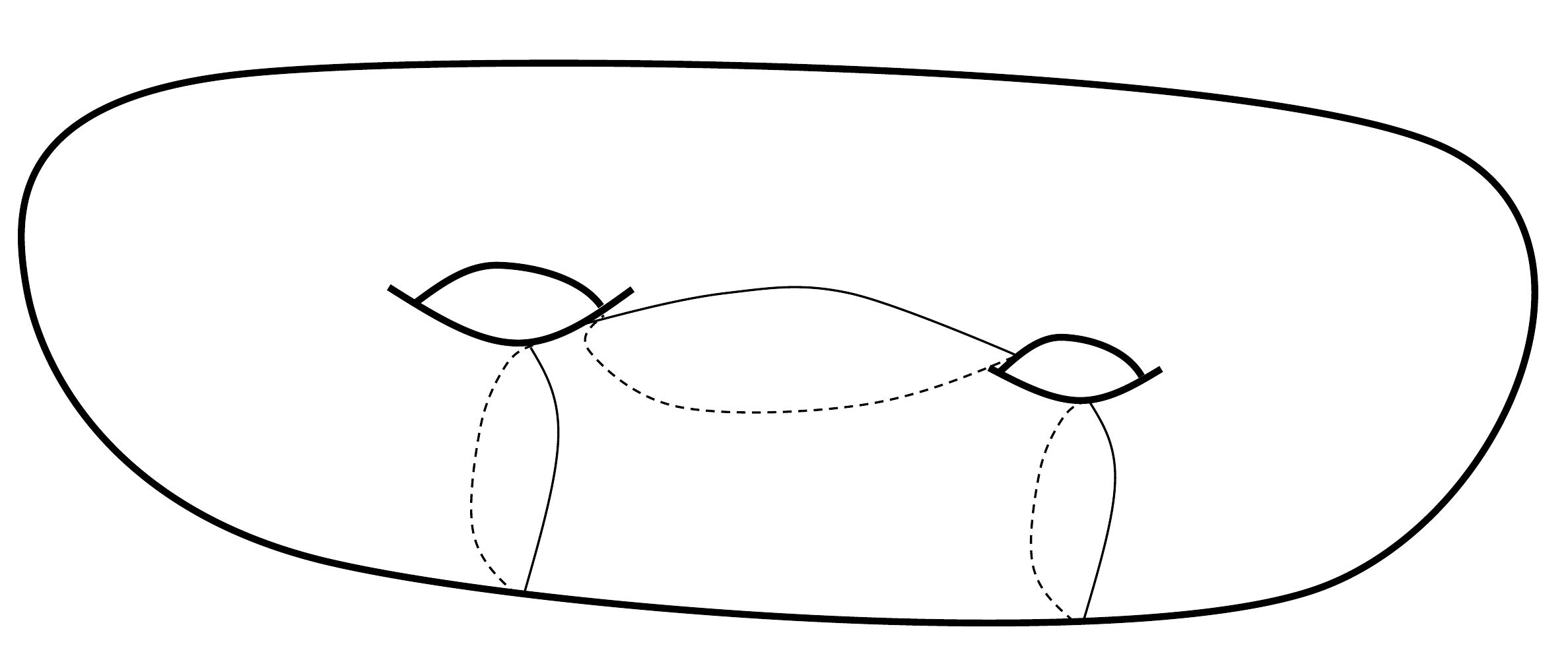}};
\draw (0,.4cm) node {$\xi_2$};
\draw (-1.2cm, -.7cm) node {$\xi_1$};
\draw (1.7cm, -.7cm) node {$\xi_3$};
\end{tikzpicture}
\caption{}\label{f1}
\end{figure}

\subsection{Decomposition of fundamental groups}\label{tree}
 Let $\Sigma$ be a compact oriented hyperbolic surface of genus $g$ with boundary components $\zeta_1, \cdots, \zeta_b$. We denote by $\Gamma$ its fundamental group $\pi_1(\Sigma)$. Let $\{\xi_1, \cdots, \xi_m\}$ be a collection of pairwise disjoint, non-isotopic essential simple closed curves in $\Sigma$ that divide the surface into subsurfaces $\Sigma_1, \cdots, \Sigma_l$. We assume that each $\Sigma_i$ is hyperbolic.  
 

Following Johnson-Millson \cite{Johnson1987}, we can construct a tree $\mathcal{T}$  as follow. Let $p:\widetilde{\Sigma}\to \Sigma$ be the universal cover. The set of vertices $V(\mathcal{T})$ consists of  connected components of $\widetilde{\Sigma} \setminus \bigcup_{i=1} ^ m p^{-1}(\xi_i)$. Two vertices are joined by an edge in $E(\mathcal{T})$ if they are adjacent along some component of $p^{-1}(\xi_i)$. Observe that each vertex corresponds to the universal cover of some $\Sigma_i$.  Johnson-Millson show in \cite[Lemma 5.3]{Johnson1987}  that $\mathcal{T}$ is indeed a tree and admits a $\Gamma$-action without inversion. Hence, we have the following theorem. 

\begin{theorem}[Johnson-Millson \cite{Johnson1987}, see also Serre \cite{serre2003}]\label{decomposerep}
	Let $\Sigma$ be a compact oriented hyperbolic surface, $\{\xi_1, \cdots, \xi_m\}$   a collection of pairwise disjoint, non-isotopic essential simple closed curves in $\Sigma$ that divides the surface into hyperbolic subsurfaces $\Sigma_1, \cdots, \Sigma_l$. Then  $\Gamma:=\pi_1(\Sigma)$ is isomorphic to the fundamental group $\pi_1(\Gamma, \mathcal{G}, \mathcal{D})$ of a graph of groups $(\Gamma,\mathcal{G})$, $\mathcal{G} = \mathcal{T}/\Gamma$ where $\mathcal{D}$ is a choice of a maximal tree of $\mathcal{G}$.  We can label vertices of  $\mathcal{G}$ by $\Sigma_1, \cdots, \Sigma_l$  and edges by $\xi_1, \cdots, \xi_m$. Choose  a lift $j$ of $\mathcal{D}$ in $\mathcal{T}$. The vertex group at $\Sigma_i$ is $\Gamma_{\Sigma_i}=\operatorname{Stab}_{\Gamma} (j(\Sigma_i))$ and the edge group at $\xi_i$ is $\Gamma_{\xi_i} = \operatorname{Stab}_{\Gamma}(j(\xi_i))$. 
\end{theorem}

Observe that $\Gamma_{\Sigma_i}$ is conjugate in $\Gamma$ to $\pi_1(\Sigma_i)$ and that $\Gamma_{\xi_i}$ is conjugate to $\pi_1(\xi_i)$. 

Let us choose a base vertex of $\mathcal{D}$ and define a relation $\le$ on $V(\mathcal{D})$ by declaring that $\Sigma_i \le {\Sigma_j}$ if and only if ${\Sigma_i}$ is nearer to the base vertex  than ${\Sigma_j}$. It is clear that $\le$ is a partial order and the set $V(\mathcal{D})$ becomes a poset. For each vertex $\Sigma_i$ we define the following subset   
\[
S({\Sigma_i}):=\{{\Sigma_j} \in V(\mathcal{D})\,|\,{\Sigma_i}\le {\Sigma_j}\}.
\] 

For an edge ${\xi_i}$ and $\gamma \in \Gamma_{\xi_i}$, denote by $\gamma ^-$ the image of $\gamma$ in the vertex group of  the origin  $\operatorname{o}({\xi_i})$. Similarly $\gamma^+$ is the image of $\gamma$ in the vertex group of the terminal $\operatorname{t}({\xi_i})$. Therefore, for each $\xi_i \in  E(\mathcal{G})$, 
\begin{align*}
\Gamma_{\xi_i} ^{+}&:= \{\gamma^{+} \,|\, \gamma\in \Gamma_{\xi_i}\}, \quad \text{ and}\\
\Gamma_{\xi_i} ^{-}&:= \{\gamma^{-} \,|\, \gamma\in \Gamma_{\xi_i}\}
\end{align*}
are subgroups of $\Gamma=\pi_1(\Gamma, \mathcal{G}, \mathcal{D})$. 

For each $\xi_i \in E(\mathcal{G})$ and each $\gamma \in \Gamma_{\xi_i}$,  we have $\gamma^+ = \gamma^-$  in the whole group $\Gamma$. If $\xi_i$ is an edge of $\mathcal{G}$ but  $\xi_i\notin E(\mathcal{D})$, then we have an additional  generator $\xi_i^{\perp}$ with relation $\xi_i^{\perp} \gamma^+ (\xi_i^{\perp})^{-1} = \gamma^-$ for each $\gamma \in \Gamma_{\xi_i}$ in $\Gamma$. Note that $\xi_i^{\perp}$  corresponds to a loop  transverse to $\xi_i$.

Let $\rho:\Gamma \to G$  be a representation. Since each vertex group $\Gamma_{\Sigma_i}$ injects into $\Gamma$, $\rho$ induces a representation $\rho_{\Gamma_{\Sigma_i}} : \Gamma_{\Sigma_i} \to G$ for each vertex group.  $\rho$  also induces a representation $\rho_{\xi_i^{\perp}}:\langle \xi_i^{\perp} \rangle \to G$ for each edge $\xi_i$ which is not in $E(\mathcal{D})$.  In this way, we obtain a collection of representations $\rho_{\Gamma_{\Sigma_i}}: \Gamma_{\Sigma_i } \to G$ for each $i=1,2,\cdots,l$ and  $\rho_{\xi_i^{\perp}}:\langle \xi_i^{\perp}\rangle \to G$ for each $\xi_i\in E(\mathcal{G})\setminus E(\mathcal{D})$.

Conversely, suppose that we are given a  collection of representations $\rho_{\Gamma_{\Sigma_i}}: \Gamma_{\Sigma_i } \to G$ for each $i=1,2,\cdots,l$ and  $\rho_{\xi_j^{\perp}}:\langle \xi_j^{\perp}\rangle \to G$ for each $\xi_j\in E(\mathcal{G})\setminus E(\mathcal{D})$, subject to relations 
\begin{itemize}
\item  If ${\xi_k} \in E(\mathcal{D})$ and if ${\Sigma_i}=\operatorname{o}(\Gamma_{\xi_k})$ and ${\Sigma_j}=\operatorname{t}(\Gamma_{\xi_k})$  then for each $\gamma\in \Gamma_{\xi_k}$, 
\begin{equation}\label{rel1}
\rho_{\Gamma_{\Sigma_i}}(\gamma ^-) = \rho_{\Gamma_{\Sigma_j}}( \gamma ^+). 
\end{equation}
 \item If ${\xi_k}\notin E(\mathcal{D})$ and if ${\Sigma_i}=\operatorname{o}({\xi_k})$ and ${\Sigma_j}=\operatorname{t}({\xi_k})$ then, for each $\gamma \in \Gamma_{\xi_k}$, 
 \begin{equation}\label{rel2}
\rho_{\xi_k^{\perp}} (\xi_k^{\perp}) \rho_{\Gamma_{\Sigma_j}}(\gamma^+) \rho_{\xi_k^{\perp}}(\xi_k^{\perp})^{-1} = \rho_{\Gamma_{\Sigma_i}}(\gamma^-).
\end{equation}
\end{itemize}
Then there is a unique representation $\rho: \Gamma \to G$  whose restrictions are precisely prescribed representations. 
\begin{figure}[ht]\label{f2}
\begin{tikzpicture}
\draw(0,0) node {\includegraphics[scale=0.3]{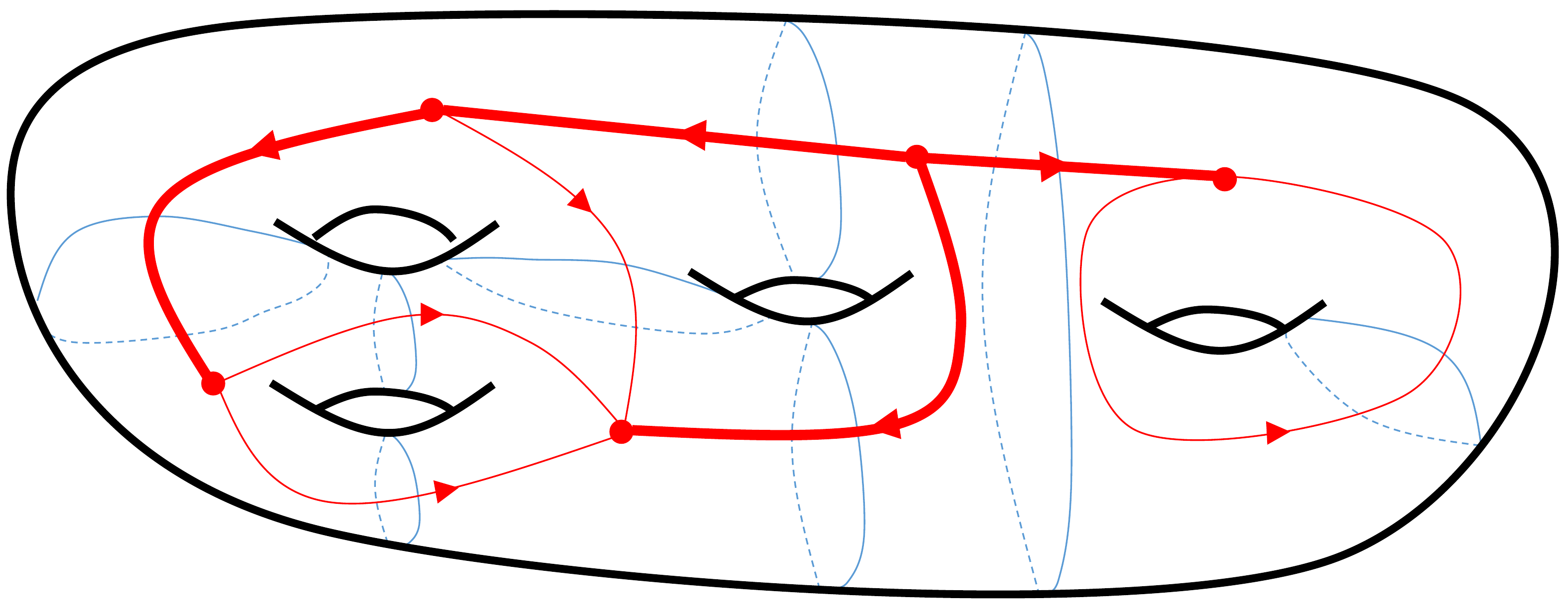}};
\end{tikzpicture}
\caption{An example of decomposition.  Curves in $\mathcal{C}$ are depicted by blue lines. The graph $\mathcal{G}$ is drawn in red with a maximal tree $\mathcal{D}$ bolded.}
\end{figure}

\subsection{The Mayer-Vietoris sequence}\label{MVS}
Before go further, we summarize the general settings that we consider in the subsequent section. 
 \begin{itemize}
 \item $\Sigma$ is a compact oriented hyperbolic surface with boundary components $\{\zeta_1, \cdots, \zeta_b\}$.  $\{\xi_1, \cdots, \xi_m\}$  is a collection of pairwise disjoint, non-isotopic essential simple closed curves in $\Sigma$ that divide the surface into hyperbolic subsurfaces $\Sigma_1, \cdots, \Sigma_l$. We  stick to the notation $\Gamma=\pi_1(\Sigma)$ (which is also isomorphic to $\pi_1(\Gamma, \mathcal{G}, \mathcal{D})$), $\Gamma_{\Sigma_1}, \cdots, \Gamma_{\Sigma_l}$ and $\Gamma_{\xi_1}, \cdots, \Gamma_{\xi_m}$ of the previous subsection.
 
 \item Denote by $\iota_{\Sigma_i}$ the map $\overline{\Sigma_i} \to \Sigma$, the extension of the inclusion $\Sigma_i \to \Sigma$ to the completion $\overline{\Sigma_i}$ of $\Sigma_i$. We sometimes use the same notation $\iota_{\Sigma_i}$ to denote the induced homomorphism  $\iota_{\Gamma_{\Sigma_i}}:\Gamma_{\Sigma_i} \to\Gamma$.
 
 \item Unless otherwise stated, $[\rho]$ denotes an element in $\overline{\Rep}_n ^{\mathscr{B}}(\Gamma, \mathscr{C})$ such that $[\rho_{\Gamma_{\Sigma_i}}]\in \overline{\Rep}_n ^{\mathscr{B}_i} (\Gamma_{\Sigma_i})$ for each $i=1,2,\cdots, l$ where 
 \[
 \mathscr{B}_i=\{(\xi,B)\,|\,\xi\text{ is a component of } \partial \overline{\Sigma_i}\text{ and } (\iota_{\Sigma_i}(\xi),B)\in \mathscr{B}\cup \mathscr{C}\}.
 \]
 \end{itemize}

The Mayer-Vietoris sequence for cohomology of group system is proven in \cite{bieri1978}. Our version of the  Mayer-Vietoris sequence is the following.
\begin{proposition}\label{mvs} Fix a representation $\rho$ in the class $[\rho]$. Let $(\Gamma, \mathcal{S})$ be a group system where $\Gamma=\pi_1(\Sigma)$ and $\mathcal{S} = \{\Gamma_{\xi_1}^+, \cdots, \Gamma_{\xi_m}^+, \langle \zeta_1 \rangle, \cdots, \langle \zeta_b \rangle\}$. Define for each $i=1,2,\cdots, l$, $\mathcal{S}_i =\{ \langle \zeta \rangle \subset \Gamma_{\Sigma_i}\,|\, \zeta \text{ is a component of } \partial \overline{\Sigma_i}  \}$ so that $(\Gamma_{\Sigma_i}, \mathcal{S}_i)$ is a group subsystem of $(\Gamma, \mathcal{S})$. 
	\begin{itemize}
	\item The sequence
	\[
	0\to \bigoplus_{i=1} ^{m} H^0(\Gamma_{\xi_i}; \mathfrak{g}_{\rho_{\Gamma_{\xi_i}^+}})\overset{\delta}{\to} H^1_{\mathrm{par}} (\Gamma,\mathcal{S}; \mathfrak{g}_{\rho})\overset{\iota^*}{\to}  \bigoplus_{i=1} ^l H^1 _{\mathrm{par}} (\Gamma_{\Sigma_i},\mathcal{S}_i ;\mathfrak{g}_{\rho_{\Gamma_{\Sigma_i}}})\to 0
	\]
	is exact. 
	\item The connecting homomorphism $\delta$ sends $X\in H^0(\Gamma_{\xi_i}; \mathfrak{g})$ to the tangent cocycle of an algebraic bending by $X$  along $\xi_i$ and $\iota^*$ is induced from the inclusions $\iota_{{\Sigma_i}}:(\Gamma_{\Sigma_i},\mathcal{S}_i)\to (\Gamma,\mathcal{S})$ that is, 
	\[
	\iota^* ([\alpha]) = \iota_{\Sigma_1} ^* [\alpha] \oplus \cdots \oplus \iota^* _{\Sigma_l}[\alpha].
	\]
	\end{itemize}
\end{proposition}


We do not prove the first statement at this moment because their proof has no dependency on remaining parts of our paper. For the sake of completeness however we give a proof in the appendix.  

The second assertion regarding the map $\delta$ is shown in \cite[Lemma 5.8]{Johnson1987}. Because we do need some details about the connecting homomorphism, we give more descriptions here.

Choose a representation $\rho$ in the class $[\rho]$.  Let $\xi_0$ be an edge of $\mathcal{G}$. Let   $X\in H^0(\Gamma_{\xi_{i_0}}; \mathfrak{g}_{\rho_{\Gamma_{\xi_0}^+}})$ where $H^0 (\Gamma_{\xi_{i_0}};  \mathfrak{g}_{\rho_{\Gamma_{\xi_0}^+}}) = \ker(\Ad_{\rho_{\Gamma_{\xi_0}^+}} - \operatorname{Id})\subset \mathfrak{g}$.  We introduce a flow $\Phi^t _{X, \xi_{i_0}}$ in $\Hit_n(\Sigma)$ as follow.  If $\xi_{i_0}$ is an edge in $E(\mathcal{D})$ joining $\Sigma_p$ and $\Sigma_q$ with $\Sigma_p<\Sigma_q$, define
\[
\Phi^t _{X, \xi_{i_0}} (\rho) (x) = \begin{cases}
\rho(x) & x\in \Gamma_{\Sigma_j} \text{, } \Sigma_j \le \Sigma_p\text{ or incomparable}\\
(\exp tX) \rho(x) (\exp - tX)& x\in \Gamma_{\Sigma_j} \text{, }\Sigma_j \ge \Sigma_q\\
(\exp tX) \rho(x) (\exp - tX) & x=\xi_k^{\perp}\text{, } \operatorname{o}({\xi_k}), \operatorname{t}({\xi_k})\in S(\Sigma_q)\\
(\exp  tX) \rho(x) & x=\xi_k^{\perp} \text{, } \operatorname{o}(\xi_k )\in S(\Sigma_q)\\
 \rho(x)(\exp  - tX) & x=\xi_k^{\perp} \text{, } \operatorname{t}(\xi_k )\in S(\Sigma_q)\\
\end{cases}.
\]
For each $t$, $\Phi^t _{X, \xi_{i_0}}(\rho)$ satisfies all relations in (\ref{rel1}) and (\ref{rel2}).

If $\xi_{i_0}$ is not in $E(\mathcal{D})$, we define
\[
\Phi^t _{X, \xi_{i_0}} (\rho) (x) = \begin{cases}
 \rho(x) (\exp tX) & x=\xi_{i_0}^{\perp} \\ 
\rho(x) & \text{otherwise}
\end{cases}.
\]
Again $\Phi^t _{X,\xi_{i_0}} (\rho)$ fulfills the relations in (\ref{rel1}) and (\ref{rel2}) for each $t$.  Therefore, in both cases, we get the flow of representations $\Phi^t _{X, \xi_{i_0}}$. Call this flow the \emph{algebraic bending} by $X$ along $\xi_{i_0}$. The last assertion of Proposition \ref{mvs} states that $\frac{\partial}{\partial t} \Phi^t_{X,\xi_{i_0}}  = \delta(0,0,\cdots,0,X,0,\cdots,0)$ where $X$ is in the $H^0 (\Gamma_{\xi_{i_0}};\mathfrak{g})$ component of $\bigoplus_{i=1} ^m H^0(\Gamma_{\xi_i}; \mathfrak{g})$.

\subsection{Local decomposition formula: separating case}
Assume $\xi$ is separating so that  $\Sigma \setminus \xi$ has two components $\Sigma_1, \Sigma_2$ each of which is hyperbolic. In view of Theorem \ref{decomposerep}, $\Gamma$ is the fundamental group of \raisebox{-2ex}{\begin{tikzpicture}
\draw (0,0) node {$\bullet$} ;
\draw (0,0) -- (1,0) node {$\bullet$};
\draw[->, thick] (.5,0)--(.51,0);
\draw (0,0) node[above] {$\Sigma_1$};
\draw (1,0) node[above] {$\Sigma_2$};
\draw (.5,0) node[above]{$\xi$};
\end{tikzpicture}}.  

From Proposition \ref{mvs}, we have the short exact sequence
\begin{multline*}
0\to H^0(\Gamma_\xi; \mathfrak{g}_{\rho_{\Gamma_\xi ^+}}) \overset{\delta}{\to}  H^1 _{\mathrm{par}} (\Gamma,\mathcal{S} ; \mathfrak{g}_\rho)  \\  \overset{(\iota_{\Sigma_1} ^*,\iota_{\Sigma_2} ^* )}{\to} H^1_{\mathrm{par}} (\Gamma_{\Sigma_1},\mathcal{S}_1; \mathfrak{g}_{\rho_{\Gamma_{\Sigma_1}}}) \oplus H^1 _{\mathrm{par}}(\Gamma_{\Sigma_2},\mathcal{S}_2; \mathfrak{g}_{\rho_{\Gamma_{\Sigma_2}}})\to 0
\end{multline*}
where, as before, $\mathcal{S}= \{\Gamma_\xi ^+, \langle \zeta_1\rangle, \cdots, \langle \zeta_b\rangle \}$ is a collection of subgroups of $\Gamma$ that forms a group system $(\Gamma, \mathcal{S})$ and $\mathcal{S}_i = \{\langle \zeta \rangle \subset \Gamma_{\Sigma_i} \,|\, \zeta \text{ is a component of } \partial \overline{\Sigma_i}\}$. We also abbreviate the inclusion $\iota_{\Gamma_{\Sigma_i}}:\Gamma_{\Sigma_i} \to \Gamma$ to $\iota_{\Sigma_i}$. 



Now we can state  the decomposition formula.

\begin{theorem}\label{decomppairingsep} Let $\Sigma$ be a compact oriented hyperbolic surface possibly with boundary components $\{\zeta_1, \cdots, \zeta_b\}$ and $\xi$ an essnetial simple closed curve that separates $\Sigma$ into two hyperbolic subsurfaces $\Sigma_1$ and $\Sigma_2$. Let $(\Gamma, \mathcal{S})$ be the group system where $\Gamma=\pi_1(\Sigma)$ and $\mathcal{S}= \{\Gamma_\xi^+, \langle \zeta_1\rangle, \cdots, \langle \zeta_b\rangle \}$. Choose a boundary frame $\mathscr{B}$ and $\{\xi\}$-frame $\mathscr{C}$.  Let $[\rho] \in \overline{\Rep}_n ^{\mathscr{B}}(\Gamma, \mathscr{C})$ be such that $[\rho_{\Gamma_{\Sigma_i}}] \in \overline{\Rep}_n ^{\mathscr{B}_i}(\Gamma_{\Sigma_i})$ where 
\[
\mathscr{B}_i=\{ (\zeta, B)\,|\, \zeta\text{ is a component of }\partial \overline{\Sigma_i}\text{ and }(\iota_{\Sigma_i}(\zeta), B) \in \mathscr{B}\cup \mathscr{C}\}
\]
for each $i=1,2$. Fix a representation $\rho$ in the class $[\rho]$. For $[\alpha],[\beta]\in H^1_{\mathrm{par}}(\Gamma, \mathcal{S};\mathfrak{g}_\rho)$, we have
	\[
	\omega_K^{\Sigma}([\alpha], [\beta]) = \omega^{\Sigma_1} _K (\iota^* _{\Sigma_1}[\alpha],\iota^* _{\Sigma_1}[\beta])+\omega^{\Sigma_2} _K(\iota^* _{\Sigma_2}[\alpha],\iota^* _{\Sigma_2}[\beta]).
	\] 
\end{theorem}

We prove the following lemma first.

\begin{lemma}\label{normalize}
	Let $\mathcal{S}_i = \{\langle \zeta \rangle \,|\, \zeta \text{ is a component of } \partial \overline{\Sigma_i}\}$, $i=1,2$. If $\iota^*_{\Sigma_2} [\alpha]=0$  then  there is a unique 1-cocycle $\widetilde{\alpha_1}\in Z^1_{\mathrm{par}}(\Gamma,\mathcal{S}; \mathfrak{g})$ such that $[\widetilde{\alpha_1}]=[\alpha]$ in $H^1_{\mathrm{par}}(\Gamma,\mathcal{S} ; \mathfrak{g})$ and that $\iota_{\Sigma_2} ^\# (\widetilde{\alpha_1} )=0$ in $Z^1_{\mathrm{par}} (\Gamma_{\Sigma_2}, \mathcal{S}_2;\mathfrak{g})$. Similarly, if $\iota^* _{\Sigma_1} [\alpha] =0$ then  there is a unique 1-cocycle $\widetilde{\alpha_2}$ such that $[\widetilde{\alpha_2}]=[\alpha]$ and that $\iota_{\Sigma_1}^\# (\widetilde{\alpha_2}) =0$ in $Z^1_{\mathrm{par}} (\Gamma_{\Sigma_1}, \mathcal{S}_1;\mathfrak{g})$. 
\end{lemma}

\begin{proof}
	We prove the first case. Pick any representative of $[\alpha]$ say $\alpha_1'\in Z^1_{\mathrm{par}}(\Gamma,\mathcal{S};\mathfrak{g})$. Since $\iota^*_{\Sigma_2} [\alpha]=\iota^*_{\Sigma_2} [\alpha_1 ']=0$,  there is $X\in \mathfrak{g}$ such that $\iota_{\Sigma_2} ^\# (\alpha_1') = \der_{\Gamma_{\Sigma_2}} X$. Let $\widetilde{\alpha_1} = \alpha_1 ' -\der_{\Gamma} X$. Then clearly $[\widetilde{\alpha_1}]$ satisfies the required properties. For the uniqueness, suppose that there is another class $\alpha_1 ''\in Z^1_{\mathrm{par}}(\Gamma,\mathcal{S}; \mathfrak{g})$ satisfying the same properties. Then since $[\alpha_1'']=[\widetilde{\alpha_1}]$ in $H^1_{\mathrm{par}} (\Gamma,\mathcal{S}; \mathfrak{g})$, we have  $\alpha_1'' = \widetilde{\alpha_1} +\der_\Gamma Y$ for some $Y\in \mathfrak{g}$. Applying $\iota_{\Sigma_2}^\#$, we have  $0=\iota^\# _{\Sigma_2}\der_\Gamma Y=\der_{\Gamma_{\Sigma_2}} Y$ in $Z^1_{\mathrm{par}}(\Gamma_{\Sigma_2},\mathcal{S}_2;\mathfrak{g})$. By Lemma \ref{NoInvariantElement}, $Y=0$ and the uniqueness follows. 
	
	The second case can be achieved along the same lines. 
\end{proof}

\begin{proof}[Proof of  Theorem \ref{decomppairingsep}] We borrow the idea of Zocca \cite{zocca1998}. We use the following presentations
	\begin{align*}
	{\Gamma}_{\Sigma_1}&=\langle {x}_{1,1}, {y} _{1,1},\cdots, x_{1,g_1}, y_{1,g_1} , {z}_{1,0}, \cdots, z_{1,b_1} \,|\,{\mathbf{r}}_1\rangle,\\
	{\Gamma}_{\Sigma_2}&= \langle {x}_{2,1}, {y} _{2,1} ,\cdots, x_{2,g_2}, y_{2,g_2}, {z}_{2,0}, \cdots, z_{2,b_2} \,|\,{\mathbf{r}}_2\rangle,\\
	\Gamma_{\xi} &= \langle \xi \rangle\text{ with }\xi^- =z_{1,0},\, \xi^+ = z_{2,0}^{-1},\qquad \text{and}\\ 
	\Gamma&=\pi_1(\Sigma)=\langle{x}_{1,j}, {y} _{1,j} , {z}_{1,j} ,{x}_{2,j}, {y} _{2,j} , {z}_{2,j} \,|\,{\mathbf{r}} \rangle
	\end{align*}
	 where  relations are given by
	\begin{align*}
	{\mathbf{r}}_1 &= (\prod_{j=1} ^{g_1}[{x} _{1,j}, {y}_{1,j}])(\prod_{j=1} ^{b_1} {z}_{1,j}){z}_{1,0},\\
	{\mathbf{r}}_2 &= {z}_{2,0} (\prod_{j=1} ^{b_2} {z}_{1,j})(\prod_{j=1} ^{g_2}[{x} _{2,j}, {y}_{2,j}]),\quad\text{and}\\
	{\mathbf{r}}&=(\prod_{j=1} ^{g_1}[{x} _{1,j}, {y}_{1,j}])(\prod_{j=1} ^{b_1} {z}_{1,j}) (\prod_{j=1} ^{b_2} {z}_{2,j})(\prod_{j=1} ^{g_2}[{x} _{2,j}, {y}_{2,j}]).
	\end{align*}

We can decompose the relative fundamental class (Lemma \ref{funclcpt}) as
	\[
	[\Sigma] = [\Sigma_1]+ [\Sigma_2]- [{z}_{1,0} ^{-1}| {z}_{1,0}].
	\]
	
	We first assume $\iota^* _{\Sigma_2} [\alpha_1] =0$ and  compute $\omega_K ^\Sigma([\alpha_1], [\beta])$.  By Lemma \ref{normalize}, we can find a representative $\widetilde{\alpha_1}$ of $[\alpha_1]$ such that $\iota_{\Sigma_2} ^\# \widetilde{\alpha_1}=0$. We use $\widetilde{\alpha_1}$ to compute $\omega_K ^\Sigma$ as follow
	\begin{align*}
	\omega_K ^\Sigma([\alpha_1],[\beta])&=\omega_K ^{\Sigma}([\widetilde{\alpha_1}],[\beta]) \\
	&=\langle \widetilde{\alpha_1} \smile \beta, [\Sigma_1] \rangle+\langle \widetilde{\alpha_1} \smile \beta, [\Sigma_2] \rangle+\langle \widetilde{\alpha_1} \smile \beta, [z_{1,0}^{-1}|z_{1,0}] \rangle \\
	&\qquad  -\sum_{i=1} ^{b_1}  \Tr X_{1,i} \beta(z_{1,i})-\sum_{i=1} ^{b_2}  \Tr X_{2,i} \beta(z_{2,i})
\end{align*}
	where $X_{i,j} \in \mathfrak{g}$ is such that $\der_{\langle z_{i,j} \rangle} X_{i,j} = \iota_{\langle z_{i,j} \rangle} ^\# \widetilde{\alpha_1}$. By the construction of $\widetilde{\alpha_1}$, we have $\langle \widetilde{\alpha_1} \smile \beta, [\Sigma_2]\rangle =0$ and  $\widetilde{\alpha_1}({z}_{i,0})=\widetilde{\alpha_1}({z}_{i,0}^{-1})= 0$, for $i=1,2$. Again, by the construction of $\widetilde{\alpha_1}$,  we can choose $X_{2,i}$ to be $0$ for all $i=1,2,\cdots, b_2$. Hence
\[
\omega_K ^\Sigma([\alpha_1],[\beta])=\omega_K ^{\Sigma_1}(\iota^*_{\Sigma_1}[\alpha_1],\iota_{\Sigma_1} ^*[\beta])+\Tr X \beta({z}_{1,0})
\]
where $X\in \mathfrak{g}$ is given by the property $\der_{\langle z_{1,0}\rangle} X =  \iota_{\langle z_{1,0} \rangle} ^\# \widetilde{\alpha_1}$. Observe that  $\iota_{\langle z_{1,0}\rangle} ^\# \widetilde{\alpha_1}=\iota_{\langle z_{2,0}\rangle} ^\# \widetilde{\alpha_1} =0$. Therefore we can  take $X$ to be zero. This leads us to 
	\[
	\omega_K ^\Sigma([\alpha_1],[\beta])=\omega_K ^{\Sigma_1}(\iota^*_{\Sigma_1}[\alpha_1],\iota_{\Sigma_1} ^*[\beta]).
	\]
	
	Similarly, if $[\alpha_2] \in H^1_{\mathrm{par}} (\Gamma_{\Sigma_2},\mathcal{S}_2;\mathfrak{g})$ is such that $\iota^*_{\Sigma_1} [\alpha_2]=0$, we choose $\widetilde{\alpha_2}$ as before to get
	\[
	\omega_K ^\Sigma([\alpha_2],[\beta]) =\omega_K ^{\Sigma_2}(\iota^*_{\Sigma_2}[\alpha_2],\iota_{\Sigma_2} ^* [\beta]).
	\]

Now suppose that we are given  a general $[\alpha]\in H^1_{\mathrm{par}} (\Gamma, \mathcal{S};\mathfrak{g})$. Since $H^1_{\mathrm{par}} (\Gamma, \mathcal{S};\mathfrak{g}) = \ker \iota^* _{\Sigma_1} + \ker \iota ^* _{\Sigma_2}$ (not direct), we can decompose $[\alpha]$ as a sum  $[\alpha]=[\alpha_1]+[\alpha_2]$ where $[\alpha_1]\in \ker \iota^* _{\Sigma_2}$ and $[\alpha_2]\in \ker \iota^* _{\Sigma_1}$. Then by linearity we have 
	\begin{align*}
	\omega_K ^\Sigma ([\alpha],[\beta]) & =\omega_K ^{\Sigma}([\alpha_1],[\beta])+\omega_K ^{\Sigma}([\alpha_2], [\beta])\\
	&=\omega_K ^{\Sigma_1}(\iota_{\Sigma_1} ^*[\alpha_1],\iota_{\Sigma_1} ^*[\beta])+\omega_K ^{\Sigma_2}(\iota_{\Sigma_2} ^*[\alpha_2],\iota_{\Sigma_2} ^*[\beta])\\
	&=\omega_K ^{\Sigma_1}(\iota_{\Sigma_1} ^*[\alpha],\iota_{\Sigma_1} ^*[\beta])+\omega_K ^{\Sigma_2}(\iota_{\Sigma_2} ^*[\alpha],\iota_{\Sigma_2} ^*[\beta]).
	\end{align*}
	It completes the proof of Theorem \ref{decomppairingsep}.
\end{proof} 

\subsection{Local decomposition formula: non-separating case}\label{secnonsep}
Suppose that $\xi$ is non-separating such that $ \Sigma_0:=\Sigma\setminus \xi$ is hyperbolic. In this case $\Gamma$ is the fundamental group of \raisebox{-1.5ex}{\begin{tikzpicture}
\draw (0,0) node[left] {$\Sigma_0$};
\draw (0,0) node {$\bullet$};
\draw (.3,0) circle (.3cm);
\draw[->, thick] (.6,.01)--(.6,.02);
\draw (.6,0) node[right] {${\xi}$};
\end{tikzpicture}} and we have the exact sequence:
\[
0\to H^0(\Gamma_\xi, \mathfrak{g}_{\rho_{\Gamma_\xi^ +}}) \to  H^1_{\mathrm{par}} (\Gamma,\mathcal{S} ; \mathfrak{g}_\rho)\overset{\iota^* _{\Sigma_0}}{\to}H^1_{\mathrm{par}} (\Gamma_{\Sigma_0},\mathcal{S}_0; \mathfrak{g}_{\rho_{\Gamma_{\Sigma_0}}})\to 0
\]
where $\mathcal{S} = \{\Gamma_\xi^+, \langle \zeta_1\rangle, \cdots, \langle \zeta_b \rangle \}$, $\mathcal{S}_0=\{\langle \zeta \rangle \,|\, \zeta \text{ is a component of } \partial\overline{ \Sigma_0}\}$ so that $(\Gamma_{\Sigma_0}, \mathcal{S}_0)$ is a group subsystem of $(\Gamma, \mathcal{S})$. Note that the homomorphism $\iota_{\Gamma_{\Sigma_0}}:\Gamma_{\Sigma_0} \to \Gamma$ and $\iota_{\Gamma_\xi ^\pm } : \Gamma_\xi ^{\pm} \to \Gamma_{\Sigma_0}$ are abbreviated to $\iota_{\Sigma_0}$ and $\iota_{\xi^\pm}$ respectively.

The corresponding decomposition formula is the following:
\begin{theorem}\label{decomppairingnonsep}  Let $\Sigma$ be a compact oriented hyperbolic surface and $\xi$ a non-separating essential simple closed curve such that  $\Sigma_0 = \Sigma\setminus \xi $ is hyperbolic subsurface. Let $(\Gamma, \mathcal{S})$ be a group system where $\Gamma=\pi_1(\Sigma)$ and $\mathcal{S}= \{\Gamma_\xi^+, \langle \zeta_1\rangle, \cdots, \langle \zeta_b\rangle \}$. Choose a boundary frame $\mathscr{B}$ and $\{\xi\}$-frame $\mathscr{C}$. Let $[\rho] \in \overline{\Rep}_n ^{\mathscr{B}}(\Gamma, \mathscr{C})$ be such that $[\rho_{\Gamma_{\Sigma_0}}] \in \overline{\Rep}_n ^{\mathscr{B}_0}(\Gamma_{\Sigma_0})$ where 
\[
\mathscr{B}_0= \{ (\zeta, B)\,|\,\zeta\text{ is a component of }\partial\overline{ \Sigma_0}\text{ and }(\iota_{\Sigma_0} (\zeta),B)\in \mathscr{B}\cup \mathscr{C} \}.
\]
Fix a representation $\rho$ in $[\rho]$. For $[\alpha],[\beta]\in  H^1 _{\mathrm{par}} (\Gamma,\mathcal{S}; \mathfrak{g}_{\rho})$, we have
	\[
	\omega_K ^\Sigma ([\alpha], [\beta]) = \omega_K ^{\Sigma_0} (\iota^* _{\Sigma_0} [\alpha], \iota^* _{\Sigma_0} [\beta]).
	\]
\end{theorem}

As in the separating case, we start with proving the following:

\begin{lemma}\label{normalizenonsep}
	Let $\mathcal{S}_0 = \{\langle \zeta \rangle\subset \Gamma_{\Sigma_0} \,|\, \zeta \text{ is a component of } \partial \overline{\Sigma_0}\}$. There are 1-cocycles $\widetilde{\alpha}, \widetilde{\alpha}'$ such that $[\widetilde{\alpha}]=[\widetilde{\alpha}']=[\alpha]$ in $H^1_{\mathrm{par}} (\Gamma, \mathcal{S}; \mathfrak{g})$ and that $\iota_{\Gamma_\xi ^-} ^\# \widetilde{\alpha} =0$, $\iota_{\Gamma_\xi ^+} ^\# \widetilde{\alpha}' =0$ in $Z^1(\Gamma_\xi ^+ ;\mathfrak{g})$ and $Z^1(\Gamma_\xi ^- ; \mathfrak{g})$ respectively.
\end{lemma}
\begin{proof}Choose any representative $\alpha'$ of $[\alpha]$. By the assumption, $\iota^\# _{\Gamma_\xi ^+}(\alpha')= \der_{\Gamma_\xi ^+} X$ for some $X\in \mathfrak{g}$.  Let $\widetilde{\alpha} = \alpha' -\der_\Gamma X$ so that $\widetilde{\alpha}(\Gamma_\xi ^+) =0$.  The construction of $\widetilde{\alpha}'$ is almost the same. 
\end{proof}
\begin{proof}[Proof of  Theorem \ref{decomppairingnonsep}]
	We use the following presentations
	\[
	\Gamma= \langle {x}_1,{y}_1, \cdots, {x}_g, {y}_g, {z}_1, \cdots, {z}_{b+1},\xi^\perp \,|\,{\mathbf{r}}\rangle,
	\] 
	\[
	\Gamma_{\xi} = \langle \xi \rangle \text{ with }\xi^+ = z_{b+1} ,\xi^- = z_{b+2}^{-1}
	\]
 and
	\[
	{\Gamma}_{\Sigma_0} = \langle {x}_1,{y}_1, \cdots, {x}_g, {y}_g, {z}_1, \cdots, {z}_{b+2}\,|\,{\mathbf{r}}_0 \rangle
	\] 
	where 
	\[
	{\mathbf{r}}=[{z}_{b+1},\xi^\perp](\prod_{j=1} ^g [{x}_j,{y}_j])\prod_{j=1} ^{b} {z}_j
	\]
 and
	 \[
	 {\mathbf{r}}_0={z}_{b+1}{z}_{b+2}(\prod_{j=1} ^g [{x}_j,{y}_j])\prod_{j=1} ^{b} {z}_j.
	 \]
Then we have
	\[
	[\Sigma] = [\Sigma_0]-[{z}_{b+1}| {z}_{b+2}]-[{z}_{b+1} \xi^\perp {z}_{b+1} ^{-1}| {z}_{b+1}] +[{z}_{b+1}|\xi^\perp]-[{z}_{b+1}{z}_{b+2}|\xi^\perp].
	\]
	 By Lemma \ref{normalizenonsep}, we can find  $\widetilde{\alpha}$ and  $\widetilde{\beta}$ such that $[\alpha] = [\widetilde{\alpha}]$ ($[\beta] = [\widetilde{\beta}]$, respectively) and $\widetilde{\alpha}({z}_{b+2} ) =0$ ($\widetilde{\beta}({z}_{b+1} ) =0$, respectively).  Now we have 
	\begin{multline*}
	\omega_K ^\Sigma([\alpha], [\beta]) =  \omega_K ^{\Sigma_0}(\iota^* _{\Sigma_0}[\widetilde{\alpha}], \iota^* _{\Sigma_0}[\widetilde{\beta}]) -\Tr (\widetilde{\alpha}({z}_{b+1}){z}_{b+1}\acts \widetilde{\beta}({z}_{b+2}))\\ +\Tr( \widetilde{\alpha}({z}_{b+1}) {z}_{b+1}\acts \widetilde{\beta}(\xi^\perp)) - \Tr( \widetilde{\alpha}({z}_{b+1}{z}_{b+2})({z}_{b+1}{z}_{b+2})\acts \widetilde{\beta}(\xi^\perp))\\ +\Tr  X_{b+1} \widetilde{\beta} (z_{b+1}) + \Tr X_{b+2} \widetilde{\beta}(z_{b+2})
	\end{multline*}
	where $X_{b+1}$ and $X_{b+2}$ are elements of $\mathfrak{g}$ such that $\iota^\#_{{\xi^+}} \widetilde{\alpha}=\der_{\Gamma_{\xi}^+} X_{b+1}$ and $\iota^\#_{\xi^-} \widetilde{\alpha}=\der_{\Gamma_{\xi}^-}X_{b+2}$. Since $\widetilde{\beta}(z_{b+1})=\widetilde{\alpha}(z_{b+2})=0$ the last two terms vanish. 
	
	We expand  using $z_{b+2} = \xi^{\perp} z_{b+1} ^{-1} (\xi^{\perp} )^{-1}$,
	\begin{align*}
	\widetilde{\alpha}({z}_{b+1}){z}_{b+1}\acts \widetilde{\beta}({z}_{b+2})& = \widetilde{\alpha}({z}_{b+1}) {z}_{b+1}\acts (\widetilde{\beta}(\xi^\perp) + (\xi^\perp {z}_{b+1}^{-1})\acts \widetilde{\beta}((\xi^\perp)^{-1}))\\
	&= \widetilde{\alpha}({z}_{b+1}) {z}_{b+1}\acts (\widetilde{\beta}(\xi^\perp) - (\xi^\perp {z}_{b+1}^{-1}(\xi^\perp)^{-1})\acts \widetilde{\beta}(\xi^\perp))\\
	&= \widetilde{\alpha}({z}_{b+1}) {z}_{b+1}\acts (\widetilde{\beta}(\xi^\perp) - {z}_{b+2}\acts \widetilde{\beta}(\xi^\perp)). 
	\end{align*}
	On the other hand, since $\widetilde{\alpha}({z}_{b+2})=0$,
	\[
	\widetilde{\alpha}({z}_{b+1}{z}_{b+2})({z}_{b+1}{z}_{b+2})\acts \widetilde{\beta}(\xi^\perp)= \widetilde{\alpha}({z}_{b+1})({z}_{b+1}{z}_{b+2})\acts \widetilde{\beta}(\xi^\perp). 
	\]
	Therefore, all terms except $ \omega_K ^{\Sigma_0}(\iota^* _{\Sigma_0}[\widetilde{\alpha}], \iota^* _{\Sigma_0}[\widetilde{\beta}]) $ cancel each other. So we get
	\[
	\omega_K ^\Sigma([\alpha], [\beta]) =  \omega_K ^{\Sigma_0}(\iota^* _{\Sigma_0}[\widetilde{\alpha}], \iota^* _{\Sigma_0}[\widetilde{\beta}]) 
	\]
	as desired. 
\end{proof}

Combining Theorem \ref{decomppairingsep} and Theorem \ref{decomppairingnonsep}, one gets the following general local decomposition theorem. 

\begin{corollary}\label{localdecomp}
 Let $\Sigma$ be a compact oriented hyperbolic surface and let $\{\xi_1, \cdots, \xi_m\}$ be  a collection of pairwise disjoint, non-isotopic essential simple closed curves in $\Sigma$ that divide the surface into hyperbolic subsurfaces $\Sigma_1, \cdots, \Sigma_l$. Let $(\Gamma, \mathcal{S})$ be a group system where $\Gamma=\pi_1(\Sigma)$ and $\mathcal{S}= \{\Gamma_{\xi_1}^+, \cdots, \Gamma_{\xi_m}^+, \langle \zeta_1\rangle, \cdots, \langle \zeta_b\rangle \}$. Choose a boundary frame $\mathscr{B}$ and $\mathcal{C}$-frame $\mathscr{C}$.  Let $[\rho]$ be an element in $\overline{\Rep}_n ^{\mathscr{B}}(\Gamma, \mathscr{C})$ such that $[\rho_{\Gamma_{\Sigma_i}}]\in \overline{\Rep}_n ^{\mathscr{B}_i} (\Gamma_{\Sigma_i})$ for each $i=1,2,\cdots, l$ where 
 \[
 \mathscr{B}_i = \{ (\zeta, B)\,|\, \zeta\text{ is a component of }\partial \overline{\Sigma_i}\text{ and } (\iota_{\Sigma_i} (\zeta), B)\in \mathscr{B}\cup \mathscr{C}\}.
 \]
Fix a representative $\rho$ of $[\rho]$. Then for any $[\alpha],[\beta]\in H^1 _{\mathrm{par}} (\Gamma, \mathcal{S};\mathfrak{g}_\rho)$, we have
 \[
 \omega^\Sigma _K ([\alpha], [\beta] ) = \sum_{i=1} ^l \omega^{\Sigma_i} _K (\iota^* _{\Sigma_i} [\alpha], \iota^* _{\Sigma_i} [\beta]). 
 \]
\end{corollary}
\begin{proof}
We use induction on the number of curves in $\mathcal{C}$. If $\mathcal{C}$ consists of a single curve $\xi$, we are done by Theorem \ref{decomppairingsep} or \ref{decomppairingnonsep} depending on whether $\xi$ is separating or not. 

Suppose that a collection $\mathcal{C} = \{\xi_1, \cdots, \xi_m\}$, $m>1$, is given where $\xi_m$  is separating. Without loss of generality, we may assume that $\Sigma_-:=\Sigma_1\cup \cdots\cup \Sigma_p$ and $\Sigma_+ :=\Sigma_{p+1} \cup \cdots \cup \Sigma _l$ are two components of $\Sigma \setminus \xi_m$. By virtue of Theorem \ref{decomposerep}, we can identify $\Gamma$ with the fundamental group of graph of groups  \raisebox{-2ex}{\begin{tikzpicture}
\draw (0,0) node {$\bullet$} ;
\draw (0,0) -- (1,0) node {$\bullet$};
\draw[->, thick] (.5,0)--(.51,0);
\draw (0,0) node[above] {$\Sigma_-$};
\draw (1,0) node[above] {$\Sigma_+$};
\draw (.5,0) node[above]{$\xi_m$};
\end{tikzpicture}}.  Let $\mathcal{S}_{\pm}=\{\langle \zeta\rangle\subset \Gamma_{\Sigma_\pm} \,|\, \zeta\text{ is a component of }\partial \overline{\Sigma_{\pm}}\}$ so that $(\Gamma_{\Sigma_{+}}, \mathcal{S}_{+})$ and $(\Gamma_{\Sigma_{-}}, \mathcal{S}_{-})$ become  group subsystems of $(\Gamma, \{\Gamma_{\xi_m}^+,\langle\zeta_1\rangle, \cdots, \langle\zeta_b\rangle\})$.  Then by Theorem \ref{decomppairingsep}, 
\[
\omega^{\Sigma}_K ([\alpha], [\beta]) = \omega^{\Sigma_-} _K (\iota^* _{\Sigma_-} [\alpha] ,\iota^* _{\Sigma_- } [\beta]) + \omega^{\Sigma_+} _K (\iota^* _{\Sigma_+} [\alpha] ,\iota^*_{\Sigma_+} [\beta]),
\] 
where $\iota_{\Sigma_{+}}$ and $\iota_{\Sigma_{-}}$ are the natural maps $(\Gamma_{\Sigma_{+}}, \mathcal{S}_{+}) \to (\Gamma, \mathcal{S})$ and  $(\Gamma_{\Sigma_{-}}, \mathcal{S}_{-}) \to (\Gamma, \mathcal{S})$ respectively. 
Observe that a collections of curves $ \{\xi\in \mathcal{C} \,|\,\xi\cap \Sigma_+ \ne \emptyset \}$ and $\{\xi\in \mathcal{C} \,|\,\xi\cap \Sigma_- \ne \emptyset \}$ cut $\overline{\Sigma_+}$ and $\overline{\Sigma_-}$ into $\Sigma_1,\cdots,\Sigma_p$ and $\Sigma_{p+1}, \cdots, \Sigma_l$ respectively.  By the induction hypothesis, we have  
\[
 \omega^{\Sigma_-} _K (\iota^* _{\Sigma_-} [\alpha] ,\iota^* _{\Sigma_- } [\beta])= \sum_{i=1} ^p \omega^{\Sigma_i} _K (\widebar{\iota _{\Sigma_i}}^* \iota^* _{\Sigma_-} [\alpha], \widebar{\iota _{\Sigma_i}}^* \iota^*_{\Sigma_-} [\beta]),
 \]
 and  
 \[
 \omega^{\Sigma_+} _K (\iota^* _{\Sigma_+} [\alpha] ,\iota^* _{\Sigma_+ } [\beta])= \sum_{i=p+1} ^l \omega^{\Sigma_i} _K (\widebar{\iota _{\Sigma_i}}^* \iota^* _{\Sigma_+} [\alpha], \widebar{\iota _{\Sigma_i}}^* \iota^*_{\Sigma_+} [\beta])
 \]
where $\widebar{\iota_{\Sigma_i}}: \Gamma_{\Sigma_i} \to \Gamma_{\Sigma_\pm}$, $i=1,2,\cdots, l$ are the natural inclusions. We observe that $\widebar{\iota _{\Sigma_i}}^*\iota^* _{\Sigma_\pm} = \iota^* _{\Sigma_i}$. Therefore, we obtain
\[
 \omega^\Sigma _K ([\alpha], [\beta] ) = \sum_{i=1} ^l \omega^{\Sigma_i} _K (\iota^* _{\Sigma_i} [\alpha], \iota^* _{\Sigma_i} [\beta]). 
\] 

Now suppose that $\xi_m$ is non-separating. Let $\Sigma_0:=\Sigma\setminus \xi_m$. Then by Theorem \ref{decomposerep}, $\Gamma$ is the fundamental group of a graph of groups \raisebox{-1.5ex}{\begin{tikzpicture}
\draw (0,0) node[left] {$\Sigma_0$};
\draw (0,0) node {$\bullet$};
\draw (.3,0) circle (.3cm);
\draw[->, thick] (.6,.01)--(.6,.02);
\draw (.6,0) node[right] {${\xi_m}$};
\end{tikzpicture}}. By Theorem \ref{decomppairingnonsep}, we have 
\[
\omega^\Sigma _K ([\alpha],[\beta]) = \omega^{\Sigma_0} _K (\iota^* _{\Sigma_0} [\alpha], \iota^* _{\Sigma_0} [\beta]). 
\]
Here $\iota_{\Sigma_0}$ is the injection from $\Gamma_{\Sigma_0}$ into $\Gamma$. Since $\mathcal{C}\setminus \{\xi_m\}$ divides $\Sigma_0$ into $\Sigma_1, \cdots, \Sigma_l$,  by the induction hypothesis, we obtain
\[
\omega^{\Sigma_0} _K (\iota^* _{\Sigma_0} [\alpha], \iota^* _{\Sigma_0} [\beta])= \sum_{i=1} ^{l} \omega^{\Sigma_i} _K (\widebar{\iota_{\Sigma_i}} ^* \iota^ *_{\Sigma_0} [\alpha], \widebar{\iota_{\Sigma_i}} ^* \iota^ *_{\Sigma_0}[\beta]).
\]
Since $\widebar{\iota_{\Sigma_i}}^* \iota^* _{\Sigma_0}=\iota^* _{\Sigma_i}$, we have 
\[
\omega^{\Sigma_0} _K (\iota^* _{\Sigma_0} [\alpha], \iota^* _{\Sigma_0} [\beta])=\sum_{i=1} ^{l} \omega^{\Sigma_i} _K (\iota_{\Sigma_i} ^* [\alpha], \iota_{\Sigma_i} ^* [\beta]).
\]
This completes the induction and Corollary \ref{localdecomp} follows.  
\end{proof}

\subsection{Global decomposition} As mentioned in the introduction, we can decompose $\pi_(\Sigma)$ into $\pi_1(\Sigma_i)$'s and this decomposition allows us to construct the map 
\begin{equation}\label{restmap}
\Rep_n (\Gamma) \to \Rep_n (\Gamma_{\Sigma_1}) \times \cdots \times \Rep_n (\Gamma_{\Sigma_l})
\end{equation}
induced from $[\rho]\mapsto ([\rho_{\Gamma_{\Sigma_1}}],[\rho _{\Gamma_{\Sigma_2}}],\cdots, [\rho _{\Gamma_{\Sigma_l}}] )$.

Recall that Theorem 9.1 of Labourie-McShane \cite{labourie2009} shows that  if $[\rho]$ is Hitchin, then so is each factor $\rho_{\Gamma_{\Sigma_i}}$.  Therefore, if we restrict (\ref{restmap}) to $ \Hit_n ^{\mathscr{B}}(\Sigma, \mathscr{C})$ we get the map
\begin{equation}\label{rest}
\overline{\Phi}: \Hit_n ^{\mathscr{B}}(\Sigma, \mathscr{C}) \to \Hit_n ^{\mathscr{B}_1} (\Sigma_1) \times  \cdots \times \Hit_n ^{\mathscr{B}_l}(\Sigma_l)
\end{equation}
where 
\[
\mathscr{B}_i=\{(\xi,B)\,|\,\xi\text{ is a component of } \partial \overline{\Sigma_i}\text{ and } (\iota_{\Sigma_i}(\xi),B)\in \mathscr{B}\cup \mathscr{C}\}.
\]

\begin{proposition}\label{JM2}    Let $\Sigma$ be a compact oriented hyperbolic surface possibly with boundary components $\{\zeta_1, \cdots, \zeta_b\}$ and let $\{\xi_1, \cdots, \xi_m\}$ be  a collection of pairwise disjoint, non-isotopic oriented essential simple closed curves in $\Sigma$ that divide the surface into hyperbolic subsurfaces $\Sigma_1, \cdots, \Sigma_l$.  We have the following:
	\begin{itemize}
		\item Let $(\Gamma, \mathcal{S})$ be a group system where $\Gamma=\pi_1(\Sigma)$, 
		\[
		\mathcal{S}=\{\langle \zeta_1\rangle, \cdots, \langle \zeta_b\rangle, \Gamma_{\xi_1}^+, \cdots, \Gamma_{\xi_m}^+\}
		\]
		and let
		\[
		\mathcal{S}_i = \{\langle\zeta\rangle\subset\Gamma_{\Sigma_i}\,|\, \zeta \text{ is a compoment of }\partial \overline{\Sigma_i}\}.
		\]
		Then we have identifications
		\[
		T_{[\rho]} \Hit ^{\mathscr{B}} _n (\Sigma, \mathscr{C}) = H^1_{\mathrm{par}} (\Gamma, \mathcal{S}; \mathfrak{g}_\rho)
		\]
		and
		\[
		T_{\overline{\Phi}([\rho])} \Hit_n ^{\mathscr{B}_1}(\Sigma_1)\times\cdots \times \Hit_n ^{\mathscr{B}_l}(\Sigma_l) =\bigoplus_{i=1} ^l H^1 _{\mathrm{par}} (\Gamma_{\Sigma_i } , \mathcal{S}_i; \mathfrak{g}_{\rho_{\Gamma_{\Sigma_i}}}).
		\]
		\item Under the above identifications,  the differential $\der \overline{\Phi} $ fits into the Mayer-Vietoris  sequence 
		\[
		0\to \bigoplus_{i=1} ^ m H^0(\Gamma_{\xi_i}; \mathfrak{g}) \overset{\delta}{\to} H^1 _{\mathrm{par}} (\Gamma,\mathcal{S}; \mathfrak{g}) \overset{\der \overline{\Phi} }{\to} \bigoplus_{i=1} ^l  H^1 _{\mathrm{par}} (\Gamma_{\Sigma_i},\mathcal{S}_i; \mathfrak{g})  \to 0,
		\]
	that is,
	\[
	\der \overline{\Phi} ([\alpha]) = \iota_{\Sigma_1} ^* [\alpha] \oplus \cdots \oplus \iota^* _{\Sigma_l}[\alpha].
	\]
	\end{itemize}
\end{proposition}

\begin{proof}

The first statement is already done in Proposition \ref{tangent}.

The second assertion follows from  the definition of $\overline{\Phi}$ and Proposition \ref{mvs}. 
\end{proof}

\begin{lemma}\label{connectedfiber}
Each fiber of $\overline{\Phi}$ is connected.  
\end{lemma}
\begin{proof}
We complete $\mathcal{C}$ to get a maximal geodesic lamination of $\Sigma$ and  construct the Bonahon-Dreyer coordinates on $\Hit_n (\Sigma)$ and on $\Hit_n (\Sigma_i)$ (see \cite{bonahon2014} or Appendix \ref{BDreview}) with respect to this maximal lamination. In this coordinates, $\Hit_n ^{\mathscr{B}}(\Sigma, \mathscr{C})$ is the set
\begin{multline*}
\{[\rho]\in \Hit_n (\Sigma)\,|\, l ^{\zeta_j}(\rho) = l ^{\zeta_j} (\rho_0),\, l  ^{\xi_k}(\rho) = l ^{\xi_k}  (\rho_0),\\ j=1,2,\cdots, b,\text{ and }k=1,2,\cdots, m\}
\end{multline*}
for some fixed reference point $[\rho_0]\in \Hit_n ^{\mathscr{B}}(\Sigma, \mathscr{C})$. Here  $l ^{\xi}$ is defined by
\[
l ^{\xi} (\rho) = \left( \log \frac{|\lambda_1(\rho(\xi))|}{|\lambda_{2}(\rho(\xi))|}, \cdots,  \log \frac{|\lambda_{n-1}(\rho(\xi))|}{|\lambda_{n}(\rho(\xi))|}\right)\in \R^{n-1}
\]
where $\xi$ is  a closed leaf or a boundary component and $\lambda_i(g)$ is the $i$th largest eigenvalue of $g\in G$. Recall that each component of $l^{\zeta_i}$ and $l^{\xi_i}$ can be expressed as a linear combination of triangle invariants and shear invariants. Moreover one can express $\overline{\Phi}$ as
\[
\overline{\Phi}=\operatorname{pr}_{\Sigma_1}\times \cdots \times \operatorname{pr}_{\Sigma_l}
\]
where $\operatorname{pr}_{\Sigma_i}$ denotes the projection onto the triangle invariants and shear invariants associated to ideal triangles and (infinite or closed) leaves contained in the interior of $\Sigma_i$. It follows that the fiber of $\overline{\Phi}$ is spanned by the shear invariants associated to closed leaves $\mathcal{C}$. Therefore the fiber of $ \overline{\Phi}$ is connected. 
\end{proof}

We now introduce a Hamiltonian $\R^{m(n-1)}$-action that makes $\overline{\Phi}$ an affine bundle over the base space $\Hit_n ^{\mathscr{B}}(\Sigma, \mathscr{C})/ \R^{m(n-1)}$. Then we prove that the base space $\Hit_n ^{\mathscr{B}}(\Sigma, \mathscr{C})/ \R^{m(n-1)}$ is the symplectic reduction.

Let $\mathbf{Hyp}^+$ be the set of purely loxodromic (or positive hyperbolic) elements in $G=\PSL_n(\R)$. By an invariant function we mean a smooth function $f:\mathbf{Hyp}^+ \to \R$ such that $f(ghg^{-1}) = f(h)$ for all $h\in \mathbf{Hyp}^+$ and $g\in G$. Given an invariant function $f$, there associated another function $F:\mathbf{Hyp}^+ \to \mathfrak{g}$ characterized by the property that $\frac{d}{dt}|_{t=0} f (g \exp tX)  = \Tr (F(g) X)$ for all $X\in \mathfrak{g}$. Observe that $\Ad_g (F(h)) =F(ghg^{-1})$. 

Let $f_1, \cdots, f_{n-1}$ be invariant functions such that 
\[
g\mapsto f(g):=(f_1(g),\cdots, f_{n-1}(g))
\]
is injective and that $\{F_1(g), F_2(g), \cdots, F_{n-1}(g)\}$ forms a basis of $\ker(\Ad_g - \operatorname{Id})$ where $g\in \mathbf{Hyp}^+$. To each oriented essential simple closed curve $\xi$, associate a map $f_\xi:\Hit_n ^{\mathscr{B}}(\Sigma)\to \R^{n-1}$ which is defined by $f_\xi ([\rho]) = f(\rho(\xi))$. 

Given $\mathcal{C}=\{\xi_1, \cdots, \xi_m\}$ a family of mutually disjoint, non-isotopic oriented essential simple closed curves, let $\mathbf{T}^t_{\xi_i, j}([\rho])= [\Phi^t_{F_j(\rho(\xi_i )), \xi_i}(\rho)]$, the algebraic bending by $F_j(\rho(\xi_i))$ along $\xi_i$.  Then for $(\textbf{t}_1, \cdots, \textbf{t}_m)\in \R^{m(n-1)}$, where $\mathbf{t}_i= (t^1_i, \cdots, t^{n-1}_i)\in \R^{n-1}$, we define the complete flow
\[
\mathbf{T}^{(\textbf{t}_1, \cdots, \textbf{t}_m)}([\rho]) = \mathbf{T}^{t^{n-1} _m} _{\xi_m,n-1}\circ\mathbf{T}^{t^{n-2} _m} _{\xi_m,n-2}\circ\cdots \circ \mathbf{T}^{t^2_1}_{\xi_1, 2}\circ  \mathbf{T}^{t^1_1}_{\xi_1, 1} ([\rho]). 
\]
The above formula is well-defined in the sense that it does not depend on the order of compositions. Hence we obtain the $\R^{m(n-1)}$-action on $\Hit^{\mathscr{B}} _n(\Sigma)$ given  by
\[
(\textbf{t}_1, \cdots, \textbf{t}_m)\cdot [\rho] = \mathbf{T}^{(\textbf{t}_1, \cdots, \textbf{t}_m)}([\rho]).
\]
Recall that $\delta(F_j(\rho(\xi_i)))$ is the fundamental vector field of the unit vector (seen as a Lie algebra element) in the direction of $t^j _i$ at $[\rho]$.

\begin{lemma}\label{freeaction}
The $\R^{m(n-1)}$-action on $\Hit_n ^{\mathscr{B}}(\Sigma, \mathscr{C})$ is free.  
\end{lemma}
\begin{proof}
Choose a representative $\rho$ of  $[\rho]$. We observe that, by construction of the $\R^{m(n-1)}$-action, $(\mathbf{t}\cdot \rho)|_{\Gamma_{\Sigma_0}}=\rho_{\Gamma_{\Sigma_0}}$  on  the vertex group $\Gamma_{\Sigma_0}$ of the base vertex. Suppose that  $\mathbf{t}\cdot [\rho]=[\rho]$ for some $\mathbf{t}\in \R^{m(n-1)}$. Then, by Lemma \ref{NoInvariantElement},  $\mathbf{t} \cdot \rho =\rho$ as representations. Now by induction and the definition of the action we have $\mathbf{t}=\mathbf{0}$.  Therefore, the $\R^{m(n-1)}$-action  is free. 
\end{proof}

\begin{lemma}\label{properness}
Let 
\[
\mathcal{H}^{\mathscr{B}} (\Gamma,\mathscr{C}):=\{\rho\in \Hom(\Gamma, G)\,|\,[\rho]\in\Hit_n ^{\mathscr{B}}(\Sigma, \mathscr{C})\}.
\]
The $\R^{m(n-1)}$-action on $\mathcal{H}^{\mathscr{B}} (\Gamma,\mathscr{C})$ is proper. 
\end{lemma}
\begin{proof}
Define 
\[
\mathcal{H} (\Gamma_{\Sigma_i}):=\{\rho\in \Hom(\Gamma_{\Sigma_i}, G)\,|\, [\rho]\in \Hit_n(\Sigma_i)\}
\]
for each $i=1,2,\cdots, l$. We know, by Lemma \ref{NoInvariantElement}, that $\mathcal{H}^{\mathscr{B}} (\Gamma,\mathscr{C})$ and $\mathcal{H}(\Gamma_{\Sigma_i})$ are subspaces of $\Hom_s(\Gamma, G)$ and $\Hom_s(\Gamma_{\Sigma_i},G)$ respectively.   Let $C$ be a compact subset of $\mathcal{H}^{\mathscr{B}} (\Gamma,\mathscr{C})$. We know that the restriction map 
\[
\iota_{\Sigma_i}: \mathcal{H}^{\mathscr{B}} (\Gamma,\mathscr{C}) \to \mathcal{H}(\Gamma_{\Sigma_i})
\]
and 
\[
\iota_{\xi_j^\perp}:\mathcal{H}^{\mathscr{B}} (\Gamma,\mathscr{C}) \to \Hom(\langle \xi_j ^\perp \rangle, G),\quad \xi_j \in E(\mathcal{G})\setminus E(\mathcal{D}),
\]
are continuous and equivariant with respect to the $\R^{m(n-1)}$-action. Let
\begin{align*}
U_i &:=\{\mathbf{t}\in \R^{m(n-1)}\,|\, \mathbf{t}\cdot \iota_{\Sigma_i}(C) \cap \iota_{\Sigma_i}(C)  \ne \emptyset \},\\
V_j&:= \{\mathbf{t}\in \R^{m(n-1)}\,|\, \mathbf{t}\cdot \iota_{\xi_j^\perp}(C) \cap \iota_{\xi_j ^\perp}(C)  \ne \emptyset \}.
\end{align*}
Since $\iota_{\Sigma_i}$ and $\iota_{\xi_j^ \perp}$ are equivariant, 
\[
\{\mathbf{t}\in \R^{m(n-1)}\,|\, \mathbf{t} \cdot C \cap C \ne \emptyset\} \subset \bigcap_{i=1} ^l  U_i \cap \bigcap_{j=1} ^N V_j
\]
where $N= |E(\mathcal{G})\setminus E(\mathcal{D})|$. We claim that  $ \bigcap_{i=1} ^l  U_i \cap \bigcap_{j=1} ^N V_j$  is compact. Since 
\begin{equation}\label{properset}
\{\mathbf{t}\in \R^{m(n-1)}\,|\, \mathbf{t} \cdot C \cap C \ne \emptyset\}
\end{equation}
is closed, this shows that (\ref{properset}) is compact. 

It is known that the $G$-action on $\Hom_s(\Gamma_{\Sigma_i}, G)$ is proper. See   Proposition 1.1 of Johnson-Millson \cite{Johnson1987}.   Hence, on each $\mathcal{H}(\Gamma_{\Sigma_i})\subset \Hom_s(\Gamma_{\Sigma_i},G)$, the set 
\[
D:=\{g\in G\,|\, g \iota_{\Sigma_i}(C) g^{-1} \cap \iota_{\Sigma_i}(C) \ne \emptyset \}
\]
is compact. Suppose that $\xi$ is in $E(\mathcal{D})$ and precedes $\Sigma_i$.  Let
\begin{multline*}
E:=\{(t_1, \cdots, t_{n-1}) \in \R^{n-1}\,|\, \exp (t_1 F_1(\rho(\xi))+ \cdots+ t_{n-1}F_{n-1}(\rho(\xi)))\in D\\ \text{ for some } \rho \in C\}.
\end{multline*}
Recall that the $\R^{n-1}$ action  on $\Hom_s(\Gamma_{\Sigma_i}, G)$ corresponding to the flow along $\xi$ is conjugation by $\exp (t_1 F_1 (\rho(\xi))+\cdots+t_{n-1}F_{n-1}(\rho(\xi)))$, $\rho \in C$. Hence we have 
\[
\{\mathbf{t} \in \R^{n-1}\,|\, \mathbf{t}\cdot \iota_{\Sigma_i}(C) \cap \iota_{\Sigma_i}(C)\ne \emptyset\} \subset E.
\]
We claim that $E$ is compact which also proves that 
\[
\{\mathbf{t} \in \R^{n-1}\,|\, \mathbf{t}\cdot \iota_{\Sigma_i}(C) \cap \iota_{\Sigma_i}(C)\ne \emptyset\}
\]
is compact.  Consider the map $k: \R^{n-1} \times C \to G$ given by 
\[
((t_1, \cdots, t_{n-1}), \rho)\mapsto \exp (t_1 F_1 (\rho(\xi))+\cdots+t_{n-1}F_{n-1}(\rho(\xi))).
\]
This map is continuous.  Moreover if $W$ is an unbounded subset of $\R^{n-1}$ then so is $k(W\times C)$ where $G$ is given the operator norm. Since $C$ is  compact, the projection  $p_1:\R^{n-1}\times C \to \R^{n-1}$ onto the first factor is a closed map. Therefore,  $E=p_1(k^{-1}(D))$ is closed and bounded  subset of $\R^{n-1}$ so $E$ is compact.  

By simple induction, we have 
\[
\bigcap_{i=1} ^l U_i = A_1 \oplus \cdots \oplus A_l \oplus \R^{N(n-1)}
\]
where  each $A_i$ is a compact subspace of a subgroup $\R^{n-1}$ of $\R^{m(n-1)}$ corresponding to the flow along an edge in $\mathcal{D}$.  

Now we claim that the set 
\[
B_j:= \{\mathbf{t}\in \R^{n-1} \,|\, \mathbf{t}\cdot \iota_{\xi_j ^\perp}(C) \cap \iota_{\xi_j ^\perp}(C) \ne \emptyset\}
\]
is compact. Recall that the $\R^{n-1}$ action on $\iota_{\xi_j ^\perp}(C)$ is the right multiplication by $\exp (t_1F_1(\rho(\xi_j ))+\cdots+t_{n-1}F_{n-1}(\rho(\xi_j)))$. Let $\overline{A}^+$ be the set of diagonal matrices with diagonal entries being sorted from largest to smallest. Consider the Cartan projection $\mathfrak{a}: G \to \overline{A}^+$ which is known to be continuous and proper. We may assume that $K\cdot \iota_{\xi_j^\perp}(C)=\iota_{\xi_j^\perp}(C)\cdot K = \iota_{\xi_j^\perp}(C)$  where $K$ is a maximal compact subgroup of $G$. Then we observe that 
\[
F:=\{ g\in \overline{A}^+\,|\, \iota_{\xi_j^\perp}(C) g \cap \iota_{\xi_j^\perp}(C) \ne \emptyset\}
\]
is compact.   Indeed if $F$ is not compact, there is an unbounded sequence $\{g_i\}$ in $\overline{A}^+$ such that $ \iota_{\xi_j^\perp}(C) g_i \cap \iota_{\xi_j ^\perp}(C) \ne \emptyset$ for all $i$. Then $\iota_{\xi_j^\perp}(C)$ must be unbounded, which contradicts  the assumption that $\iota_{\xi_j^\perp}(C)$ is compact. Since $\mathfrak{a}$ is proper, $\mathfrak{a}^{-1}(F)=\{ g\in G\,|\, \iota_{\xi_j^\perp}(C)g \cap \iota_{\xi_j^\perp}(C) \ne \emptyset\}$ is compact in $G$. Thus $B_j=p_1(k^{-1}(\mathfrak{a}^{-1}(F)))$ is also closed and bounded. It follows that $B_j$  must be compact. 

Hence topologically, $\bigcap_{i=1} ^l  U_i \cap \bigcap_{j=1} ^N V_j$ is a closed subspace of $A_1\times \cdots \times A_l \times B_1 \times \cdots \times B_N$. Since each $A_i$ and $B_j$ are compact, $\bigcap_{i=1} ^l  U_i \cap \bigcap_{j=1} ^N V_j$ is also compact. 
\end{proof}

Let $\mu:\Hit_n ^\mathscr{B}(\Sigma) \to \R^{m(n-1)}$ be the function defined by 
\begin{equation}\label{momentmap}
\mu([\rho])= (f_{\xi_1}(\rho), \cdots, f_{\xi_m}(\rho))
\end{equation}
and let $\mathbf{L}=\operatorname{image} \mu$. $\mu$ is the complete invariant of conjugacy classes of $\mathbf{Hyp}^+$. Therefore, the value  $\mu(g)$ determines the conjugacy class in which $g\in \mathbf{Hyp}^+$ is contained.

\begin{theorem}[Generalization of Goldman \cite{goldman1986}]\label{twist} Keep the assumption of Proposition \ref{JM2}. For any boundary frame $\mathscr{B}$, the $\R^{m(n-1)}$-action on $\Hit_n ^{\mathscr{B}}(\Sigma)$ is Hamiltonian whose moment map is given by (\ref{momentmap}).  Each $y\in \mathbf{L}$ is a regular value of $\mu$ and the action is proper on $\mu^{-1}(y)$.
\end{theorem} 
\begin{proof}
Theorem 4.3 of Goldman \cite{goldman1986} states that when $\Sigma$ is closed, the $\R^{m(n-1)}$-action on  $(\Hit_n(\Sigma),\omega_G)$ is weakly Hamiltonian. Since curves in $\mathcal{C}$ are pairwise disjoint, non-isotopic, Theorem 3.5 of \cite{goldman1986} implies that the Hamiltonian functions $[\rho] \mapsto f_i (\rho(\xi_j))$ commute each other. Therefore this action is Hamiltonian. 

Now we assume that $\Sigma$ has boundary. Let $(\Gamma, \mathcal{S})$ be a group system where $\Gamma=\pi_1(\Sigma)$ and $\mathcal{S}=\{\langle \zeta_1\rangle, \cdots, \langle \zeta_b\rangle, \Gamma_{\xi_1}^+, \cdots, \Gamma_{\xi_m}^+\}$.   We first consider the cohomological operation
\[
H^1 (\Gamma ,\mathcal{S} ; \mathfrak{g}) \otimes H^1(\Gamma; \mathfrak{g} ) \overset{\smile}{\to} H^2(\Gamma,\mathcal{S} ; \R ) \overset{\cap [\Sigma]}{\to} \R
\] 
where the first arrow is the usual cup product and the second is the cap product with a relative fundamental class $[\Sigma] \in H_2(\Gamma, \mathcal{S};\R)$. It descends to the operation
\[
H^1 _{\mathrm{par}} (\Gamma,\mathcal{S};\mathfrak{g}) \otimes H^1_{\mathrm{par}}(\Gamma, \mathcal{S}; \mathfrak{g}) \to \R
\]
which is the same as the explicitly defined form $\omega^\Sigma _K$. See  Lemma 8.4 of \cite{guruprasad1997}. Then as in the proof of Proposition 3.7 of  \cite{goldman1986}, we can show that the Poincar\'{e} dual of the cohomology element $\mathbb{X}_{f_{\xi_i}}|_{[\rho]}\in H^1_{\mathrm{par}}(\Gamma, \mathcal{S};\mathfrak{g}_\rho)\subset H^1(\Gamma;\mathfrak{g}_\rho)$ is $\xi_i \otimes F_j (\rho(\xi_i))\in H_1(\Gamma,\mathcal{S};\mathfrak{g})$. This follows from the commutativity of following diagram, absolute version of which appears in the proof of Proposition 3.7 of  \cite{goldman1986}:
\[
\xymatrix{
H^1(\Gamma;\mathfrak{g})\ar[d] ^{\tilde{\omega_K }}\ar[r]^{\cap [\Sigma]}\ar[rd]^{\theta} & H_1(\Gamma,\mathcal{S}; \mathfrak{g})\ar[d]^{\eta}\\
H^1(\Gamma,\mathcal{S};\mathfrak{g})^* &  \ar[l]^{\Tr^*} H^1(\Gamma,\mathcal{S};\mathfrak{g}^*)^*
}.
\]
For the precise definition of each map, we refer to Goldman \cite{goldman1986}. One can also prove that, by exactly the same argument of Theorem 4.3 of \cite{goldman1986},  the  Poincar\'{e} dual of $[\frac{\partial}{\partial t^j _i } \mathbf{T}]=\delta(F_j(\rho(\xi_i)))\in H^1(\Gamma; \mathfrak{g})$ is  given by $\xi_i \otimes F_j (\rho(\xi_i))$ as well. This proves that the action is weakly Hamiltonian. To prove that the action is Hamiltonian, we again use the fact that $\xi_i$ are all disjoint which implies that $\{f_{\xi_i}, f_{\xi_j}\}=0$ for all $i, j$. 

It remains to prove the properness of the action. We show the following claim first.
\begin{claim}
Let $\mathscr{C}$ be the $\mathcal{C}$-frame such that $\mu^{-1}(y) = \Hit_n ^{\mathscr{B}}(\Sigma, \mathscr{C})$. For any given compact subset $C\subset \mu^{-1}(y)=\Hit_n ^{\mathscr{B}}(\Sigma, \mathscr{C})$, there is a compact set $C' \subset \mathcal{H}^{\mathscr{B}} (\Gamma,\mathscr{C})$ satisfying the following properties:
\begin{itemize}
\item $C= p(C')$ where $p: \mathcal{H}^{\mathscr{B}} (\Gamma,\mathscr{C}) \to \Hit_n ^{\mathscr{B}}(\Sigma, \mathscr{C})$ is the projection $\rho\mapsto [\rho]$
\item If $\rho \in C'$ and $[\mathbf{t} \cdot \rho ] \in C$ for some $\mathbf{t} \in \R^{m(n-1)}$ then $\mathbf{t} \cdot \rho \in C'$. 
\end{itemize}
\end{claim}
\begin{proof}[Proof of the Claim]
We extend $\mathcal{C}$ to a maximal geodesic lamination on $\Sigma$ and fix an ideal triangle $T$ contained in $\Sigma_0$, the origin of the tree $\mathcal{D}$. By Labourie-McShane \cite{labourie2009}, there is a equivariant flag curve $\mathcal{F}_\rho:\partial_\infty \widetilde{\Sigma} \to \Flag(\R^n)$ for each $\rho \in \mathcal{H}^{\mathscr{B}} (\Gamma,\mathscr{C})$ where $\widetilde{\Sigma}$ is the universal cover of $\Sigma$. Fix also flags $P$, $Q$ and a projective line $R$ such that $(P,Q,R)$ is generic. Denote by $p,q,r$ the three vertices of a lift $\widetilde{T}$ of $T$. Then for each $[\rho]\in \Hit_n ^{\mathscr{B}}(\Sigma, \mathscr{C})$ there is a unique $\widetilde{\rho}\in \mathcal{H}^{\mathscr{B}} (\Gamma,\mathscr{C})$ such that $[\widetilde{\rho}] = [\rho]$, $\mathcal{F}_{\widetilde{\rho}} (p)=P,$ $\mathcal{F}_{\widetilde{\rho}} (q) = Q$ and $\mathcal{F}_{\widetilde{\rho}}(r)^{(1)}= R$. Let $s: \mu^{-1}(y) \to \mathcal{H}^{\mathscr{B}} (\Gamma,\mathscr{C})$ be the map defined by $s([\rho]) = \widetilde{\rho}$.  Define $C' = s(C)$. Since the map $s$ is continuous, $C'$ is also compact. Suppose that $p(\mathbf{t}\cdot \rho)=[\mathbf{t}\cdot \rho] \in C$ for some $\mathbf{t}\in \R^{m(n-1)}$ and some $\rho \in C'$. Since $(\mathbf{t}\cdot \rho)|_{\Gamma_{\Sigma_0}} = \rho|_{\Gamma_{\Sigma_0}}$, we have that $\mathcal{F}_{\mathbf{t}\cdot \rho} (p)= P$,  $\mathcal{F}_{\mathbf{t}\cdot \rho}(q)= Q$, and  $\mathcal{F}_{\mathbf{t}\cdot \rho} (r) ^{(1)}= R$. Hence $\mathbf{t} \cdot \rho = s([\mathbf{t}\cdot \rho])$. It follows that $C'$ is the desired compact set. 
\end{proof}
To prove the properness, we lift the compact set $C$ to $C'$ as above.  We observe that
\[
\{\mathbf{t}\in \R^{m(n-1)} \,|\, \mathbf{t} \cdot C\cap C \ne \emptyset\} = \{\mathbf{t}\in \R^{m(n-1)} \,|\, \mathbf{t}\cdot C' \cap C' \ne \emptyset\}.
\]
The right hand side is compact by Lemma \ref{properness}. Therefore, the $\R^{m(n-1)}$-action is proper on $\mu^{-1}(y)$. 
\end{proof}

In particular, by virtue of Theorem \ref{MWq}, we can construct the Marsden-Weinstein quotient
\[
q:  \mu^{-1}(y) \to \mu^{-1}(y)/ \R^{m(n-1)}
\]
We denote by $\widetilde{\omega}_K ^\Sigma$ the induced symplectic form on $ \mu^{-1}(y)/ \R^{m(n-1)}$.

Let $\mathscr{C}$ be the $\mathcal{C}$-frame such that $\mu^{-1}(y) = \Hit_n ^{\mathscr{B}}(\Sigma, \mathscr{C})$. As we mentioned above  the quotient space $\Hit_n ^{\mathscr{B}}(\Sigma, \mathscr{C}) /\R^{m(n-1)}$ carries the symplectic form $\widetilde{\omega}_K ^\Sigma$.  On the other hand, the target of $\overline{\Phi}$, $\Hit_n ^{\mathscr{B}_1}(\Sigma_1)\times\cdots \times \Hit_n ^{\mathscr{B}_l}(\Sigma_l)$, also admits a symplectic form $\omega_K ^{\Sigma_1}\oplus \cdots \oplus \omega_K ^{\Sigma_l}$. 

\begin{theorem}\label{gendecomp}   Let $\Sigma$ be a compact oriented hyperbolic surface possibly with boundary components $\{\zeta_1, \cdots, \zeta_b\}$ and let $\{\xi_1, \cdots, \xi_m\}$ be  a collection of pairwise disjoint, non-isotopic oriented essential simple closed curves in $\Sigma$ that divide the surface into hyperbolic subsurfaces $\Sigma_1, \cdots, \Sigma_l$. Let $\mathscr{B}$ and $\mathscr{C}$ be a boundary frame and $\mathcal{C}$-frame respectively. Then $\overline{\Phi}$ in (\ref{rest}) induces the natural map 
	\[
	\Phi:\Hit_n ^{\mathscr{B}}(\Sigma, \mathscr{C})/\R^{m(n-1)} \to   \Hit_n ^{\mathscr{B}_1}(\Sigma_1)\times\cdots \times \Hit_n ^{\mathscr{B}_l}(\Sigma_l)
	\]
	where
	\[
			\mathscr{B}_i=\{(\xi,B)\,|\,\xi\text{ is a component of } \partial \overline{\Sigma_i}\text{ and } (\iota_{\Sigma_i}(\xi),B)\in \mathscr{B}\cup \mathscr{C}\}.
\]
Moreover $\Phi$ is a symplectic diffeomorphism onto an open submanifold of 
\[
\Hit_n ^{\mathscr{B}_1}(\Sigma_1)\times\cdots \times \Hit_n ^{\mathscr{B}_l}(\Sigma_l).
\]
\end{theorem}
\begin{proof}
Since $\R^{m(n-1)}$ acts as conjugation on each $\Gamma_{\Sigma_i}$, $\Phi$ is well-defined.

We now prove that $\Phi$ is symplectic. Since
\[
T_{q[\rho]} \Hit_n ^{\mathscr{B}}(\Sigma,\mathscr{C})/\R^{m(n-1)}\approx \der q ( H^1_{\mathrm{par}} (\Gamma, \mathcal{S};\mathfrak{g}_\rho)),
\]
each element in $T_{q[\rho]} \Hit_n ^{\mathscr{B}}(\Sigma,\mathscr{C})/\R^{m(n-1)}$ can be written as $ \der q ([\alpha])$ for some $[\alpha] \in H^1_{\mathrm{par}}(\Gamma, \mathcal{S};\mathfrak{g}_\rho)$. Moreover, by Proposition \ref{JM2},  we have  
\[
\der (\Phi \circ  q) ([\alpha]) =\der \overline{\Phi}([\alpha])= \iota^* _{\Sigma_1} [\alpha]\oplus \cdots \oplus \iota^* _{\Sigma_l} [\alpha].
\]
Therefore, it follows that 
\[
\Phi^* (\omega^{\Sigma_1} _K \oplus \cdots \oplus \omega^{\Sigma_l} _K ) ( \der q [\alpha], \der q [\beta]) = \sum_{i=1} ^l \omega^{\Sigma_i} _K (\iota^* _{\Sigma_i} [\alpha], \iota^* _{\Sigma_i}  [\beta]). 
\]
By Corollary \ref{localdecomp}, we have
\[
 \sum_{i=1} ^l \omega^{\Sigma_i} _K (\iota^* _{\Sigma_i}  [\alpha], \iota^* _{\Sigma_i}  [\beta])= \omega^\Sigma _K (  [\alpha], [\beta]).
\]
Since $q^* (\widetilde{\omega}_K ^{\Sigma}) = \omega^\Sigma _K$ on $T_{[\rho]} \Hit_n ^{\mathscr{B}}(\Sigma,\mathscr{C})=H^1_{\mathrm{par}}(\Gamma, \mathcal{S};\mathfrak{g}_\rho)$, and since $[\alpha]$ and $[\beta]$ were chosen in $H^1_{\mathrm{par}}(\Gamma, \mathcal{S};\mathfrak{g}_\rho)$, it follows that 
\[
 \omega^\Sigma _K ([\alpha], [\beta])=\widetilde{\omega}^\Sigma _K (\der q[\alpha], \der q[\beta]).
\]
Therefore, $\widetilde{\omega}^\Sigma _K=\Phi^*(\omega_K ^{\Sigma_1}\oplus \cdots \oplus \omega_K ^{\Sigma_l})$ as we wanted.

$\Phi$ is one-to-one. Indeed by  Lemmas \ref{connectedfiber} and \ref{freeaction},  $\R^{m(n-1)}$ acts on each fiber of $\overline{\Phi}$ transitively and freely. Hence $\Phi([\rho_1])= \Phi([\rho_2])$ if and only if $[\rho_1]$ and $[\rho_2]$ are in the same fiber of $\overline{\Phi}$ if and only if there is a unique $\mathbf{t}\in \R^{m(n-1)}$ such that $\mathbf{t}\cdot [\rho_1]=[\rho_2]$. Therefore, $[\rho_1]$ and $[\rho_2]$ represent the same element in $\Hit_n ^{\mathscr{B}}(\Sigma, \mathscr{C})/\R^{m(n-1)}$. 

	We now show that $\Phi$ is an open embedding. Observe that $\delta( \bigoplus_{i=1} ^ m H^0(\Gamma_{\xi_i}; \mathfrak{g}))$ is tangent to the orbits of the $\R^{m(n-1)}$-action. Therefore we have,  by Proposition \ref{JM2},   
\begin{align*}
T_{q[\rho]} \Hit_n ^{\mathscr{B}}(\Sigma,\mathscr{C})/\R^{m(n-1)}&\approx \der q ( H^1_{\mathrm{par}} (\Gamma, \mathcal{S};\mathfrak{g}_\rho))\\ 
&=  H^1_{\mathrm{par}} (\Gamma, \mathcal{S};\mathfrak{g}_\rho)/\delta (\bigoplus_{i=1} ^m H^0(\Gamma_{\xi_i}; \mathfrak{g})) \\
&\approx \bigoplus_{i=1} ^l H^1(\Gamma_{\Sigma_i}, \mathcal{S}_i; \mathfrak{g}).
\end{align*}
Therefore, Proposition \ref{JM2} shows that $\Phi$ has the full rank and that  
\[
\dim \Hit_n ^{\mathscr{B}}(\Sigma, \mathscr{C})/\R^{m(n-1)} =\dim   \Hit_n ^{\mathscr{B}_1}(\Sigma_1)\times\cdots \times \Hit_n ^{\mathscr{B}_l}(\Sigma_l). 
\]
It follows that $\Phi$ is an open embedding. 
\end{proof}

\section{Global Darboux coordinates on $\Hit_3(\Sigma)$}

In this section we prove Theorem \ref{globaldarbouxintro}.  

We first review Goldman's construction of global parametrization on $\Hit_3(\Sigma)$ where $\Sigma$ is a closed surface and then compute $\omega_G$ between some coordinate vector fields. We then construct an $\R^{8g-8}$-valued function and prove that this function satisfies all the conditions of Theorem \ref{existenceofdarboux}.  Corollary \ref{comm} is essentially used in the proof. 

Throughout this section, $\Sigma$  denotes a closed oriented hyperbolic surface unless otherwise stated. 

\subsection{Review on the Goldman coordinates}
Choi-Goldman \cite{choi1993} show that $\Hit_3(\Sigma)$ can be seen as the deformation space of convex projective structures on the surface $\Sigma$. It allows Goldman \cite{goldman1990} to construct  global coordinates of $\Hit_3(\Sigma)$ based on  projective geometry. Let us briefly summarize the construction of Goldman coordinates. 

	Take a maximal collection of disjoint, non-isotopic essential simple closed curves $\mathcal{C}=\{\xi_1, \cdots, \xi_{3g-3}\}$ in  $\Sigma$. This collection $\mathcal{C}$ cuts the surface into $2g-2$ pants $P_1, \cdots, P_{2g-2}$. As mentioned in Lemma \ref{NoInvariantElement}, if $[\rho]\in \Hit_n(\Sigma)$, then each $\rho(\xi_i)$ is in $\mathbf{Hyp}^+$. Therefore, by giving an orientation to each $\xi_i$,  we can associate the following invariants $m_i$ and $\ell_i$ to each oriented simple closed curve $\xi_i$:
	\[
	\ell_i (\rho )= \log  \frac{|\lambda_1(\rho(\xi_i))|}{|\lambda_3(\rho(\xi_i))|}, \quad m_i(\rho)= 3 \log| \lambda_2(\rho(\xi_i))|.
	\]
	Here $\lambda_i$ denotes the $i$th largest eigenvalue. 
	
	 Recall that there is a Hamiltonian $\R^{6g-6}$-action on $\Hit_3(\Sigma)$ with moment map 
	 \[
	 \mu : [\rho] \mapsto (\ell_1(\rho), m_1(\rho), \cdots, \ell_{3g-3}(\rho), m_{3g-3}(\rho)).
	 \]
	 The quotient $q: \Hit_3(\Sigma) \to \Hit_3(\Sigma)/\R^{6g-6}$ is an affine bundle. Recall also that $\Hit_3(\Sigma)$ is foliated by $\bigcup_{y\in \mathbf{L}} \mu^{-1}(y)$. As each $\mu^{-1}(y)$ is invariant under the $\R^{6g-6}$-action, the quotient space $\Hit_3(\Sigma)/\R^{6g-6}$  is also foliated by symplectic manifolds of the form $\mu^{-1}(y)/\R^{6g-6}$. We have seen that each leaf $\mu^{-1}(y)/\R^{6g-6}$ can be identified with the symplectic manifold of the form $\Hit_3 ^{\mathscr{B}^y _1}(P_1)\times \cdots\times \Hit_3 ^{\mathscr{B}^y _{2g-2}}(P_{2g-2})$ via the map $\Phi_y$ in Theorem \ref{gendecomp} where $\mathscr{B}^y _i$ are boundary frames corresponding to $y$. 
	 
	 Goldman \cite{goldman1990} shows that each factor $\Hit_3 ^{\mathscr{B}_i}(P_i)$ is parametrized by two coordinates $(s, t)$.  We can therefore parametrize the quotient $\Hit_3(\Sigma)/\R^{6g-6}$ by $\ell_i$, $m_i$ coordinates together with interior coordinates $s_i$ and $t_i$ defined by
	 \[
	 s_i:=\log s\circ \operatorname{pr}_i\circ \Phi_y,\quad \text{and}\quad t_i:=\log t\circ\operatorname{pr}_i \circ\Phi_y
	 \]
where $\operatorname{pr}_i$ is the projection onto $i$th factor of $\Hit_3 ^{\mathscr{B}^y _1}(P_1)\times \cdots\times \Hit_3 ^{\mathscr{B}^y _{2g-2}}(P_{2g-2})$. 
	 	
	 To  complete our discussion, we have to parametrize the fiber of the affine bundle $q: \Hit_3(\Sigma) \to \Hit_3(\Sigma)/\R^{6g-6}$.  To this end, we need to specify the origin of the affine bundle $\Hit_3(\Sigma) \to \Hit_3(\Sigma)/\R^{6g-6}$. We make the following observation first:
	\begin{lemma}\label{origin}
 Let $([\rho_1], \cdots, [\rho_{2g-2}]) \in \Hit_3 ^{\mathscr{B}_1 ^y}(P_1)\times \cdots \times\Hit_3^{\mathscr{B}^y _{2g-2}}(P_{2g-2})$ be in the image of $\Phi_y$. Then there is a unique $[\rho]\in \mu^{-1}(y)\subset \Hit_3(\Sigma)$ such that $\Phi_y (q([\rho]))=([\rho_1],\cdots,[\rho_{2g-2}])$ and that $\sigma^\rho  _j(\xi_i)=0$ for each $i=1,2,\cdots, 3g-3$ and $j=1,2$. Here $\sigma^\rho  _j(\xi_i) $ are shear invariants of $\xi_i$ given in  (\ref{ClosedLeafInvariant}) of Appendix \ref{BDreview}.
	\end{lemma}
	\begin{proof}
	We have to show the uniqueness of $[\rho]$. Suppose that there is another $[\rho']\in \Hit_3(\Sigma)$ such that $\Phi([\rho'])=([\rho_1], \cdots, [\rho_{2g-2}])$ and that $\sigma^{\rho'}_j(\xi_i) =0$ for all $i=1,2,\cdots, 3g-3$ and $j=1,2$. Then we can find a non-zero vector $\mathbf{t}\in\R^{6g-6}$ such that $\mathbf{t}\cdot[\rho']=[\rho]$. Due to Proposition 5.2 of Bonahon and I. Kim \cite{bonahon2018}, there is a block-diagonal matrix
	\[
	A:=\begin{pmatrix}
	D_1 & 0 &\cdots &0 \\
	0 & D_2 & \cdots &0\\
	0 & 0 &\ddots&0\\
	0& 0& 0 & D_{3g-3}
	\end{pmatrix}, \qquad D_i = \begin{pmatrix} 1 & -3 \\ 1 & 3 \end{pmatrix}
	\]
	 such that  $A \, \mathbf{t}=(\sigma^{\rho}_1(\xi_1),\sigma^\rho _2 (\xi_1), \cdots, \sigma^\rho _1(\xi_{3g-3}),\sigma^\rho _2 (\xi_{3g-3}))^T =\mathbf{0}$. Since $A$ is non-singular, $\mathbf{t}$ must be the null vector which is a contradiction. 
	\end{proof}
	 Therefore, we obtain  a section $\mathfrak{s}: \operatorname{Image} \Phi_y \to  \mu^{-1}(y) \subset \Hit_3(\Sigma)$ of $\Phi_y \circ q|_{\mu^{-1}(y)}$ by assigning  to each $([\rho_1], \cdots, [\rho_{2g-2}])$ in the image of $\Phi_y$, the unique $[\rho]\in \Hit_3(\Sigma)$ constructed in Lemma \ref{origin}. Use the image of $\mathfrak{s}$ as the origins of the action to get the well-defined  twist-bulge parameters $u_i$, $v_i$ to each $\xi_i$. In summary, the global coordinates of $\Hit_3(\Sigma)$ are given by
\[
\{\mathbf{s}_1,\mathbf{t}_1, \cdots, \mathbf{s}_{2g-2},\mathbf{t}_{2g-2}, \ell_1, m_1, \cdots, \ell_{3g-3},m_{3g-3},u_1, v_1, \cdots, u_{3g-3},v_{3g-3}\}
\]
where
\[
\mathbf{s}_i := s_i \circ q,\quad\text{and}\quad\mathbf{t}_i := t_i \circ q.
\]
We have to remark that these coordinates may not be compatible with the symplectic form $\omega_G$. 

\subsection{Proof of Theorem \ref{globaldarbouxintro}}
In this subsection we give a proof of Theorem \ref{globaldarbouxintro}.  We start with some lemmas aiming to apply Theorem \ref{existenceofdarboux} at the end.   

Recall that for a function $f$, we denote by $\mathbb{X}_{f}$ its Hamiltonian vector field.  Recall also that
\[
\frac{\partial}{\partial \mathbf{s}_i} = \der \mathfrak{s} \frac{\partial}{\partial s_i},\quad\text{and}\quad \frac{\partial}{\partial \mathbf{t}_i} = \der \mathfrak{s} \frac{\partial}{\partial t_i}
\]
for each $i=1,2,\cdots, 2g-2$. 

\begin{corollary}[See also \cite{kim1999}] \label{comm}
	Let $\Sigma$ be a closed oriented hyperbolic surface. For any  section $\mathfrak{s}: \operatorname{Image} \Phi \to \Hit_3(\Sigma)$ of $\Phi\circ q:\Hit_3(\Sigma) \to \operatorname{Image} \Phi$ and at any point $[\rho]\in \Hit_n(\Sigma)$, we have: 
	\[
\omega_G\left(\frac{\partial}{\partial \mathbf{s}_i},\frac{\partial}{\partial \mathbf{s}_j}\right)=\omega_G\left(\frac{\partial}{\partial \mathbf{s}_i},\frac{\partial}{\partial \mathbf{t}_j}\right)=\omega_G\left(\frac{\partial}{\partial \mathbf{t}_i},\frac{\partial}{\partial \mathbf{t}_j}\right)=0
\]
	whenever $i\ne j$ and
	\[
	\omega_{G} \left( \frac{\partial}{\partial \mathbf{s}_i} , \frac{\partial}{\partial \mathbf{t}_i}\right) =-1.
	\]
\end{corollary}
\begin{remark}
	The first part of Corollary \ref{comm} is partially proven by H. Kim, see Proposition 6.4 of \cite{kim1999}. He uses the mathematica to prove them. We can obtain the same result and more without mathematica. 
\end{remark}
\begin{proof}
	Enough to consider the case when $P_i$ and $P_j$ are adjacent. 

	Suppose that $[\rho]\in \Hit_n(\Sigma, \mathscr{C})$ for some $\mathscr{C}$. Then $\frac{\partial}{\partial\mathbf{s}_i}=\der\mathfrak{s} \frac{\partial}{\partial s_i}$ and $\frac{\partial}{\partial \mathbf{t}_i}=\der\mathfrak{s} \frac{\partial}{\partial t_i}$ are tangent to $\Hit_n(\Sigma, \mathscr{C})$. Observe that $\omega^\Sigma _K = \omega_G ^\Sigma$ when $\Sigma$ is closed and that $(q^* \widetilde{\omega}_G)|_{\Hit_n(\Sigma, \mathscr{C})}=\omega_G|_{\Hit_n(\Sigma, \mathscr{C})}$.  Theorem \ref{gendecomp} yields 
	\begin{align*}
	\omega_G\left(\der\mathfrak{s}\frac{\partial }{\partial s_i}, \der\mathfrak{s}\frac{\partial }{\partial s_j}\right)&  =\widetilde{\omega}_K \left(\der(q\circ \mathfrak{s})\frac{\partial }{\partial s_i}, \der(q\circ \mathfrak{s})\frac{\partial }{\partial s_j}\right)\\
	& =  \omega_K ^{P_i} \left(\iota_{\Gamma_{P_i}}^*\frac{\partial }{\partial s_i}, \iota_{\Gamma_{P_i}}^* \frac{\partial }{\partial s_j}\right)+\omega_K ^{P_j} \left(\iota_{\Gamma_{P_j}}^*\frac{\partial }{\partial s_i}, \iota_{\Gamma_{P_j}}^*\frac{\partial }{\partial s_j}\right). 
	\end{align*}
	If $i\ne j$, we have $\iota_{\Gamma_{P_j}}^* \frac{\partial}{\partial s_i}=\iota_{\Gamma_{P_i}}^* \frac{\partial}{\partial s_j}=0$ and the result follows. 
	
	Similarly, 
	\begin{align*}
	\omega_G\left(\der\mathfrak{s}\frac{\partial }{\partial s_i}, \der\mathfrak{s}\frac{\partial }{\partial t_j}\right)&  =\widetilde{\omega}_K \left(\der(q\circ \mathfrak{s})\frac{\partial }{\partial s_i}, \der(q\circ \mathfrak{s})\frac{\partial }{\partial t_j}\right)\\
	& =  \omega_K ^{P_i} \left(\iota_{\Gamma_{P_i}}^*\frac{\partial }{\partial s_i}, \iota_{\Gamma_{P_i}}^* \frac{\partial }{\partial t_j}\right)+\omega_K ^{P_j} \left(\iota_{\Gamma_{P_j}}^*\frac{\partial }{\partial s_i}, \iota_{\Gamma_{P_j}}^*\frac{\partial }{\partial t_j}\right)\\
	&=0,
	\end{align*}
	since $\iota_{\Gamma_{P_j}}^* \frac{\partial}{\partial s_i}=\iota_{\Gamma_{P_i}}^* \frac{\partial}{\partial t_j}=0$.

	When $i=j$, we argue in the same fashion:  
	\[
	\omega_G\left(\der\mathfrak{s}\frac{\partial }{\partial s_i}, \der\mathfrak{s}\frac{\partial }{\partial t_i}\right)=\widetilde{\omega}_K \left(\der(q\circ \mathfrak{s})\frac{\partial }{\partial s_i}, \der(q\circ \mathfrak{s})\frac{\partial }{\partial t_i}\right)= \omega_K ^{P_i} \left(\frac{\partial }{\partial s_i}, \frac{\partial }{\partial t_i}\right)=-1.
	\]
	Here $\omega_K ^{P_i} (\frac{\partial }{\partial s_i}, \frac{\partial }{\partial t_i})=-1$ is due to Theorem 5.8 of  H. Kim \cite{kim1999}.
\end{proof}

\begin{lemma}\label{form} For each $i=1,2,\cdots, 2g-2$, the Hamiltonian vector field $\mathbb{X}_{\mathbf{s}_i}$ is of the form
	\[
	\mathbb{X}_{\mathbf{s}_i} = \frac{\partial}{\partial \mathbf{t}_i } + \sum_{j=1} ^{3g-3} \left(a_j \frac{\partial}{\partial u_j} + b_j \frac{\partial}{\partial v_j}\right)
	\]
	for some smooth functions $a_j$ and $b_j$. 
\end{lemma}
\begin{proof}
	The most generic form of the Hamiltonian vector field $\mathbb{X}_{\mathbf{s}_i}$ of $\mathbf{s}_i$ is 
	\begin{multline*}
	\mathbb{X}_{\mathbf{s}_i}= \sum_{j=1} ^{2g-2} a_{\mathbf{s}_j} \frac{\partial}{\partial \mathbf{s}_j}+ \sum_{j=1} ^{2g-2} a_{\mathbf{t}_j} \frac{\partial}{\partial \mathbf{t}_j}+ \sum_{j=1} ^{3g-3} a_{u_j} \frac{\partial}{\partial u_j}+\sum_{j=1} ^{3g-3} a_{v_j} \frac{\partial}{\partial v_j}\\ +\sum_{j=1} ^{3g-3} a_{\ell_j} \frac{\partial}{\partial \ell_j}+\sum_{j=1} ^{3g-3} a_{m_j} \frac{\partial}{\partial m_j}.
	\end{multline*}
	If we compute $\omega_G(\mathbb{X}_{\mathbf{s}_i}, \mathbb{X}_{\ell_k})$, we get
	\[
	\omega_G(\mathbb{X}_{\mathbf{s}_i}, \mathbb{X}_{\ell_k})= \der \mathbf{s}_i \left(\frac{\partial }{\partial u_k} \right)= \frac{\partial \mathbf{s}_i}{\partial u_k} =0.
	\]
	On the other hand, 
	\begin{align*}
	-\omega_G(\mathbb{X}_{\mathbf{s}_i}, \mathbb{X}_{\ell_k}) & =\sum_{j=1} ^{2g-2} a_{\mathbf{s}_j}  \frac{\partial  \ell_k}{\partial \mathbf{s}_j}+ \sum_{j=1} ^{2g-2} a_{\mathbf{t}_j}\frac{\partial  \ell_k}{\partial \mathbf{t}_j}+ \sum_{j=1} ^{3g-3} a_{u_j} \frac{\partial  \ell_k}{\partial u_j}+\sum_{j=1} ^{3g-3} a_{v_j} \frac{\partial  \ell_k}{\partial v_j}\\
	& \qquad+\sum_{j=1} ^{3g-3} a_{\ell_j} \frac{\partial  \ell_k}{\partial \ell_j}+\sum_{j=1} ^{3g-3} a_{m_j} \frac{\partial  \ell_k}{\partial m_j}\\
			&=a_{\ell_k}.
	\end{align*}
	It follows that $\mathbb{X}_{\mathbf{s}_i}$ does not have $\frac{\partial}{\partial \ell_k}$ components,  $k=1,2,\cdots, 3g-3$. Similarly, since $-\omega_G(\mathbb{X}_{\mathbf{s}_i} , \mathbb{X}_{m_k}) = 0 = a_{m_k}$, we can conclude that $\mathbb{X}_{\mathbf{s}_i}$ does not contain $\frac{\partial}{\partial m_k}$, $k=1,2,\cdots, 3g-3$ factors either. Thus,
	\[
	\mathbb{X}_{\mathbf{s}_i} = \sum_{j=1} ^{2g-2} a_{\mathbf{s}_j} \frac{\partial}{\partial \mathbf{s}_j}+ \sum_{j=1} ^{2g-2} a_{\mathbf{t}_j} \frac{\partial}{\partial \mathbf{t}_j}+ \sum_{j=1} ^{3g-3} a_{u_j} \frac{\partial}{\partial u_j}+\sum_{j=1} ^{3g-3} a_{v_j} \frac{\partial}{\partial v_j}.
	\]
	
	We showed in Corollary \ref{comm}  that
	\[
	\omega_G \left(\frac{\partial}{\partial \mathbf{s}_j},\frac{\partial}{\partial \mathbf{s}_k}\right)=0,\quad\text{ and }\quad \omega_G \left(\frac{\partial}{\partial \mathbf{s}_j},\frac{\partial}{\partial \mathbf{t}_k}\right)=\begin{cases} -1& \text{if }j=k \\ 0 & \text{if }j\ne k\end{cases}.
	\]
	Recall also that $\mathbb{X}_{\ell_i}=\frac{\partial}{\partial u_i}$, and $\mathbb{X}_{m_i}=\frac{\partial}{\partial v_i}$. Hence for any $j$ and $k$, 
	\[
	\omega_G\left(\frac{\partial }{\partial u_j}, \frac{\partial }{\partial \mathbf{s}_k}\right) =\omega_G\left(\frac{\partial }{\partial v_j}, \frac{\partial }{\partial \mathbf{s}_k}\right) =\omega_G\left(\frac{\partial }{\partial u_j}, \frac{\partial }{\partial \mathbf{t}_k}\right) =\omega_G\left(\frac{\partial }{\partial v_j}, \frac{\partial }{\partial \mathbf{t}_k}\right) =0.
	\]
	Combining these two results, we get 
	\[
	1=\omega_G \left(\mathbb{X}_{\mathbf{s}_i}, \frac{\partial}{\partial \mathbf{s}_i}\right)= a_{\mathbf{t}_i} \omega_G\left(\frac{\partial}{\partial \mathbf{t}_i} , \frac{\partial}{\partial \mathbf{s}_i}\right) = a_{\mathbf{t}_i}
	\]
	and, whenever $k$ is not equal to $i$,
	\[
	0=\omega_G\left(\mathbb{X}_{\mathbf{s}_i},\frac{\partial}{\partial \mathbf{s}_k}\right)=a_{\mathbf{t}_k} \omega_G\left(\frac{\partial}{\partial \mathbf{t}_k}, \frac{\partial}{\partial \mathbf{s}_k}\right)=a_{\mathbf{t}_k}.
	\]
	Thus it follows that 
	\[
	a_{\mathbf{t}_k}=\begin{cases} 1 & \text{if } k=i\\ 0 & \text{if }k\ne i\end{cases}.
	\]
	In particular $\mathbb{X}_{\mathbf{s}_i}$ does not have $\frac{\partial}{\partial \mathbf{t}_k}$ components for all $k$  different from $i$. 
	
	Finally by computing 
	\[
	0=\frac{\partial \mathbf{s}_i}{\partial \mathbf{t}_k}=\omega_G\left(\mathbb{X}_{\mathbf{s}_i}, \frac{\partial}{\partial \mathbf{t}_k}\right) = a_{\mathbf{s}_k}
	\]
	we can show that $\mathbb{X}_{\mathbf{s}_i}$ does not have  $\frac{\partial}{\partial \mathbf{s}_k}$ components for all $k$. 
\end{proof}
\begin{lemma}\label{complete}
	Each vector field $\mathbb{X}_{\mathbf{s}_i}$, $i=1,2,\cdots, 2g-2$, is complete.
\end{lemma}
\begin{proof} From Lemma \ref{form}, $\mathbb{X}_{\mathbf{s}_i}$ is of the form
	\[
\mathbb{X}_{\mathbf{s}_i} = \frac{\partial}{\partial \mathbf{t}_i }+ \sum_j \left(a_j \frac{\partial}{\partial u_j} + b_j \frac{\partial}{\partial v_j}\right)
\]
We investigate the coefficient functions $a_i$, $b_i$. 
	Observe  that
	\[
	\omega_G\left(\mathbb{X}_{\mathbf{s}_i}, \frac{\partial}{\partial \ell_j}\right)=\frac{\partial \mathbf{s}_i}{\partial \ell_j}=0.
	\]
	On the other hand
	\[
	\omega_G\left(\mathbb{X}_{\mathbf{s}_i}, \frac{\partial}{\partial \ell_j}\right)=\omega_G\left(\frac{\partial }{\partial \mathbf{t}_i}, \frac{\partial}{\partial \ell_j}\right)+a_{j}.
	\]
	Therefore
	\[
	a_{j} =  -\omega_G\left(\frac{\partial}{\partial \mathbf{t}_i}, \frac{\partial}{\partial \ell_j}\right).
	\]
	Similarly,
	\[
	\omega_G\left(\mathbb{X}_{\mathbf{s}_i}, \frac{\partial}{\partial m_j}\right)=\frac{\partial \mathbf{s}_i}{\partial m_j}=\omega_G \left(\frac{\partial}{\partial \mathbf{t}_i}, \frac{\partial}{\partial m_j}\right)+b_j=0
	\]
	shows that
	\[
	b_j =- \omega_G\left(\frac{\partial}{\partial \mathbf{t}_i}, \frac{\partial}{\partial m_j}\right).
	\]
	Since $\mathbb{X}_{\ell_k}= \frac{\partial}{\partial u_k}$ is a Hamiltonian vector field as well as a coordinate vector field, we have 
	\[
	0=(\mathcal{L}_{\mathbb{X}_{\ell_k}} \omega_G) \left(\frac{\partial}{\partial \mathbf{t}_i}, \frac{\partial}{\partial \ell_j}\right) = \mathbb{X}_{\ell_k}\omega_G \left(\frac{\partial}{\partial \mathbf{t}_i},\frac{\partial}{\partial \ell_j}\right)=-\frac{\partial a_j}{\partial u_k}
	\]
	which yields that functions $a_{1},\cdots,a_{3g-3}$ do not depend on  $u_1, \cdots, u_{3g-3}$. Same argument using $\mathbb{X}_{m_k}$ instead of $\mathbb{X}_{\ell_k}$ shows that $a_{1},\cdots,a_{3g-3}$ do not depend on the $v_1,\cdots, v_{3g-3}$ variables either. Similarly, $b_{1},\cdots,b_{3g-3}$ are functions depending only on $\mathbf{s}_i,\mathbf{t}_i, \ell_i$ and $m_i$. 
	
The equation $\dot{\mathbf{x}}(t) = \mathbb{X}_{\mathbf{s}_i} (\mathbf{x}(t))$ for an integral curve reads
	\begin{gather}
	\frac{\der \mathbf{s}_{j}(t)}{\der t} = 0, \quad j=1,2,\cdots, 2g-2 \label{flow1}\\
	\frac{\der \mathbf{t}_j (t)}{\der t}= 0, \quad j=1,2,\cdots, 2g-2, \, j\ne i \\
	\frac{\der \mathbf{t}_i (t)}{\der t}=1\\
	\frac{\der \ell_j (t)}{\der t} = 0, \quad j=1,2,\cdots, 3g-3\\
	\frac{\der m_j (t)}{\der t}=0, \quad j=1,2,\cdots, 3g-3 \label{flow2}\\
	\frac{\der u_j (t)}{\der t}=a_j, \quad j=1,2,\cdots, 3g-3 \label{flow3}\\
	\frac{\der v_j (t)}{\der t}=b_j,  \quad j=1,2,\cdots, 3g-3.\label{flow4}
	\end{gather}
	A solution for equations (\ref{flow1})-(\ref{flow2}) is 
	\begin{equation}\label{sol}
	\begin{cases}
	\mathbf{t}_i =t  + \text{const.}& \\
	\mathbf{t}_j = \text{const.} & j=1,2,\cdots, 2g-2,j \ne i\\
	\mathbf{s}_j= \text{const.} & j=1,2,\cdots, 2g-2\\
	\ell_j,m_j = \text{const.} & j =1,2,\cdots, 3g-3
	\end{cases}.
	\end{equation}
	Having the fact that $a_j$ and $b_j$ are functions of $\mathbf{s}_i, \mathbf{t}_i, \ell_i, m_i$ in mind, plug  the solution (\ref{sol})  into $a_j$ and $b_j$. Then  $a_j$ and $b_j$ become purely smooth functions of the time $t$. It means that the equations (\ref{flow3}), (\ref{flow4}) have a solution for all $t$. Therefore the vector field $\mathbb{X}_{\mathbf{s}_i}$ for each  $i=1,2,\cdots, 2g-2$ is complete.
\end{proof}

We define a function $F:\Hit_3(\Sigma) \to \R^{8g-8}$ to be
\[
F([\rho])= (\mathbf{s}_1(\rho), \cdots, \mathbf{s}_{2g-2}(\rho), \ell_1(\rho),\cdots,\ell_{3g-3} (\rho), m_1(\rho),\cdots, m_{3g-3}(\rho)).
\]
\begin{lemma}\label{lagrangian}
	 For each $x\in \operatorname{Image} F\subset \R^{8g-8}$, the fiber $F^{-1}(x)$ is a simply connected Lagrangian submanifold. 
\end{lemma}
\begin{proof}
Tangent space at each point of $F^{-1}(x)$ is spanned by vectors 
\[
\mathbb{X}_{\mathbf{s}_1},\cdots, \mathbb{X}_{\mathbf{s}_{2g-2}}, \mathbb{X}_{\ell_1}, \cdots, \mathbb{X}_{\ell_{3g-3}}, \text{ and } \mathbb{X}_{m_1}, \cdots, \mathbb{X}_{m_{3g-3}}.
\]
By Corollary \ref{comm} and Lemma \ref{form}, $F^{-1}(x)$ is a Lagrangian submanifold.

By Lemma \ref{complete},  $\mathbb{X}_{\mathbf{s}_1},\cdots, \mathbb{X}_{\mathbf{s}_{2g-2}}$, $\mathbb{X}_{\ell_1},\mathbb{X}_{m_1},\cdots,\mathbb{X}_{\ell_{3g-3}} , \mathbb{X}_{m_{3g-3}}$  are commuting complete vector fields tangent to each fiber $F^{-1}(x)$. Thus, the Hamiltonian flows of $\mathbf{s}_1, \cdots, \mathbf{s}_{2g-2}, \ell_1,m_1, \cdots, \ell_{3g-3},m_{3g-3}$  induce an $\R^{8g-8}$-action on  $F^{-1}(x)$. By Lemma \ref{form}, this action is free and transitive on each fiber $F^{-1}(x)$. Therefore, each fiber is diffeomorphic to $\R^{8g-8}$ which is simply connected.
%
\end{proof}

Now we can prove Theorem \ref{globaldarbouxintro}. Let $B$ be the image of the function $F: \Hit_3(\Sigma) \to \R^{8g-8}$. According to section 1.8 of Goldman \cite{goldman1990}, $B$ is diffeomorphic to $\R ^{2g-2}\times \mathfrak{R}$ and 
 \[
 \mathfrak{R}=\{(\ell_1, \cdots, \ell_{3g-3}, m_1, \cdots, m_{3g-3})\in \R_+^{3g-3}\times \R^{3g-3}\,|\,|m_i|< \ell_i \}
 \]
 where $\R_+=\{x\in \R\,|\, x>0\}$. In particular, $B$ is contractible. In addition to this, due to  Lemma \ref{complete} and  Lemma \ref{lagrangian},  we observe that $F: \Hit_3(\Sigma) \to \R^{8g-8}$ satisfies the conditions of Theorem \ref{existenceofdarboux}. Therefore the result follows from Theorem \ref{existenceofdarboux}.

\appendix
\section{More details on group cohomology}
This appendix is devoted to prove Propositions \ref{mvs} and \ref{tangent}. Let us recall our settings. $\Sigma$ is a compact hyperbolic surface with boundary $\zeta_1, \cdots, \zeta_b$. We choose a collection of pairwise disjoint, non-isotopic essential simple closed curves $\xi_1, \cdots, \xi_m$ that separate $\Sigma$ into subsurfaces $\Sigma_1, \cdots, \Sigma_l$ which are hyperbolic. We use the notation in section \ref{tree} and section \ref{MVS}. We denote by $\Gamma_{\Sigma_i}$ the conjugacy class of the subgroup $\pi_1(\Sigma_i)$  in $\pi_1(\Sigma)$ and by $\Gamma_{\xi_i}$ the conjugacy class of $\pi_1(\xi_i)$ in $\pi_1(\Sigma)$. Let $\mathcal{S} = \{\Gamma_{\xi_1}^+, \cdots, \Gamma_{\xi_m}^+, \langle \zeta_1 \rangle, \cdots, \langle \zeta_b \rangle\}$ and $\mathcal{S}_i =\{ \langle \zeta \rangle \,|\, \zeta \text{ is a component of } \partial \overline{\Sigma_i}  \}$,  $i=1,2,\cdots, l$. Then $(\Gamma_{\Sigma_i}, \mathcal{S}_i)$ are group subsystem of $(\Gamma, \mathcal{S})$.

 We give a CW-structure on $\Sigma$ as follows:
\begin{itemize}
	\item Choose points $p_1, \cdots, p_l$ on the interior  of each $\Sigma_1, \cdots, \Sigma_l$, $q_1, \cdots, q_m$  on $\xi_1, \cdots, \xi_m$ and $q_{m+1}, \cdots, q_{m+b}$ on  $\zeta_1, \cdots, \zeta_b$. They are 0-cells
	\item Each $\xi_i$ and $\zeta_i$ is adjacent to at most two $\Sigma_j$. If $\xi_i$ is adjacent to say $\Sigma_1$ and $\Sigma_2$, we connect $q_i$ to each $p_1$ and $p_2$ by path $\eta_{1,i}$ and $\eta_{2,i}$ in $\Sigma_1$ and $\Sigma_2$ respectively. Do the same thing for each $\zeta_i$.   On each $\Sigma_i$, we choose simple closed curves $x_{i,1},y_{i,1}, \cdots, x_{i,g_i}, y_{i,g_i}$ based at $p_i$ in such a way that they are not intersecting and $\Sigma_i\setminus (\partial\Sigma_i \cup \bigcup_j \eta_{i,j} \cup \bigcup_j (x_{i,j}\cup y_{i,j}))$ is a disk. $x_{i,j}, y_{i,j}, \eta_{i,j}, \xi_i, \zeta_i$ are 1-cells
	\item 2-cells are disks corresponding to each $\Sigma_i$. 
\end{itemize}
We can check that this CW-structure has the following properties:
\begin{itemize}
	\item The natural inclusion $\Sigma_i\to \Sigma$, $\xi_i \to \Sigma$ and $\zeta_i \to \Sigma$ are cellular. 
	\item Let $p:\widetilde{\Sigma} \to \Sigma$ be the universal cover. CW-structure of $\Sigma$ lifts to a CW-structure of $\widetilde{\Sigma}$. We observe that the natural inclusion  $p^{-1} (\Sigma_i) \to \widetilde{\Sigma}$ is also cellular. 
	\item $\bigcup_{i=1} ^m p^{-1} (\xi_i)$ decomposes $\widetilde{\Sigma}$ into contractible regions.  Each component corresponds to the universal cover of $\Sigma_i$ for some $i$. 
\end{itemize}

Denote by $C_\ast (\widetilde{\Sigma})$ the cellular chain complex over $\R$ for $\widetilde{\Sigma}$ associated to this CW-structure. We know that $\Gamma$ acts on each $C_i(\widetilde{\Sigma})$ as covering transformations so that  $C_i(\widetilde{\Sigma})$ becomes a $\R\Gamma$-module. In fact, $C_i(\widetilde{\Sigma})$ is a free $\R\Gamma$-module.  Since $\widetilde{\Sigma}$ is contractible, the cellular chain complex $C_\ast (\widetilde{\Sigma})$ is exact. The augmented complex $C_\ast(\widetilde{\Sigma})\to \R$  is therefore a free resolution over $\Gamma$.

As we did previously,  $\Gamma_{\Sigma_i}$ action on $\widetilde{\Sigma_i}$ turns  $C_\ast(\widetilde{\Sigma_i})$ into a complex of $\R\Gamma_{\Sigma_i}$-modules. Since $\widetilde{\Sigma_i}$ is also contractible, $C_\ast(\widetilde{\Sigma_i})\to \R$ is a free resolution over $\Gamma_{\Sigma_i}$.  Moreover, we see that $\R\Gamma \otimes C_\ast (\widetilde{\Sigma_i}) \approx C_\ast (p^{-1}(\Sigma_i))$. Since $\widetilde{\Sigma_i}$ is again a subcomplex of $\widetilde{\Sigma}$, the natural inclusion induces a chain map (as $\R\Gamma$-modules) $C_\ast (p^{-1}(\Sigma_i)) \to C_\ast (\widetilde{\Sigma})$. Therefore, we get the natural surjective chain map
\[
\bigoplus_{i=1} ^l C_\ast (p^{-1}(\Sigma_i)) \to C_\ast (\widetilde{\Sigma})\to 0.
\]

We do the same thing for $\Gamma_{\xi_i}$. If $\xi_i$  connects $\mathrm{o}(\xi_i)=\Sigma_a$ and $\mathrm{t}(\xi_i)=\Sigma_b$,  we have a natural chain map from $C_\ast(p^{-1}(\xi_i))$ to $ C_\ast(p^{-1}(\Sigma_a))$ and to $C_\ast (p^{-1}(\Sigma_b))$ both are induced by the inclusion. Then 
\[
(\iota^\# , -\iota^\#): C_\ast(p^{-1}(\xi_i)) \to C_\ast(p^{-1}(\Sigma_a))\oplus C_\ast (p^{-1}(\Sigma_b))
\]
is also injective chain map. So we have the  exact sequence of chain maps
\[
0\to \bigoplus _{i=1} ^m C_\ast (p^{-1}(\xi_i)) \to \bigoplus_{i=1} ^l C_\ast (p^{-1}(\Sigma_i)).
\]
Two maps fit into the exact sequence of complexes
\[
0\to \bigoplus_{i=1}^m C_\ast(p^{-1}(\xi_i)) \to \bigoplus_{i=1} ^l C_\ast (p^{-1}(\Sigma_i)) \to C_\ast (\widetilde{\Sigma})\to 0
\]
of $\R \Gamma$-modules. We apply $\Hom_\Gamma(-,\mathfrak{g}_\rho)$ functor to compute the group cohomology where $\mathfrak{g}$ is a $\R\Gamma$-module via the $\Ad \rho$ action. Let $C^\ast(\Gamma;\mathfrak{g}) = \Hom_\Gamma(C_\ast(\widetilde{\Sigma}), \mathfrak{g})$. 

The natural map $\Hom_\Gamma (C_\ast(p^{-1}(\Sigma_i)), \mathfrak{g}) \to \Hom_{\Gamma_{\Sigma_i}} (C_\ast(\widetilde{\Sigma_i}), \mathfrak{g})$  that sends $f$ to a homomorphism $a\mapsto f(1\otimes a)$ is an isomorphism of chain complexes of $\R$-vector spaces. Therefore the middle chain computes $\bigoplus_{i=1} ^l H^q(\Gamma_{\Sigma_i},\mathfrak{g})$. Similarly, $\Hom_{\Gamma} ( C_\ast(p^{-1}(\xi_i)),\mathfrak{g})\approx  \Hom_{\Gamma_{\xi_i}} (C_\ast (\widetilde{\xi_i}) , \mathfrak{g})$ computes $H^q(\Gamma_{\xi_i}, \mathfrak{g})$. This proves that the following sequence 
\begin{equation}\label{ordinarymvs}
0 \to \bigoplus_{i=1} ^{m}  H^0(\Gamma_{\xi_i}; \mathfrak{g}) \overset{\delta}{\to} H^1(\Gamma; \mathfrak{g}) \to \bigoplus_{i=1} ^l H^1(\Gamma_{\Sigma_i}; \mathfrak{g}) \to \bigoplus_{i=1} ^{m} H^1(\Gamma_{\xi_i}; \mathfrak{g}) \to 0 
\end{equation}
is exact. 

To prove  Theorem \ref{mvs}, we have to characterize the parabolic cohomology in terms of cocycles.  We state this as the following lemma.

\begin{lemma}\label{paraboliccocycle}
	Let 
	\begin{multline*}
	Z^1 _{\mathrm{par}}(\Gamma,\mathcal{S};\mathfrak{g})\\=\{f\in Z^1(\Gamma;\mathfrak{g})\,|\,f(\xi)=0\text{ for all }\xi\in \bigoplus_{i=1} ^mC_1(p^{-1}(\xi_i))\oplus \bigoplus_{i=1} ^b C_1(p^{-1}(\zeta_i))\}.
	\end{multline*}
	Then 
	\[
	H^1_{\mathrm{par}}(\Gamma, \mathcal{S};\mathfrak{g})=\frac{Z^1 _{\mathrm{par}} (\Gamma, \mathcal{S}; \mathfrak{g})}{B^1(\Gamma; \mathfrak{g}) \cap Z^1_{\mathrm{par}} (\Gamma, \mathcal{S}; \mathfrak{g})}. 
	\]
\end{lemma}
\begin{proof}
We see that $(C_\ast (\widetilde{\Sigma}),C_\ast (\widetilde{\xi_i})\oplus C_\ast(\widetilde{\zeta_i}))$ is an auxiliary resolution over the group system $(\Gamma, \mathcal{S})$. Now the lemma follows from the definition of parabolic cohomology. 
\end{proof}

It follows that the map 
\[
\bigoplus_{i=1} ^l  H^1_{\mathrm{par}}( \Gamma_{\Sigma_i}, \mathcal{S}_i;\mathfrak{g})\to \bigoplus_{i=1} ^m  H^1(\Gamma_{\xi_i};\mathfrak{g})
\]
 in (\ref{ordinarymvs}) is trivial. Suppose that an element $([\alpha_1], \cdots, [\alpha_l])$  in $\bigoplus Z^1 _{\mathrm{par}} (\Gamma_{\Sigma_i}, \mathcal{S}_i;\mathfrak{g})$ is given. For each $[\alpha_i]$, we define a parabolic cocycle $\overline{\alpha_i}\in Z^1_{\mathrm{par}} (\Gamma,\mathcal{S};\mathfrak{g})$ by 
 \[
 \overline{\alpha_i} (x) =\begin{cases}
 \alpha_i (x) & x\in C_1(p^{-1}(\Sigma_i) )\\
 0 & \text{otherwise}
 \end{cases}. 
 \]
 Then $[\overline{\alpha_1}]+\cdots+ [\overline{\alpha_l}]$ is an element of $H^1_{\mathrm{par}} (\Gamma,\mathcal{S};\mathfrak{g})$ which maps to $([\alpha_1], \cdots, [\alpha_l])$. Thus
 \[
H^1_{\mathrm{par}}(\Gamma,\mathcal{S};\mathfrak{g})\to \bigoplus_{i=1} ^l  H^1_{\mathrm{par}}( \Gamma_{\Sigma_i}, \mathcal{S}_i;\mathfrak{g})
 \]
 is surjective.  Finally the exactness at $H^1_{\mathrm{par}} (\Gamma,\mathcal{S};\mathfrak{g})$ and  $\bigoplus H^0(\Gamma_{\xi_i};\mathfrak{g})$ follow from the exactness of (\ref{ordinarymvs}) and the fact that a parabolic cohomology group is a subgroup of the ordinary cohomology.

\begin{proposition}\label{tangentA}
	Let $[\rho]\in\overline{\Rep}_n ^{\mathscr{B}}(\Gamma,\mathscr{C})$. We have
	\[
	T_{[\rho]}\overline{\Rep}_n ^{\mathscr{B}}(\Gamma,\mathscr{C}) \approx H^1_{\mathrm{par}} (\Gamma,\mathcal{S};\mathfrak{g}_{\rho}). 
	\]
\end{proposition}
\begin{proof}
We use the discrete connection model to describe the space of representations. Namely, we see the 1-skeleton $\Sigma^{(1)}$ of $\Sigma$ as a graph and consider a discrete flat connection over the graph, which is the rule that assigns to each edge an element of $G$ in such a way that the holonomy around each 2-cell is trivial. In this setting $\overline{\Rep}_n ^{\mathscr{B}}(\Gamma,\mathscr{C})$ is an open subset of $R^{-1}(e)$ modulo the gauge 
\[
\mathcal{G}=G^{|\Sigma^{(0)}\setminus\{q_1, \cdots, q_{m+b}\}|}\times Z_{G}(\rho(\xi_1))\times \cdots \times Z_G(\rho(\xi_m))\times Z_G(\rho(\zeta_1))\times \cdots \times Z_G(\rho(\zeta_b))
\] 
where $R: G^{|E(\Sigma^{(1)})\setminus S|}\to G^{l}$  is the curvature and 
\[
S=\{\xi_1, \cdots, \xi_m,\zeta_1, \cdots, \zeta_b\}\subset E(\Sigma^{(1)}).
\]
See  Labourie \cite{labourie2013} for more details.

Then we have the following commutative diagram 
\[
\xymatrix{
C^0 _{\mathrm{par}} (\Gamma, \mathcal{S};\mathfrak{g}) \ar[d]^{=} \ar[r]^{\der} & C^1 _{\mathrm{par}} (\Gamma, \mathcal{S};\mathfrak{g}) \ar[r]^{\der} \ar[d]& C^2 _{\mathrm{par}} (\Gamma, \mathcal{S}; \mathfrak{g})\ar[d]^{=} \\
T_e \mathcal{G} \ar[r] & T_{[\rho]}G^{|E(\Sigma^{(1)})\setminus S|} \ar[r]_{\der R} & T_e G^{l}
}
\]
where 
\begin{align*}
C^0 _{\mathrm{par}} (\Gamma, \mathcal{S};\mathfrak{g})&= \{ f\in C^0(\Gamma;\mathfrak{g})\,|\, \der f \in Z^1 _{\mathrm{par}} (\Gamma, \mathcal{S};\mathfrak{g})\},\\
C^1_{\mathrm{par}}(\Gamma, \mathcal{S}; \mathfrak{g})& = \{f\in C^1(\Gamma; \mathfrak{g})\,|\, f(\xi) = 0 \\
&\qquad \quad \text{ for all } \xi\in \bigoplus_{i=1} ^mC_1(p^{-1}(\xi_i))\oplus \bigoplus_{i=1} ^b C_1(p^{-1}(\zeta_i))\},\text{ and } \\
C^2_{\mathrm{par}} (\Gamma, S; \mathfrak{g}) &= C^2(\Gamma; \mathfrak{g}).
\end{align*}
It is then clear from Lemma \ref{paraboliccocycle}, that the tangent space is $H^1_{\mathrm{par}}(\Gamma, \mathcal{S};\mathfrak{g})$. 
\end{proof}

\begin{lemma}\label{relativefundclassappendix}
 Let 
\[
	[\Sigma] = \sum_{i=1} ^g\left(\, \left[ \left. \frac{\partial R} {\partial x_i} \right| x_i\right] + \left[\left. \frac{\partial R}{\partial y_i} \right| y_i\right]\,\right) + \sum_{j=1} ^ b  \left[\left. \frac{\partial R}{\partial z_j}\right| z_j\right]
\]
 be a (absolute) 2-chain in $\mathbf{B}_2(\Gamma)\otimes \Z$. Let $(R_\ast, A^i _\ast)$ be an auxiliary resolution over the group system  $(\Gamma, \mathcal{S})$ constructed in the proof of Lemma \ref{paraboliccocycle}. Choose a chain equivalence $\mathbf{B}_\ast (\Gamma)\otimes \Z \to R_\ast \otimes \Z$. Then  the  image of $[\Sigma]$ under the map 
\[
	\mathbf{B}_2(\Gamma)\otimes \Z \to R_2 \otimes \Z \to (R_2/A_2) \otimes \Z
\]
represents a generator of $H_2(\Gamma,\mathcal{S}; \Z)\approx \Z$. 
\end{lemma}
\begin{proof}
Let $\langle \gamma\in \Gamma\setminus \{1\}  \,|\, \gamma_1\cdot  \gamma_2 = \gamma_1\gamma_2\rangle$ be a  `tautological' presentation for $\Gamma$ and let $X_0$ be a presentation complex for this presentation. Let  $X$ be a  $K(\Gamma, 1)$ space obtained by attaching appropriate cells to $X_0$. Observe that the cellular chain complex of $\widetilde{X}$ corresponds to the normalized bar resolution $\mathbf{B}_\ast(\Gamma)$ over $\Gamma$. The chain equivalence $\mathbf{B}_\ast (\Gamma)\otimes \Z \to R_\ast \otimes \Z$ is given by the homotopy equivalence between $K(\Gamma, 1)$ spaces $X$ and $\Sigma$.  This chain equivalence sends $[ \Sigma ]$ to the sum of all 2-cells of $\Sigma$. Certainly its image in $(R_2/A_2)\otimes \Z$ represents the generator of $H_2(\Gamma,\mathcal{S};\Z)$.
\end{proof}

\section{Review on  the Bonahon-Dreyer coordinates}\label{BDreview}
In this appendix, we review Bonahon-Dreyer's parametrization of $\Hit_n(\Sigma)$. Complete discussion on this coordinates can be found in their original paper \cite{bonahon2014}. See also \cite{bonahon2018} for the relationship between coordinates of Bonahon-Dreyer and that of Goldman.

For the remaining of this appendix, we assume that $\Sigma$ is a closed oriented surface of genus $g>1$. We identify the universal cover $\widetilde{\Sigma}$ of $\Sigma$  with the Poincar\'e disk $\mathbb{H}^2$ with the ideal boundary $\partial_\infty \widetilde{\Sigma}$ homeomorphic to the circle. 

To parametrize $\Hit_n(\Sigma)$, we have to fix some topological data
\begin{itemize}
\item We equip $\Sigma$ with an auxiliary hyperbolic metric and take a maximal geodesic lamination $\Lambda$ on $\Sigma$.
\item We give an arbitrary orientation on each leaf.
\item To each closed leaf, we choose a ``short'' transverse arc intersecting the leaf exactly once.
\end{itemize}
Observe that since $\Lambda$ is maximal,  $\Sigma\setminus \Lambda$ consists of $4g-4$ triangles $T_1, T_2 \cdots, T_{4g-4}$. 

The starting point of this parametrization is the following characterization theorem due to Labourie. 
\begin{theorem}[Labourie \cite{labourie2006}]
A representation $\rho:\pi_1(\Sigma) \to \PSL_n(\R)$ is Hitchin if and only if there is a $\rho$-equivariant flag curve $\mathcal{F}_\rho:\partial_\infty \widetilde{\Sigma}\to \Flag(\R^n)$ such that $\mathcal{F}_\rho ^{(1)}: \partial_\infty \widetilde{\Sigma} \to P \R^n$ is hyperconvex and Frenet. 
\end{theorem}
Therefore if $\rho\in \Hit_n(\Sigma)$ is given, each point of $\partial_\infty\widetilde{\Sigma}$ is decorated by a flag. We define three types of invariants associated to each ideal triangle, bi-infinite leaf and closed leaf. 

\begin{figure}[ht]
\begin{tikzpicture}
\draw (0,1)--(-1,-.6)--(1,-.6)--cycle;
\draw (0,1) node[above]{$v_1$};
\draw (-1,-0.6) node[left]{$v_3$};
\draw (1,-0.6) node[right]{$v_2$};
\draw(0,0) node{$\widetilde{T}_\ell$};
\end{tikzpicture}
\caption{The local configuration of $\widetilde{T_\ell}$}\label{triangle}
\end{figure}
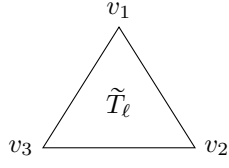

For each ideal triangle $T_\ell$, we consider its lift $\widetilde{T}_\ell$ in the universal cover. Let  $v_1, v_2, v_3$ be ideal vertices of $\widetilde{T}_\ell$  in clockwise cyclic order as in Figure \ref{triangle}. Then by Fock-Goncharov \cite{fock2006}, $\mathcal{F}_\rho(v_1)=:A$, $\mathcal{F}_\rho(v_2)=:B$, and $\mathcal{F}_\rho(v_3)=:C$ is a positive triple in the sense that the following \emph{triangle invariants}
\begin{align*}
\tau^\rho _{i,j,k}(T_\ell ,v_1) &:= \log T_{i,j,k} (A,B,C), \\
\tau^\rho _{i,j,k}(T_\ell ,v_2) &:= \log T_{i,j,k} (B,C, A), \quad \text{and}\\
\tau^\rho _{i,j,k}(T_\ell ,v_3) &:= \log T_{i,j,k} (C, A,B)
\end{align*}
for each $i,j,k>0$ with $i+j+k=n$ are well-defined. Here, 
\begin{multline*}
T_{i,j,k} (A,B,C) :=  \frac{A^{(i+1)} \wedge B^{(j)}\wedge C^{(k-1)}}{A^{(i-1)}\wedge B^{(j)}\wedge C^{(k+1)}}\times  \\ \frac{A^{(i)} \wedge B^{(j-1)}\wedge C^{(k+1)}}{A^{(i)}\wedge B^{(j+1)}\wedge C^{(k-1)}} \times \frac{A^{(i-1)} \wedge B^{(j+1)}\wedge C^{(k)}}{A^{(i+1)}\wedge B^{(j-1)}\wedge C^{(k)}}.
\end{multline*}
These invariants $\tau^\rho _{i,j,k}(T_\ell, v_1)$, $\tau^\rho _{i,j,k}(T_\ell, v_2)$,  and $\tau^\rho _{i,j,k}(T_\ell, v_3)$ are subject to the following rotation condition
\[
\tau_{i,j,k} ^\rho (T_\ell, v_1)= \tau^\rho _{j,k,i}(T_\ell, v_2) =\tau^\rho _{k,i,j}(T_\ell, v_3). 
\]

\begin{figure}[ht]
\begin{tikzpicture}
\draw (0,1.5)--(-1.5,0)--(0,-1.5)--(1.5,0)--(0,1.5)--(0,-1.5) ;
\draw (0,1.5) node[above]{$v_+$};
\draw[->,thick] (0,-.1)--(0,0);
\draw (0,-1.5) node[below]{$v_-$};
\draw (1.5,0) node[right]{$v_R$};
\draw (0,0) node[right]{$\widetilde{\gamma_\ell}$};
\draw (-1.5,0) node[left] {$v_L$};
\draw (-.5,0) node {$\widetilde{T}_L$};
\draw (.8,0) node {$\widetilde{T}_R$};
\end{tikzpicture}
\caption{The local configuration near $\widetilde{\gamma}_\ell$}\label{infquad}
\end{figure}
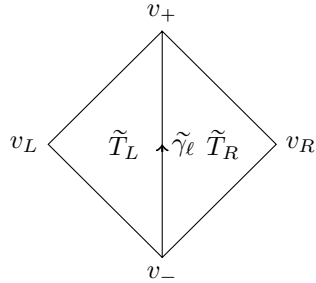

Consider an infinite leaf $\gamma_\ell$ and its lift $\widetilde{\gamma_\ell}$. There are two ideal triangles $\widetilde{T}_L$ and $\widetilde{T}_R$ sharing the edge $\widetilde{\gamma_\ell}$.  In the universal cover, it looks like Figure \ref{infquad}. Again, by Fock-Goncharov \cite{fock2006},   $\mathcal{F}_\rho(v_+)=:E$, $\mathcal{F}_\rho(v_-)=:F$, $\mathcal{F}_\rho(v_R)=:Y$ and $\mathcal{F}_\rho(v_L)=:X$ is a positive quadruple in the sense that the following \emph{shear invariants} of the infinite leaf $\gamma_\ell$
\[
\sigma_i^\rho (\gamma_\ell) := \log D_i (E,F,X,Y), \quad i=1,2,\cdots,n-1
\]
are well-defined where
\[
D_i (E,F,X,Y) := - \frac{E^{(i)}\wedge F^{(n-1-i)}\wedge X^{(1)}}{E^{(i)}\wedge F^{(n-1-i)}\wedge Y^{(1)}}\times \frac{E^{(i-1)}\wedge F^{(n-i)}\wedge Y^{(1)}}{E^{(i-1)}\wedge F^{(n-i)}\wedge X^{(1)}}.
\]

\begin{figure}[ht]
\begin{tikzpicture}
\draw (0,1.5)--(-1,-1)--(-0.5,-1.4)--cycle ;
\draw (-1,-1) node[left]{$v_L$};
\draw (0,1.5) -- (-0.3, -1.45);
\draw (0,1.5) -- (-0.2, -1.47);
\draw[->, very thick] (0,-.1)--(0,0);
\draw (0,1.5) -- (-0.1, -1.49);
\draw (.4,-.2)--(-.4,-.2);
\draw (0,1.5)--(1,-1)--(0.5,-1.4)--cycle ;
\draw (1,-1) node[right]{$v_R$};
\draw (0,1.5) -- (0.3, -1.45);
\draw (0,1.5) -- (0.2, -1.47);
\draw (0,1.5) -- (0.1, -1.49);
\draw[thick] (0,1.5) -- (0, -1.5);
\draw (0,1.5) node[above]{$v_+$};
\draw (0,-1.5) node[below]{$v_-$};
\draw (0,0) node[right]{$\widetilde{c}_\ell$};
\end{tikzpicture}
\begin{tikzpicture}
\draw (0,1.5)--(-1,-1)--(-0.5,-1.4)--cycle ;
\draw (-1,-1) node[left]{$v_L$};
\draw (0,1.5) -- (-0.3, -1.45);
\draw (0,1.5) -- (-0.2, -1.47);
\draw (0,1.5) -- (-0.1, -1.49);
\draw (.4,-.2)--(-.4,-.2);
\draw[->, very thick] (0,-.1)--(0,0);
\draw (0,-1.5)--(1,1)--(0.5,1.4)--cycle ;
\draw (1,1) node[right]{$v_R$};
\draw (0,-1.5) -- (0.3, 1.45);
\draw (0,-1.5) -- (0.2, 1.47);
\draw (0,-1.5) -- (0.1, 1.49);
\draw[thick] (0,1.5) -- (0, -1.5);
\draw (0,1.5) node[above]{$v_+$};
\draw (0,-1.5) node[below]{$v_-$};
\draw (0,0) node[right]{$\widetilde{c}_\ell$};
\end{tikzpicture}
\caption{The local configuration near $\widetilde{c}_\ell$. Note that there are two possibilities.}\label{closedquad}
\end{figure}
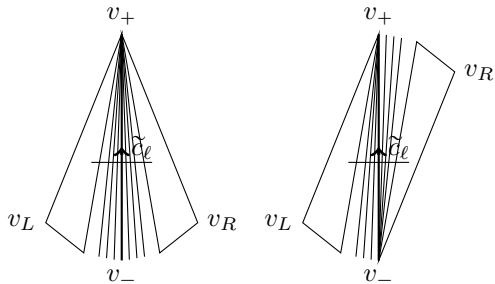

Lastly we consider a closed leaf $c_\ell$. For each $c_\ell$ there are infinite leaves that spiral to it. In the universal cover, we have configurations as in Figure \ref{closedquad}. We use small transverse arc, which is one of our topological data, to choose two ideal triangles $\widetilde{T_L}$ and $\widetilde{T_R}$ on each side of $\widetilde{c_\ell}$.  Let $v_L$ be the ideal vertex of $\widetilde{T_L}$ which is farthest from $\widetilde{c_\ell}$. We choose $v_R$ in the same fashion. Then $\mathcal{F}_\rho(v_+)=:E$, $\mathcal{F}_\rho(v_-)=:F$, $\mathcal{F}_\rho(v_R)=:Y$ and $\mathcal{F}_\rho(v_L)=:X$ is again positive quadruple of flags which allows us to compute the following \emph{shear invariants} of the closed leaf $c_\ell$:
\begin{equation}\label{ClosedLeafInvariant}
\sigma_i^\rho(c_\ell) := \log D_i (E,F,X,Y) 
\end{equation}
for $i=1,2,\cdots,n-1$.

These three types of parameters defined above are not free and must satisfy certain relations so called closed leaf equalities, closed leaf inequalities and rotation conditions. Bonahon and Dreyer argue that they the only relations among the parameters defined above. 

\begin{theorem}[Bonahon-Dreyer \cite{bonahon2014}]\label{BonahonDreyerCoordinates}
The rule $\mathcal{B}$ that assigns to each $[\rho]\in \Hit_n(\Sigma)$ the invariants $\tau^\rho _{i,j,k}(T_\ell, v_q)$, $\sigma^\rho _i(\gamma_\ell)$, and $\sigma^\rho _i (c_\ell)$ is an analytic injective map from $\Hit_n (\Sigma)$ onto the interior of a convex polytope defined by a collection of linear equations and inequalities.
\end{theorem}

\begin{remark} Suppose that $\Sigma$ is a hyperbolic surface with boundary. We apply the above technique to the Hitchin double $\widehat{\rho}$ of $\rho\in \Hit_n(\Sigma)$. This allows us to construct the Bonahon-Dreyer coordinates for the Hitchin component $\Hit_n(\Sigma)$ of the compact hyperbolic surface $\Sigma$. 
\end{remark}
\bibliographystyle{amsplain}

\bibliography{references_symp}

\end{document}